\documentclass[12pt]{amsart}
\usepackage{amsmath}
\usepackage{amssymb}

\usepackage[all]{xy}
\usepackage{longtable}
\usepackage{calrsfs}
\usepackage{color}
\usepackage{multirow,bigdelim}
\usepackage[normalem]{ulem}
\usepackage{graphicx}
\usepackage{hyperref}
\usepackage{ulem}
\usepackage{eucal}
\let\mathscr=\CMcal

\newtheorem{theorem}[equation]{Theorem}
\newtheorem{lemma}[equation]{Lemma}
\newtheorem{proposition}[equation]{Proposition}
\newtheorem{corollary}[equation]{Corollary}

\newtheorem{definition}[equation]{Definition}

\theoremstyle{remark}
\newtheorem{remark}[equation]{Remark}

\numberwithin{equation}{subsection}

\allowdisplaybreaks[1]

\newcommand{\FF}{\mathbb{F}}
\newcommand{\ZZ}{\mathbb{Z}}
\newcommand{\QQ}{\mathbb{Q}}

\newcommand{\GG}{\mathbb{G}}

\newcommand{\CC}{\mathbb{C}}
\newcommand{\PP}{\mathbb{P}}

\newcommand{\MM}{\mathbb{M}}
\newcommand{\NN}{\mathbb{N}}

\renewcommand{\AA}{\mathbb{A}}

\newcommand{\cA}{\mathcal{A}}

\newcommand{\cM}{\mathcal{M}}

\newcommand{\cT}{\mathcal{T}}
\newcommand{\cH}{\mathcal{H}}
\newcommand{\cJ}{\mathcal{J}}

\newcommand{\cC}{\mathcal{C}}
\newcommand{\cR}{\mathcal{R}}

\newcommand{\Acal}{\CMcal{A}}

\newcommand{\Rcal}{\CMcal{R}}

\newcommand{\Ucal}{\CMcal{U}}

\DeclareMathAlphabet{\matheur}{U}{eur}{m}{n}
\newcommand{\eC}{\matheur{C}}
\newcommand{\eF}{\matheur{F}}
\newcommand{\eK}{\matheur{K}}

\newcommand{\eM}{\matheur{M}}
\newcommand{\eR}{\matheur{R}}

\newcommand{\itM}{\mathit{M}}

\newcommand{\fm}{\mathfrak{m}}

\newcommand{\Rscr}{\mathscr{R}}

\newcommand{\bomega}{\boldsymbol{\omega}}

\DeclareMathOperator{\Aut}{Aut} 
 \DeclareMathOperator{\GL}{GL}
\DeclareMathOperator{\Mat}{Mat} \DeclareMathOperator{\Cent}{Cent}
\DeclareMathOperator{\End}{End} 
 
\DeclareMathOperator{\Hom}{Hom}

\DeclareMathOperator{\trdeg}{tr.deg}

\DeclareMathOperator{\rank}{rank}

\DeclareMathOperator{\ari}{ari}

\newcommand{\ok}{\overline{k}}

\newcommand{\ag}{\Gamma_{\mathrm{ari}}}
\newcommand{\agv}{\Gamma_{\mathrm{ari},v}}
\newcommand{\agi}{\Gamma_{\mathrm{ari},\infty}}

\newcommand\redsout{\bgroup\markoverwith{\textcolor{red}{\rule[0.5ex]{2pt}{0.4pt}}}\ULon}
\newcommand\mapsfrom{\mathrel{\reflectbox{\ensuremath{\mapsto}}}}
\newcommand\subfrac[2]{\genfrac{}{}{0pt}{}{#1}{#2}}

\newenvironment{Subsubsec}[1]{\vspace{\topsep}
\noindent\refstepcounter{equation}\theequation.\ \textit{#1.}}{\vspace{\topsep}}

\DeclareMathVersion{normal2}

\begin{document}

\title[$v$-adic periods of Carlitz motives]{$v$-adic periods of Carlitz motives and Chowla--Selberg formula revisited}

\author{Chieh-Yu Chang}
\address{Department of Mathematics, National Tsing Hua University, Hsinchu City 30042, Taiwan
  R.O.C.}

\email{cychang@math.nthu.edu.tw}

\author{Fu-Tsun Wei }
\address{Department of Mathematics, National Tsing Hua University, Hsinchu City 30042, Taiwan
  R.O.C.}
\email{ftwei@math.nthu.edu.tw}

%    Information for third author
\author{Jing Yu}
\address{Department of Mathematics, National Taiwan University, Taipei City 10617, Taiwan
  R.O.C.}
\email{yu@math.ntu.edu.tw}

\thanks{}

%    General info

\date{\today}

\begin{abstract} Let $v$ be a finite place of $\FF_q(\theta)$. In this paper, we interpret $v$-adic arithmetic gamma values in terms of the $v$-adic crystalline--de Rham periods of Carlitz motives with Complex Multiplication, and establish an Ogus-type Chowla--Selberg formula. 
Furthermore, we prove the algebraic independence of these $v$-adic periods by employing the technique of ``switching $v$ and $\infty$'', and determining the dimension of relevant motivic Galois groups on the ``$\infty$-adic'' side 
through an adaptation and refinement of existing methods.  
As a consequence, all algebraic relations among $v$-adic arithmetic gamma values over $\FF_q(\theta)$ can be derived from standard functional equations together with Thakur’s analogue of the Gross--Koblitz formula.
\end{abstract}


\maketitle
\section{Introduction}

\subsection{Motivation}
This study is inspired by the Gross--Koblitz formula~\cite{GK79} and its function field analogue established by Thakur~\cite{T88}. 
Classically, one interpolates the factorials $n!$ to obtain the Euler gamma function, and a folklore conjecture asserts that all
the special gamma values $\Gamma (z)$, for $z\in \QQ\setminus \ZZ$, are transcendental.
A key reason for this conjecture comes from the geometric interpretation of these values as ``abelian CM periods,'' as illuminated by the Chowla–Selberg formula and further explored in Gross’s foundational work~\cite{Gr78}. This geometric viewpoint has greatly deepened our understanding of the significance of gamma values (see \cite{Anderson82}, \cite{Co93}, \cite{MR04}, \cite{Yang10}, \cite{BM16}, and \cite{Fr17}, etc.).

The trancendence of the gamma values at rationals is only known when the denominator is $2$ as $\Gamma(\frac{1}{2})=\sqrt{\pi}$, or when the denominator is $4$ or $6$. The latter case is due to Chudnovsky's work~\cite[p.~6-8]{Chud84} on the algebraic independence of a nonzero period $\varpi$ and its corresponding quasi-period $\eta(\varpi)$ of a CM elliptic curve defined over $\overline{\mathbb{Q}}$ (see also~\cite[p.~30]{BW07}).
%Here $\eta$ is the {\it quasi-period map}, see \cite[Proposition 5.2 (b) and Remark 5.3]{Sil94}}.
%{\color{blue}(how about ``algebraic independence of $\pi$ and a nonzero period of a CM elliptic curve defined over $\overline{\mathbb{Q}}$''?)}. 
The Lang-Rohrlich conjecture~\cite{Lang90} asserts that all $\overline{\QQ}$-algebraic relations among special gamma values and $2\pi \sqrt{-1}$ are explained by the standard functional equations satisfied by the Euler gamma function, which can be derived if the Shimura conjecture on period symbols is true (see \cite{BCPW22}).

For a prime number $p$, Morita~\cite{Mo75} defined a continuous function  $\Gamma_p:\mathbb{Z}_p\rightarrow \ZZ_{p}^{\times}$ via a $p$-adic interpolation of the factorials $n!$ (up to a sign). Moreover,
Ogus' $p$-adic Chowla--Selberg-type formula in \cite{O89} interprets 
the special $p$-adic gamma values  $\Gamma_{p}(z)$ for $z$ in $\QQ\cap \ZZ_p$ in terms of matrix coefficients of the isocrystal (Frobenius) action on the crystalline cohomology of the Jacobians of Fermat curves.
%Since the very unique $\pi$ is missing in the $p$-adic domains, $\Gamma_{p}(1/2)$ for odd prime $p$, becomes algebraic.
Note that for an odd prime $p$, Gross--Koblitz managed to express a Gauss sum explicitly in terms of a product of specific special Morita $p$-adic gamma values, which provides ``extra'' algebraic relations among these values other than those coming from standard functional equations.
%. As a consequence, they derived the algebraicity of the special $p$-adic gamma value $\Gamma_{p}(z)$ whenever the denominator of $z\in \QQ\cap \ZZ_p$ is a factor of $p-1$. 
%Gross-Koblitz's algebraicity result clearly reveals a deviation of the $p$-adic interpolations from the classical $\infty$-adic interpolation.  
It is a natural question %whether for any odd prime $p$, all the special Morita gamma values are transcendental except those  explained in Gross-Koblitz's paper. 
%One may ask further 
whether there is a \lq $p$-adic Lang--Rohrlich phenomenon\rq\ (namely, the known existing algebraic relations arising from the standard functional equations that Morita's $p$-adic gamma function satisfies together with the Gross-Koblitz formula) accounting for all $\overline{\QQ}$-algebraic relations among the special $p$-adic gamma values.
Our ultimate goal is to explore the ``$v$-adic'' Chowla--Selberg phenomenon through analyzing ``$v$-adic periods'', and to provide an affirmative answer for this question in the positive characteristic world.

%To explain the algebraic relations among these special values, a natural approach is to seek suitable interpretations in terms of periods. The classical Chowla-Selberg formula, which connects a nonzero period of a CM elliptic curve over $\overline{\QQ}$ to a specific product of special gamma values, leads to a new perspective for analyzing “abelian CM periods” from a geometric standpoint in the pioneer work by Gross \cite{Gross}. This perspective has significantly advanced our understanding of the geometric meanings associated with gamma values \cite{}...\cite{}.
%In the $p$-adic setting, Ogus \cite{O89} successfully expresses the special values of Morita’s $p$-adic gamma function as specific matrix coefficients of the isocrystal (Frobenius) action on the $p$-adic crystalline cohomology of the Jacobians of Fermat curves. This provides a “motivic description” of the algebraic relations among these $p$-adic values.

The present paper represents our first successful attempt (in this direction) to
investigate the $v$-adic arithmetic gamma function for the rational function field $k$ of one variable over a finite field, after L.~Carlitz~\cite{Car35}, D.~Goss~\cite{Go80} and D.~Thakur  \cite{T88}, \cite{T91}, and examine the period interpretation of its special values.
Note that for the $\infty$-adic gamma values over function fields, the Chowla--Selberg phenomenon has been fully clarified from the work of \cite{Sinha}, \cite{BP02}, \cite{ABP04}, \cite{CPTY10}, and \cite{Wei22}. 
In this paper, we shift our focus to the $v$-adic case 
for any finite place $v$ of $k$, 
and establish an Ogus-type Chowla--Selberg formula for the $v$-adic arithmetic gamma values connecting with the $v$-adic ``cystalline--de Rham periods'' of ``Carlitz motives with complex multiplication''.
We also derive the algebraic independence of these $v$-adic periods, which enables us to show that the algebraicity of special arithmetic $v$-adic gamma values comes only from Thakur's analogue of  Gross--Koblitz. 

%Another motivation of our study on special $v$-adic arithmetic gamma values comes from the classical Chowla-Selberg formula, which relates a nonzero period of a CM elliptic curve over $\overline{\QQ}$ to a specific product of special gamma values. Inspired by the $p$-adic Chowla-Selberg formula established by Ogus~\cite{O89}, we prove in Theorem~\ref{thm: main-2} that certain crystalline--de Rham periods of Carlitz motives with complex multiplication can be expressed as products of special $v$-adic arithmetic gamma values. It turns out that such a period interpretation plays a crucial role in the process of proving algebraic independence result of Theorem~\ref{thm: main-1}.

\subsection{\texorpdfstring{$v$}{v}-adic periods of Carlitz motives and gamma values}

Let $A:=\FF_q[\theta]$ be the polynomial ring in one variable $\theta$ over a finite field $\FF_q$ of characteristic $p$ with $q$ elements, and $k:=\FF_q(\theta)$ be its fraction field.
Put $D_{0}:=1$, and for $i\geq 1$, $D_{i}:=\prod_{j=0}^{i-1}(\theta^{q^{i}}-\theta^{q^{j}})\in A$. For each $i\in \NN$,  $D_i$ equals the product of all monic polynomials of degree $i$ in $A$. Carlitz  \cite{Car35} introduced his factorial for $A$ as follows:
given a non-negative integer $n$, write $n=\sum_{i\geq 0} n_{i}q^{i}$ for $0\leq n_{i}<q$ and put 
\[ \agi (n+1):=\prod_{i\geq 0}D_{i}^{n_i} \in A,\] which becomes the positive characteristic analogue of the factorial $n!$, see~\eqref{E:valCarlitzFac}.
Taking the unit part of these factorial values, Goss \cite{Go80} discovered that one may do interpolation in Morita style, 
arriving at a continuous function: 
$\agi : 
\ZZ_p\rightarrow \FF_q(\!(1/\theta)\!)^\times$.

Goss successfully applied his strategy to all places of $k$.
Given a monic irreducible polynomial $v \in A$, let $k_v$ be the completion of $k$ at $v$.
%Let $\bar{k}_v$ be a fixed algebraic closure of $k_v$, and $\bar{k}$ be the algebraic closure of $k$ inside $\bar{k}_v$.
Using $v$-adic interpolation for the Carlitz factorials, Goss~\cite{Go88} and Thakur~\cite{T91} defined a $v$-adic arithmetic gamma function \[
\agv:\ZZ_p\rightarrow k_v^\times,\]
whose precise definition is given in Section~\ref{sec: v-adic agv}.

Note that in the $\infty$-adic case, the period interpretation of $\agi(z)$ for $z \in \QQ\cap \ZZ_p$ is illustrated in \cite{CPTY10}, which leads to the determination of all algebraic relations among these values.
More precisely, let $\FF_q[t]$ be the polynomial ring in another variable $t$ over $\FF_q$.
Given a positive integer $\ell$, we let ${\bf{C}}_{\ell}$ be the Carlitz $\FF_{q^{\ell}}[t]$-module, then ${\bf{C}}_{\ell}$ is a CM Drinfeld $\FF_{q}[t]$-module of rank $\ell$ with CM field $\FF_{q^{\ell}}(t)$.  In~\cite[Theorems~3.3.1 and 3.3.2]{CPTY10}, it is shown that periods of the dual $t$-motive associated to ${\bf{C}}_{\ell}$ are expressed as products of special $\infty$-adic gamma values (up to algebraic multiples). By~\cite[Sec.~4.2]{Pe08} (see also~\cite[(3.4.5) and (3.4.8)]{CP12}),  periods and quasi-periods of ${\bf{C}}_{\ell}$ are explicitly related to the periods of the dual $t$-motive in question, and hence combining these gives a Chowla-Selberg formula for special $\infty$-adic arithmetic gamma values.

Here we shall present a $v$-adic counterpart of the phenomenon above. 
Let $\eC_{\ell}$ be the Carlitz $t$-motive of degree $\ell$ over $A$, and denote by $\eR_\tau(\eC_{\ell})$
the `restriction of scalars' of $\eC_{\ell}$ (see Section~\ref{sec: Car-t}). We note that the terminology \lq degree $\ell$\rq\ indicates that $\eC_{\ell}$ has complex multiplication by the degree $\ell$ constant field extension over $\FF_{q}(t)$.
Let $v$ be a monic irreducible polynomial in $A$ and $k_v$ be the completion of $k$ at $v$.
As $\eR_\tau(\eC_{\ell})$ is defined `integrally' over $A$ and has `good reduction' at $v$, one may associate its crystalline module $H_{\textrm{cris}}(\eR_\tau(\eC_{\ell}),k_v)$ and de~Rham module $H_{\textrm{dR}}(\eR_\tau(\eC_{\ell}),k_v)$ over $k_v$ (see Remark~\ref{rem: cris-dR-kv}). 
Moreover, the theory of Hartl--Kim in \cite{HK20} provides a crystalline--de~Rham comparison isomorphism (see \eqref{eqn: period-iso})
\[
\phi_v: H_{\textrm{cris}}(\eR_\tau(\eC_{\ell}),k_v) \stackrel{\sim}{\longrightarrow} H_{\textrm{dR}}(\eR_\tau(\eC_{\ell}),k_v).
\]
Therefore the natural ``isocrystal action'' $\tau_v$ on $H_{\textrm{cris}}(\eR_\tau(\eC_{\ell}),k_v)$ transforms to a $k_v$-linear automorphism $\varrho_v$ on $H_{\textrm{dR}}(\eR_\tau(\eC_{\ell}),k_v)$ through $\phi_v$ (see \eqref{eqn: rho-v} for the precise definition).
The first theorem of this paper is to express ${\varrho}_v$ explicitly in terms of special $v$-adic arithmetic gamma values:

\begin{theorem}\label{thm: main-2}
\text{\rm (See Theorem~\ref{thm: CSF-1}.)} 
Choose the `global' basis $\{\omega_0,...,\omega_{\ell-1}\}$ of the de~Rham module $H_{\textrm{dR}}(\eR_\tau(\eC_{\ell}),k)$ as in \eqref{eqn: global-omega}.
Let $d:=\deg v$.
For each integer $s$ with $0\leq s<\ell$, we write $s+d = s_0 + n_s'\ell$, where $s_0 \in \ZZ$ with $0\leq s_0<\ell$ and $n_s' \in \ZZ_{\geq 0}$.
Then we have
\[
\varrho_{v}\cdot \omega_s =  (-1)^{n_s'} \cdot C_{s}\cdot
\displaystyle 
\frac{\displaystyle\agv\Big(1-\Big\langle\frac{q^{s+d}}{q^{\ell}-1}\Big\rangle_{\ari}\Big)}{\displaystyle\agv\Big(1-\Big\langle\frac{q^{s+d-1}}{q^{\ell}-1}\Big\rangle_{\ari}\Big)^q}\cdot \omega_{s_0},
\quad \text{ where } \quad 
C_s:= 
\begin{cases}
-v , & \text{ if $s=0$,}\\
1, & \text{ if $s>0$,}
\end{cases}
\]
and $\langle z \rangle_{\ari}$ is the fractional part of $z$ for every $z\in \QQ$, i.e.\ $0\leq \langle z \rangle_{\ari}<1$ and $z-\langle z \rangle_{\ari} \in \ZZ$.
\end{theorem}

\begin{remark}
We may regard Theorem~\ref{thm: main-2} as a function field analogue of Ogus' $p$-adic Chowla--Selberg formula in \cite[Theorem 3.15]{O89}.    
\end{remark}

The proof of Theorem~\ref{thm: main-2} is based on an explicit description of the crystalline--de~Rham comparison isomorphism $\phi_v$ and a product expansion formula for the special values of $\agv$ in question (see~Corollary~\ref{cor: agv-form}).
By choosing
a specific `global' $A[t]$-base $\mathfrak{B}$ of $\eR_\tau(\eC_\ell)$ as in \eqref{eqn: basis-B}, which induces a basis $\Lambda_v = \{\lambda_{v,0},...,\lambda_{v,\ell-1}\}$ of $H_{\textrm{cris}}\big(\eR_\tau(\eC_\ell),k_v\big)$ (in \eqref{eqn: lambda-v}) and a basis $\beta_v=\{\omega_{v,0},...,\omega_{v,\ell-1}\}$ of $H_{\textrm{dR}}\big(\eR_\tau(\eC_\ell),k_v\big)$ (in \eqref{eqn: beta-v}), we particularly have
 
\[
\omega_{v,s} = \begin{cases}    
\left(
\displaystyle\prod_{n=1}^{n_s-1}(\theta-\theta^{q^{\ell n + s}})
\right)
\cdot \omega_{s}, & \text{ if $s=0$,} \\ 
\left(
\displaystyle\prod_{n=0}^{n_s-1}(\theta-\theta^{q^{\ell n + s}})
\right)
\cdot \omega_{s}, & \text{ if $s>0$,}
\end{cases}
\] 
given in~\eqref{eqn: dR-basis-2}
(where $n_s$ is the smallest non-negative integer so that $s+n_s\ell \geq \deg v$),
and the ``period matrix'' of the crystalline--de Rham isomorphism $\phi_v$ with respect to the bases $\Lambda_v$ and $\beta_v$ is determined by the following result:

\begin{theorem}\label{thm: pm}
{\rm (See Theorem~\ref{thm: period-matrix}.)}
For every integer $s$ with $0\leq s<\ell$, we have
\[
\phi_v(\lambda_{v,s}) = \prod_{n=n_s}^\infty (1-\frac{(\theta-\varepsilon)^{q^{\ell n + s}}}{\theta-\varepsilon^{q^{\ell n + s}}}) \cdot \omega_{v,s}.
\]    
Here $\varepsilon = \lim_{i\rightarrow \infty}\theta^{q^{i\deg v}} $ $\in k_v$ is the Teichm\"uller image of $\theta \bmod v\in A/(v)$, and $n_s$ is the smallest non-negative integer so that $s+n_s\ell \geq \deg v$.
Consequently,
the period matrix of the isomorphism $\phi_v$ in \eqref{eqn: period-iso} with respect to $\beta_v$ and $\Lambda_v$ is 
\[
\begin{pmatrix}
\displaystyle\prod_{n=n_0}^{\infty}(1- \frac{(\theta-\varepsilon)^{q^{\ell n }}}{\theta-\varepsilon^{q^{\ell n }}}) & & 0\\
& \ddots & \\
0& & \displaystyle\prod_{n=n_{\ell-1}}^{\infty}(1- \frac{(\theta-\varepsilon)^{q^{\ell n + \ell-1}}}{\theta-\varepsilon^{q^{\ell n + \ell-1}}})
\end{pmatrix}.
\]
\end{theorem}

Having this period matrix at hand, we are able to realize the action of $\varrho_v$ on the basis $\beta_v$ as in \eqref{eqn: Fro-act}.
Combining the product formula of $v$-adic arithmetic gamma values in Corollary~\ref{cor: agv-form} enables us to derive Theorem~\ref{thm: main-2}.

\subsection{Algebraic independence of the \texorpdfstring{$v$}{v}-adic crystalline--de~Rham periods}

%From the $v$-adic Chowla--Selberg formula in Theorem~\ref{thm: main-2} and the concrete description of the period matrix in Theorem~\ref{thm: pm}, our proof of Theorem~\ref{thm: main-1} relies on the algebraic independence of the following $v$-adic periods:

Let $\bar{k}_v$ be an algebraic closure of $k_v$, and let $\bar{k}$ be the algebraic closure of $k$ in $\bar{k}_v$.
By Theorem~\ref{thm: pm}, the field generated by the entries of the period matrix of the crystalline--de Rham isomorphism $\phi_v$ over $\bar{k}$ is
\[
\bar{k}\Big(
\Omega_{\ell,v,s}(v)\ \Big|\ 1\leq  s \leq \ell\Big),
\quad \text{ where }\quad 
\Omega_{\ell,v,s}(v):=\prod_{n=0}^\infty(1-\frac{(\theta-\varepsilon)^{q^{\ell n + s}}}{\theta-\varepsilon^{q^{\ell n + s}}}) \quad \in k_v, \quad 1\leq s \leq \ell.
\]
Rather than directly using the entries of the period matrix in Theorem~\ref{thm: pm}, we take this opportunity to introduce the infinite products  $\Omega_{\ell,v,s}(v)$, whose period-theoretic interpretation will play a crucial role in Section~\ref{sec: str-main3}.
The second theorem of this paper is to determine their algebraic independence:

\begin{theorem}\label{thm: main-3}
Given a positive integer $\ell$,
the periods
$\Omega_{\ell,v,1}(v),\cdots, \Omega_{\ell,v,\ell}(v)$
are algebraically independent over $\bar{k}$.
\end{theorem}

%\begin{remark}
%{\color{blue}(Maybe put this remark after the proof of Thm 1.3.1?)} Let $\Gamma_{\eR_{\tau}(\eC_{\ell})}$ be the ``$t$-motivic Galois group of the $t$-motive $\eR_{\tau}(\eC_{\ell})$'' (see \cite[Lemma~3.2.1]{CPTY10}).
%{\color{blue} (Should we point out that $\Gamma_{\eR_{\tau}(\eC_{\ell})}$ and $\Gamma_{\Rcal_\sigma(\tilde{\cC}_{d\ell})}$ are different?)}
%Theorem~\ref{thm: main-3} implies that
%\[
%\trdeg_{\bar{k}}\bar{k}\Big(
%\Omega_{\ell,v,s}(v)\ \Big|\ 1\leq s \leq \ell\Big) = \ell = \dim \Gamma_{\eR_{\tau}(\eC_{\ell})},
%\]
%This coincidence appears to provide evidence for a suitable analogue of Papanikolas' theorem in \cite{P08} (Grothendieck period conjecture) associated with the $v$-adic periods that arise from crystalline--de~Rham comparison isomorphisms. 
%
%{\color{red}Perhaps we do not necessarily make this remark? The confusion may arise from the different roles of $v$-adic and $\infty$-adic. }
%\end{remark}

\begin{Subsubsec}{Strategy of deriving Theorem~\ref{thm: main-3}}\label{sec: str-main3}
In order to apply the powerful machinery of Papanikolas' theory on Grothendieck's period conjecture in \cite{P08},
our initial step is to regard $v$ as the ``infinite place'' of the projective $\frac{1}{v}$-line over $\FF_q$.
More precisely, put $\tilde{v}:= \frac{1}{v}$.
Then $k$ is a finite extension of $\FF_q(\tilde{v})$, and
$k_v$ is a finite extension of the Laurent series field $\FF_q(\!(\frac{1}{\tilde{v}})\!)$ (which is the completion of $\FF_q(\tilde{v})$ at the place $\infty_{\tilde{v}}$ corresponding to $\frac{1}{\tilde{v}}$). So $k_{v}$ is an \lq$\infty_{\tilde{v}}$-adic\rq\ field.
Let $\iota: \FF_q[t]\cong \FF_q[\theta]$ be the $\FF_q$-algebra isomorphism sending $t$ to $\theta$, which extends to a field isomorphism $\iota: \FF_q(t) \cong \FF_q(\theta)$.
Put $\tilde{T}_v:=\iota^{-1}(\tilde{v}) \in \FF_q(t)$. 
Then $\iota$ induces an $\FF_q$-algebra isomorphism $\FF_q[\tilde{T}_v]\cong \FF_q[\tilde{v}]$. To apply Papanikolas's theory~\cite{P08} for those special values in our $\infty_{\tilde{v}}$-adic fields, we make the following replacements \[\theta \leftarrow \tilde{v},\quad{ }t\leftarrow \tilde{T}_v \]
to switch $\ok[t]$-modules (resp.~dual $t$-motives)  to 
$\ok[\tilde{T}_{v}]$-modules (resp.~dual $\tilde{T}_v$-motives). A detailed description of how to view $v$-adic fields as $\infty_{\tilde{v}}$-adic fields is given in Sec.~\ref{sec: v to infinite}.

Let $d:=\deg v$.  Our proof of Theorem~\ref{thm: main-3}
is outlined in the following:
\begin{itemize}
    \item[(1)] Define a ``Carlitz-type'' Hartl--Juschka module $\Rcal_\sigma(\tilde{\cC}_{d\ell})$ (see Definition~\ref{def: RC}), where the infinite products $\Omega_{d\ell,v,s}(v)$ in Theorem~\ref{thm: main-3} (replacing $\ell$ by $d\ell$), for $0<s\leq d\ell$, show up as periods of $\Rcal_\sigma(\tilde{\cC}_{d\ell})$ (see Proposition~\ref{prop: Omega-period}).
    \item[(2)] By Papanikolas' theorem, determining the transcendence degree of the field generated by the periods of $\Rcal_\sigma(\tilde{\cC}_{d\ell})$ is equivalent to finding the dimension of its $\tilde{T}_v$-motivic Galois group $\Gamma_{\Rcal_\sigma(\tilde{\cC}_{d\ell})}$ (see \eqref{eqn: trdeg-2}).
    \item[(3)] The ``CM structure'' of $\Rcal_\sigma(\tilde{\cC}_{d\ell})$ provides an upper bound of $\dim \Gamma_{\Rcal_\sigma(\tilde{\cC}_{d\ell})}$ (cf.~\eqref{eqn: CI}).
    On the other hand, from the analysis of the Hodge--Pink weights of $\Rcal_\sigma(\tilde{\cC}_{d\ell})$, we obtain sufficiently many ``Hodge--Pink cocharacters'' of $\Rcal_\sigma(\tilde{\cC}_{d\ell})$, from which we get 
    \[
    \dim \Gamma_{\Rcal_\sigma(\tilde{\cC}_{d\ell})} = d^2\ell \quad (=[\FF_{q^{d\ell}}(t):\FF_q(\tilde{T}_v)]),
    \] see Theorem~\ref{thm: dim-G}.
    \item[(4)] Therefore, the period interpretation of $\Omega_{d\ell,v,s}(v)$ in (1) above for $0<s\leq d\ell$ shows their algebraic independence over $k$ (see Section~\ref{sec: pf-AL-Omega}).
    Finally, the result follows from the identity (see \eqref{eqn: omega-dltol}):
    \[
    \Omega_{\ell,v,s}(v) = \prod_{j=0}^{d-1}\Omega_{d\ell,v,j\ell+s}(v), \quad 0<s\leq \ell.
    \]
\end{itemize}
\end{Subsubsec}

\begin{remark}
Put $\tilde{\theta}_v:= (\theta-\varepsilon)^{-1}$, and set
\[
\tilde{\Omega}_{\ell}(x):= \prod_{n=0}^{\infty}(1-\frac{x}{\tilde{\theta}_v^{q^{\ell n}}}) \quad \in k_v[\![x]\!], \quad \forall \ell \in \NN.
\]
When $\deg v = 1$, we get $v = \theta-\varepsilon$, for which $\varepsilon \in \FF_q$.
Therefore
\[
\Omega_{\ell,v,s}(v) = \tilde{\Omega}_{\ell}^{(s)}(x)\big|_{x=\tilde{\theta}_v} = 
\prod_{n=0}^{\infty}(1-\frac{\tilde{\theta}_v}{\tilde{\theta}_v^{q^{\ell n+s}}}),
\]
and Theorem~\ref{thm: main-3} can be simply derived from the algebraic independence of periods and the quasi-periods of the Carlitz module of degree $\ell$ in \cite{CPTY10} (when replacing $\theta$ to $\tilde{\theta}_v$).
However, when $\deg v=d>1$,
the infinite product $\Omega_{d\ell,v,s}(v)$, for $0<s\leq d\ell$, is actually the $s$th Frobenius twist of $\tilde{\Omega}_{d\ell}(x)$ evaluated at $x=\tilde{\theta}_{v,s}:=(\theta-\varepsilon^{q^s})^{-1}$:
\[
\Omega_{d\ell,v,s}(v) = \tilde{\Omega}_{d\ell}^{(s)}(x)\big|_{x=\tilde{\theta}_{v,s}} = \prod_{n=0}^{\infty}(1-\frac{\tilde{\theta}_{v,s}}{\tilde{\theta}_v^{q^{d\ell n+s}}}).
\]
Thus, to show the algebraic independence of $\Omega_{d\ell,v,s}(v)$ is equivalent to determine the algebraic relations among the evaluations of the Frobenius twists of $\tilde{\Omega}_{d\ell}(x)$ at $d$ distinct points $x=\tilde{\theta}_{v,1},..., \tilde{\theta}_{v,d}$.
In order to overcome this issue,
we consider the ``restriction of scalars'' down to $\FF_q[\tilde{T}_v]$, so that the periods of the Hartl--Juschka $\bar{k}[\tilde{T}_v]$-module $\Rcal_{\sigma}(\tilde{\cC}_{d\ell})$ do include all of these evaluations.
\end{remark}

Together with Theorem~\ref{thm: main-3}, the $v$-adic Chowla--Selberg formula in Theorem~\ref{thm: main-2} enables us to construct sufficiently many algebraically independent elements in the field 
\[
\bar{k}\big(\agv(z)\ \big|\ z \in \QQ \text{ with } (q^\ell-1)z \in \ZZ\big).
\]
As a consequence, we are able to illustrate the Lang--Rohrlich phenomenon for $v$-adic arithmetic gamma values in the next subsection.
%for which one ensures that its transcendence degree over $\bar{k}$ is bounded below by $\ell-\gcd(\ell,\deg v)$ (see Theorem~\ref{thm: AL-gamma}), and hence proves Theorem~\ref{thm: main-1}.\\

\subsection{Algebraic relations among  \texorpdfstring{$v$}{v}-adic arithmetic gamma values}

In the $\infty$-adic case, the special values $\agi (z)$, for $z\in \QQ\cap \ZZ_p\setminus \ZZ$, are all transcendental over $k$ by Thakur~\cite{T96} and Allouche~\cite{All96}. Furthermore, using transcendence theory in terms of $t$-motives developed by Anderson-Brownawell-Papanikolas~\cite{ABP04} and Papanikolas~\cite{P08}, the period interpretation of these special values implies that all algebraic relations among these special values are shown to follow from the standard monomial relations ~\cite{CPTY10}. These relations come directly from the functional equations satisfied by $\agi$. Note that $\QQ \cap \ZZ_p = \bigcup_{\ell \in \NN} \frac{1}{q^\ell-1}\cdot \ZZ$.
In particular, one has:
 \[
\trdeg_k k\big(\agi(z)\ \big|\ z \in \QQ \text{ with } (q^\ell-1)z \in \ZZ\big) = \ell.
\]

In the $v$-adic case, the functional equations of $\Gamma_{\text{ari},v}$ derived in \cite{Go88} and \cite{T91} also give rise to standard monomial relations \eqref{E:Standard FE} among $v$-adic arithmetic gamma values at rational $p$-adic integers. However for the finite place $v$,  one naturally encounters the Gross--Koblitz phenomenon, as shown by Thakur \cite{T88}. 
This gives additional relations among the special values in question (see Theorem~\ref{thm: GK-formula} and Corollary~\ref{cor: GK-formula}).
From the period interpretation of these $v$-adic values in Theorem~\ref{thm: main-2} and the algebraic independence of the $v$-adic periods in Theorem~\ref{thm: main-3}, we thereby obtain: 

\begin{theorem}\label{thm: main-1}
(See Theorem~\ref{thm: AL-gamma}) Given a positive integer $\ell$, we have that
\[
\trdeg_{\bar{k}} \bar{k}\big(\agv(z)\ \big|\ z \in \QQ \text{ with } (q^\ell-1)z \in \ZZ\big) = \ell - \gcd(\ell,\deg v).
\]
Consequently, all the $\bar{k}$-algebraic relations among $v$-adic arithmetic gamma values at rational $p$-adic integers are explained by the functional equations in Proposition~\ref{prop: FE}, the standard monomial relations \eqref{E:Standard FE} and  Thakur's analogue of Gross--Koblitz formula.
\end{theorem}

The decrease of transcendence degree for $v$-adic special values resulting from the analogue of Gross-Koblitz formula is precisely given by $\gcd(\ell,\deg v)$. 
So the naturally existing monomial relations indicate that the transcendence degree specified in Theorem~\ref{thm: main-1} is bounded above by $\ell - \gcd(\ell,\deg v)$, see Theorem~\ref{thm: up-bd}. 
The opposite inequality is assured by the algebraic independence of the $v$-adic crystalline--de Rham periods in Theorem~\ref{thm: pm}).
Also, the above theorem complements Thakur's theorem (cf.~Remark~\ref{Rem: Equiv}):

\begin{corollary}
{\rm (See Corollary~\ref{cor: tran-gamma}.)} Given $a,b \in \ZZ$ with $\gcd(a,b)=1$ and $p \nmid b$, we have that
\[
\agv(\frac{a}{b}) \text{ is algebraic over $k$}\quad 
\text{if and only if}
\quad b \text{ divides } q^{\deg v} - 1.
\]
\end{corollary}

\begin{remark}
In \cite[\S 5.9]{T91}, Thakur also introduced special $v$-adic \textit{geometric} gamma values $\Gamma_{\text{geo},v}(a/b) \in k_v$ for $a/b \in k$ with $\text{gcd}(a,b)=1$ and $v \nmid b$.
We conjecture that $\Gamma_{\text{geo},v}(a/b)$ is algebraic over $k$ if and only if $b$ divides $v-1$.
Note that the condition ``$b \mid v-1$'' is related to the fact that the Carlitz $b$-torsions are contained in the residue field $\FF_v$ as well as the $v$-adic $k_v$, because of an analogue of the little theorem of Fermat for the Carlitz module.
\end{remark}

\begin{remark}
In the case of $\deg v = 1$,
Thakur \cite{T98} initially derived the transcendence of $\agv(\frac{a}{b})$ in the aforementioned corollary using methods from automata theory for special cases.  Thakur's approach was later generalized to encompass general $\frac{a}{b}$ in \cite{WY02}.
Also, we point out that a $v$-adic Anderson--Brownawell--Papanikolas criterion in \cite{Chen23} would lead to an alternative proof of the above corollary when $\deg v =1$.
\end{remark}

%expression of the $v$-adic arithmetic gamma values in terms of the $v$-adic periods of the ``crystalline--de~Rham comparison isomorphism'' for Carlitz $t$-motives with complex multiplication by constant field extensions.

\begin{remark}
In the proof of Theorem~\ref{thm: main-1}, we reduce to the special case when $\deg v$ divides $\ell$. After completing our paper,
we were informed that David Adam also derived this special case in \cite{Ad23}  using Mahler's method. Adam applies an analogue of Nishioka's theorem to establish the algebraic independence of the $v$-adic values by interpreting them as special values of specific Mahler functions. In contrast, we provides a period interpretation of these $v$-adic values and proves Theorem~\ref{thm: main-1} immediately from the algebraic independence of the $v$-adic periods in Theorem~\ref{thm: main-3} (which is not covered by \cite{Ad23}).
In the proof of Theorem~\ref{thm: main-3}, we uses Papanikolas' theorem on Grothendieck's period conjecture in \cite{P08}, along with one key technique of ``switching $v$ and $\infty$'' explained in Section~\ref{sec: str-main3}.

Despite the coincidence of our transcendence results in Theorem~\ref{thm: main-1} with Adam’s work \cite{Ad23}, our approach offers a natural geometric perspective on the $v$-adic arithmetic gamma values in question. We believe that our method can also be extended to the study of $v$-adic geometric gamma values, including  ``two-variable ones'' as well.
Together with the framework of CM motives over function fields developed in \cite{BCPW22}, we expect that all $v$-adic ``abelian CM periods'' can be expressed in terms of suitable twisted products of $v$-adic gamma values, thereby revealing the complete picture of the Chowla–Selberg phenomenon in the $v$-adic settings. This will be explored in depth in a subsequent paper.
\end{remark}

\subsection{Organization of this paper}

In Section~\ref{sec: agv}, we start with fixing basic notation and recalling the Carlitz factorials in Section~\ref{sec: notation}.
The definition of Goss' $v$-adic arithmetic gamma function and its functional equations is given in Section~\ref{sec: v-adic agv}, as well as Thakur's analogue of Gross--Koblitz formula.
In Section~\ref{sec: prod agv}, we derive a  product expansion of $v$-adic arithmetic gamma values, a pivotal observation that underscores their connection with the    ``$v$-adic periods'' of Carlitz $t$-motives with complex multiplication by constant field extensions.

In Section~\ref{sec: Car-cry-dR}, we first introduce the ``integral'' Carlitz $t$-motive $\eC_\ell$ of arbitrary degree $\ell$, along with its restriction of scalars $\eR_\tau(\eC_\ell)$ in Section~\ref{sec: Car-t}.
After setting up the power series rings to be used and clarifying their relations in Section~\ref{sec: no-2}, we revisit the crystalline module and de~Rham module of $\eR_{\tau}(\eC_{\ell})$ in Section~\ref{sec: Cris-RC} and \ref{sec: dR-RC}, respectively.
The crystalline--de~Rham comparison isomorphism is explicitly constructed in Section~\ref{sec: cris-dR-iso}, leading to
a detailed portrayal of the crystalline--de~Rham period matrix in Section~\ref{sec: CD matrix}. 
The proof of Theorem~\ref{thm: main-2} is provided in Section~\ref{sec: vCSF}.

In Section~\ref{sec: tran-agv}, we prove Theorem~\ref{thm: main-1} on the algebraic independence of $v$-adic arithmetic gamma values in Section~\ref{sec: AI-agv}, assuming the transcendence result of crystalline periods in Theorem~\ref{thm: main-3}.
Also, the distribution associated with $v$-adic arithmetic gamma values is discussed in Section~\ref{sec: agv-dist}.

In Section~\ref{sec: tran-FM}, we lay  the groundwork for the proof of Theorem~\ref{thm: main-3}. 
Section~\ref{sec: v to infinite} introduces additional notation specifically concerning the treatment of $v$ as the infinite place of the projective $\frac{1}{v}$-line.
For completeness, Section~\ref{sec: HJC} offers the definition of Hartl--Juschka modules, the associated motivic Galois groups, and the de~Rham pairing.
Also stated in Section~\ref{sec: HJC} is Papanikolas's theorem, which serves as the function field analogue of Grothedieck period conjecture.
The Hodge--Pink cocharacters of a Hartl--Juschka module is defined in Section~\ref{sec: H-P-co}.

In Section~\ref{sec: Car-HJ-module}, we introduce Carlitz-type Hartl--Juschka modules and examine their corresponding Betti modules.
The definition of the module $\tilde{\cC}_{d\ell}$, as well as its restriction of scalars $\Rcal_{\sigma}(\tilde{\cC}_{d\ell})$, is provided in Section~\ref{sec: CHJ}.
Section~\ref{sec: CHJ-U} gives a detailed account of the uniformizability of $\Rcal_{\sigma}(\tilde{\cC}_{d\ell})$, after writing down explicitly the structure of the corresponding Betti modules.
We articulate how the crystalline--de Rham periods of $\eR_{\tau}(\eC_{d\ell})$ can be expressed as ``abelian integrals'' in Section~\ref{sec: CHJ-P}.
Finally, the motivic Galois group of $\Rcal_{\sigma}(\tilde{\cC}_{d\ell})$ is determined in Section~\ref{sec: CHJ-MG-HPC}, culminating with the proof of Theorem~\ref{thm: main-3}.

\section{\texorpdfstring{$v$}{v}-adic arithmetic gamma values}\label{sec: agv}
\subsection{Notation} \label{sec: notation}
Throughout this paper, we fix the following notation. Let $p$ be a prime number, and let $q$ be a power of $p$. Denote by $\FF_q$ the finite field of $q$ elements. Let $\theta$ be a variable and let $A:=\FF_{q}[\theta]$ be the polynomial ring in $\theta$ over $\FF_q$ with quotient field $k:=\FF_{q}(\theta)$.  For a monic irreducible polynomial $v\in A$,  we still denote by $v$ the corresponding finite place of $k$ when it is clear from the context. The completion of $k$ at $v$ is denoted by $k_{v}$, and the residue field is denoted by $\FF_v$.
In particular, we may regard $\FF_v$ as a subfield of $k_v$ via the Teichm\"uller embedding, which sends $\theta \bmod v$ to $\varepsilon := \lim_{i\rightarrow \infty} \theta^{q^{i\deg v}}$ (see \cite[Chap.~XII Ex.~16]{La02}), and
$k_v$ is then identified with the Laurant series field $\FF_v(\!( \theta - \varepsilon)\!)$. Note that in the case $v=\theta$ we have $\varepsilon=0$, otherwise $\varepsilon$  generates all the roots of unity in the local field $k_v$.

Choose an algebraic closure $\bar{k}_v$ of $k_v$, and denote by $\CC_v$ the completion of $\bar{k}_v$. 
Let $\bar{k}$ be the algebraic closure of $k$ in $\CC_v$.
Let $A_{+}$ be the set of all monic polynomials in $A$. Put $D_{0}:=1$, and for $i\geq 1$, $D_{i}:=\prod_{j=0}^{i-1}(\theta^{q^{i}}-\theta^{q^{j}})\in A$.  For each $i\in \NN$, we denote by $A_{+,i}$ the set of all monic polynomials of degree $i$, and note that $D_{i}=\prod_{a\in A_{+,i}}a$ (see \cite[p.~140]{Car35}). Carlitz \cite{Car35} introduced the so-called {\it{Caritz factorials}} as the following. For a non-negative integer $n$, we write $n=\sum_{i\geq 0} n_{i}q^{i}$ for $0\leq n_{i}<q$ and define 
\[ \ag (n+1):=\prod_{i\geq 0}D_{i}^{n_i} \in A.\] Note that $\ag(n+1)$ is analogous to the  classical Euler $\Gamma$-values $\Gamma(n+1)=n!$ for $n\geq 0$ as one has the property that (due to W.~Sinnott, see \cite[Proposition 2.1.4]{Go88}) 
\begin{equation}\label{E:valCarlitzFac}
{\rm{ord}}_{v} \left(\ag (n+1) \right) =\sum_{e\geq 1} \lfloor \frac{n}{q^{e\deg v}} \rfloor,
\end{equation}
where $q^{\deg v}$ is the cardinality of $A/(v)$.

\subsection{Definition of \texorpdfstring{$v$}{v}-adic arithmetic gamma function}\label{sec: v-adic agv}

Following Goss \cite{Go80, Go88}, and Thakur \cite{T91}, we consider the $v$-adic arithmetic gamma function that interpolates the Carlitz factorials $v$-adically as follows. For each $i\geq 0$, we define 
\[D_{i,v}:=\prod a,\] where $a$ in the product above runs over all monic polynomial of degree $i$ with $v \nmid a$. For any $z\in \ZZ_{p}$, the ring of $p$-adic integers, we express $z$ as $z=\sum_{i\geq 0}z_{i}q^{i}$ for $0\leq z_{i}<q$ and define
\[ \agv (z+1):=\prod_{i\geq 0} (-D_{i,v})^{z_{i}}.\]
It is shown in \cite[Appendix]{Go80} that the infinite product $\prod_{i\geq 0} (-D_{i,v})^{z_{i}}$ converges in $k_{v}^{\times}$, and $\agv: \ZZ_{p}\rightarrow k_{v}^{\times}$ is a continuous function satisfying the following functional equations:

\begin{proposition}\label{prop: FE}
(See \cite[Thm.~4.6.1, Lem.~4.6.2, and Thm.~4.6.3]{Thakur}) Let $v\in A_{+}$ be irreducible. For each $z \in \ZZ_{p}$, we have
\begin{itemize}
    \item[(1)] (Reflection)
    \[
    \agv(z)\cdot \agv(1-z) = \agv(0).
    \]
    \item[(2)] (Translation)
    When $ z \neq 0$, put 
    \[
    r = \text{\rm ord}_q(z):=\max\big\{ m \in \ZZ_{\geq 0} \ \big|\ q^m| z\big\}.
    \]
    Then
    \[
    \frac{\agv(z+1)}{\agv(z)} = -\frac{D_{r,v}}{\prod_{i=0}^{r-1}D_{i,v}^{q-1}}.
    \]
    \item[(3)] (Multiplication)
    For every $n \in \NN$ with $p \nmid n$, we have
    \[
    \prod_{i=0}^{n-1}\agv(z+\frac{i}{n}) = \agv(nz) \cdot \agv(0)^{\frac{n-1}{2}}.
    \]
\end{itemize}
\end{proposition}

As it is known that (see \cite[Theorem~4.4]{T91})
\[
\agv(1-\frac{1}{q-1})^{q-1} = \agv(0) = (-1)^{\deg v -1},
\]
the square root of $\agv(0)$ on the right hand side of the equality in Proposition~\ref{prop: FE}~(3) is chosen to be
\[
\agv(0)^{\frac{1}{2}}:=\agv(1-\frac{1}{q-1})^{\frac{q-1}{2}}.
\]

When taking values at a rational $p$-adic integer $z$, one has the following standard monomial relations which directly come from the $q$-adic expansion of $z$: 

\begin{proposition}\label{prop: MR} Let $v\in A_{+}$ be irreducible. 
Given $\ell \in \NN$ and $c_0,...,c_{\ell-1} \in \ZZ$ with $0\leq c_0,...,c_{\ell-1}<q$, we have
\begin{equation}\label{E:Standard FE}
\agv(1-\sum_{i=0}^{\ell-1}\frac{c_i q^i}{q^\ell-1}) = \prod_{i=0}^{\ell-1}\agv(1-\frac{q^i}{q^\ell-1})^{c_i}.
\end{equation}
\end{proposition}

\begin{remark} \label{rem: relations}
Regarding $\bar{k}$-algebraic relations among $v$-adic arimetic gamma values at rational $p$-adic integers, in~\cite[Sec.~2]{T91} Thakur showed that once we have the  monomial relations~\eqref{E:Standard FE} together with translation~(2) of Proposition~\ref{prop: FE}  (for $z\in \ZZ_p \cap \QQ$), then the reflection~(1) and multiplication~(3) of Proposition~\ref{prop: FE} are valid for $z\in \ZZ_p \cap \QQ$.
\end{remark}

Thakur's analogue of the Gross--Koblitz formula in \cite[Theorem~VI]{T88} implies:

\begin{theorem}\label{thm: GK-formula}
(See \cite[Corollary of Thm.~V]{T88}.) Let $v\in A_{+}$ be irreducible.
For each $z \in \QQ$ with $(q^{\deg v}-1)\cdot z \in \ZZ$, we have
\[
\agv(z) \in \bar{k}^\times.
\]
\end{theorem}

\begin{remark}
In fact, Thakur's formula in \cite[Theorem~VI]{T88} shows that when $z \in \QQ$ with $(q^{\deg v}-1)\cdot z \in \ZZ$, the special value $\agv(z)$ is related to Thakur's analogue of Gauss sums.
Consequently, in this case one knows that 
\[
\agv(z) \in \FF_{q^{\deg v}}(\theta)\Big(\lambda_v,\sqrt[q^{\deg v}-1]{-v}\Big),
\]
where $\lambda_v$ is a non-zero {\it Carlitz $v$-torsion}.
\end{remark}

As a consequence of Thakur's analogue of the Gross--Koblitz formula, we have the following monomial relations:
\begin{corollary}\label{cor: GK-formula} Let $v\in A_{+}$ be irreducible. 
Given a positive integer $n$ with $p \nmid n$, let $r$ be the order of $q^{\deg v} \bmod n$ in $(\ZZ/n\ZZ)^\times$.
For every $z \in \QQ$ satisfying $n\cdot z \in \ZZ$, we have that
\[
\prod_{i=0}^{r-1} \agv(q^{i\deg v}\cdot  z) \in \bar{k}^\times.
\]
\end{corollary}

\begin{proof}
Without loss of generality, we may assume that $z = 1-\frac{a}{n}$ with $0\leq a < n$.
The assumption of $n$ indicates that $n \mid q^{r\deg v}-1$.
Write $q^{r\deg v}-1 = n\cdot n'$.
Then there exist $z_0,...,z_{r\deg v-1} \in \ZZ$ with $0\leq z_0,...,z_{r\deg v-1}<q$ satisfying that
\[
\frac{a}{n} = \frac{an'}{q^{r\deg v}-1} = \sum_{j=0}^{r\deg v-1}\frac{z_j q^j}{q^{r\deg v}-1}.
\]
Therefore
\begin{align*}
\prod_{i=0}^{r-1} \agv(q^{i\deg v}\cdot  z) 
&= \prod_{i=0}^{r-1}\agv\Big(q^{i\deg v} \cdot (1-\frac{a}{n})\Big) \\
&\sim \prod_{i=0}^{r-1}\prod_{j=0}^{r\deg v-1}\agv\Big(1-\frac{z_j q^{j+i\deg v}}{q^{r\deg v}-1}\Big) \\
&\sim \prod_{j=0}^{r\deg v-1} \agv\Big(1-\frac{\sum_{i=0}^{r-1}q^{j+i\deg v}}{q^{r\deg v}-1}\Big)^{z_j} \\
&= \prod_{j=0}^{r\deg v-1} \agv\Big(1-\frac{q^j}{q^{\deg v}-1}\Big)^{z_j} \quad \in \bar{k}^\times \quad \text{(by Theorem~\ref{thm: GK-formula})}.
\end{align*}
Here, for $x,y \in \CC_v^\times$, we denote by $x\sim y$ if $x/y \in \bar{k}^\times$.
\end{proof}

\begin{remark}
Since when $n \mid (q^{\deg v}-1)$, the order of $q^{\deg v} \bmod n$ in $(\ZZ/n\ZZ)^\times$ is $1$.
Therefore Corollary~\ref{cor: GK-formula} is in fact equivalent to Theorem~\ref{thm: GK-formula}. 
\end{remark}

The above monomial relations among special values of $\agv$ at rational $p$-adic integers give rise to the following upper bound of the transcendence degree in question.

\begin{theorem}\label{thm: up-bd} Let $v\in A_{+}$ be irreducible.
Given $\ell \in \NN$, we have 
\[
\trdeg_{\bar{k}} \bar{k}\Big(\agv(z)\ \Big|\ z \in \QQ \text{ with } (q^\ell-1)\cdot z \in \ZZ \Big) \leq \ell - \gcd(\ell,\deg v).
\]
\end{theorem}

\begin{proof}
The functional equation in Proposition~\ref{prop: FE}~(2) and the monomial relation in Proposition~\ref{prop: MR} show that
\[
\bar{k}\Big(\agv(z)\ \Big|\ z \in \QQ \text{ with } (q^\ell-1)\cdot z \in \ZZ \Big)
=
\bar{k}\Big(\agv(1-\frac{q^i}{q^\ell-1})\ \Big|\ 0\leq i \leq \ell-1\Big).
\]
Let $m = \gcd(\ell,\deg v)$ and $\ell' = \ell/m$.
By Proposition~\ref{prop: MR} and Theorem~\ref{thm: GK-formula}, we get
\begin{equation}\label{E:Monomial relations}
\prod_{j=0}^{\ell'-1}\agv(1-\frac{q^{i+mj}}{q^\ell-1})=\agv(1-\frac{q^i}{q^m-1}) \in \bar{k}, \quad \forall i \in \ZZ,\ 0\leq i <m,
\end{equation}
since $(q^{m}-1)$ divides $ (q^{\deg v}-1)$.
The monomial relations~\eqref{E:Monomial relations} for $0\leq i\leq m-1$ give rise to $m$ independent algebraic relations, from which the result follows.
\end{proof}

We will prove in Section~\ref{sec: tran-agv} that the inequality stated in Theorem~\ref{thm: up-bd} is, in fact, an equality.
Before doing so, we first provide ``period interpretations'' of the $v$-adic arithmetic gamma values in question.

\subsection{Product expansions of \texorpdfstring{$v$}{v}-adic arithmetic gamma values}\label{sec: prod agv}

We begin with the following observation:

\begin{lemma}\label{lem: agv-form} Let $v\in A_{+}$ be irreducible of degree $d$.
Given $\ell,s \in \NN$,
let $n_s$ be the smallest non-negative integer such that $s+n_s \ell \geq d$.
We have that
\begin{align*}
\frac{\agv(1-\displaystyle \frac{q^s}{q^{\ell}-1})}{\agv(1-\displaystyle\frac{q^{s-1}}{q^{\ell}-1})^q}
& = \frac{\displaystyle \prod_{n=0}^\infty
\bigg(1-\frac{(\theta-\varepsilon)^{q^{s+n\ell}}}{\theta-\varepsilon^{q^{s+n\ell}}}\bigg)}{\displaystyle \prod_{n=n_s+1}^\infty \bigg(1-\frac{(\theta-\varepsilon)^{q^{s+n\ell-d}}}{\theta-\varepsilon^{q^{s+n\ell-d}}}\bigg)}
\cdot \bigg(\prod_{n=0}^{n_s}(\varepsilon^{q^{s+n\ell}}-\theta)\Big) \\
& \hspace{2.8cm}\cdot 
\begin{cases}
(\theta^{q^{s+n_s\ell-d}}-\theta)^{-1}, & \text{ if $s+n_s\ell > d$,}\\
v^{-1}, & \text{ if $s+n_s\ell = d$.}
\end{cases}
\end{align*}
\end{lemma}

\begin{proof}
From the definition of $\agv$, we get
\[
\frac{\agv(1-\displaystyle \frac{q^s}{q^{\ell}-1})}{\agv(1-\displaystyle\frac{q^{s-1}}{q^{\ell}-1})^q}
= \prod_{n=0}^\infty \frac{D_{s+n\ell,v}}{D_{s-1+n\ell,v}^q}.
\]
Note that for every $n \in \NN$, one has that
\[
\frac{D_n}{D_{n-1}^q} = (\theta^{q^n}-\theta) 
\quad \text{ and } \quad 
D_{n,v} = 
\begin{cases}
\displaystyle\frac{D_n}{v^{q^{n-d}}D_{n-d}}, & \text{ if $n\geq d$,} \\    
D_n, & \text{ if $n<d$.}
\end{cases}
\]
Thus 
\begin{equation}\label{E:Dn,v/Dn-1,v q}
\frac{D_{n,v}}{D_{n-1,v}^q}
= 
\begin{cases}
\displaystyle \frac{\theta^{q^{n}}-\theta}{\theta^{q^{n-d}}-\theta} & \text{ if $n >d$,} \\
\displaystyle  \frac{\theta^{q^n}-\theta}{v},
& \text{ if $n =d$,} \\
\theta^{q^{n}}-\theta, & \text{ if $0<n<d$.}
\end{cases}
\end{equation}
Observe that:
\begin{equation}\label{E:th^{q^n}-th}
\theta^{q^n} - \theta = (\varepsilon^{q^n}-\theta) \cdot \Big(1-\frac{(\theta-\varepsilon)^{q^n}}{\theta-\varepsilon^{q^n}}\Big)
\end{equation}
and
\begin{equation}\label{E:(th^{q^{n}}-th)/(th^{q^{n-d}}-th)}
\frac{\theta^{q^n}-\theta}{\theta^{q^{n-d}}-\theta}
= \bigg(1-\frac{(\theta-\varepsilon)^{q^n}}{\theta-\varepsilon^{q^n}}\bigg)\bigg/\bigg(1-\frac{(\theta-\varepsilon)^{q^{n-d}}}{\theta-\varepsilon^{q^{n}}}\bigg) \quad \text{ when $n >d$.}
\end{equation}
As $n_s$ is the smallest non-negative integer such that $s+n_s \ell \geq d$,
we obtain that
\begin{align*}
&\frac{\agv(1-\displaystyle \frac{q^s}{q^{\ell}-1})}{\agv(1-\displaystyle\frac{q^{s-1}}{q^{\ell}-1})^q}
= \prod_{n=0}^{\infty}\frac{D_{s+n\ell,v}}{D_{s-1+n\ell,v}^q}  \\
=\ &  \left( \prod_{n=n_s+1}^{\infty} \frac{D_{s+n\ell, v}}{D_{s-1+n\ell,v}^q} \right) 
\cdot \left(\prod_{n=0}^{n_{s}-1} \frac{D_{s+n\ell,v}}{D_{s-1+n\ell,v}^q} \right)
\cdot \frac{D_{s+n_{s}\ell,v}}{D_{s-1+n_{s}\ell,v}^q }
\\
=\ & \left( \prod_{n=n_{s}+1}^{\infty} \frac{\theta^{q^{s+n\ell}}-\theta}{\theta^{q^{s+n\ell-d}}-\theta} \right) \cdot \left(\prod_{n=0}^{n_{s}-1} (\theta^{q^{s+n\ell}}-\theta) \right) 
\cdot   \begin{cases}
\displaystyle\frac{\theta^{q^{s+n_s\ell}}-\theta}{\theta^{q^{s+n_s\ell-d}}-\theta}, & \text{ if $s+n_s\ell > d$,}\\
\displaystyle\frac{\theta^{q^{s+n_s\ell}}-\theta}{v}, & \text{ if $s+n_s\ell = d$}
\end{cases} 
\\
& \hspace{13.9cm} 
\text{(by \eqref{E:Dn,v/Dn-1,v q})}
\\
=\ & \left(\prod_{n=n_{s}+1}^{\infty} \frac{1-\displaystyle\frac{(\theta-\epsilon)^{q^{s+n\ell}}}{\theta-\epsilon^{q^{s+n\ell}}}}{1-\displaystyle\frac{(\theta-\epsilon)^{q^{s+n\ell-d}}}{\theta-\epsilon^{q^{s+n\ell}}}}  \right)\cdot
\prod_{n=0}^{n_{s}}(\epsilon^{q^{s+n\ell}}-\theta)(1-\frac{(\theta-\epsilon)^{q^{s+n\ell}}}{\theta-\epsilon^{q^{s+n\ell}}}) 
 \\
&\cdot 
\begin{cases}
(\theta^{q^{s+n_s\ell-d}}-\theta)^{-1}, & \text{ if $s+n_s\ell > d$},\\
v^{-1}, & \text{ if $s+n_s\ell = d$}
\end{cases} 
\hspace{5cm} \text{(by \eqref{E:th^{q^n}-th} \& \eqref{E:(th^{q^{n}}-th)/(th^{q^{n-d}}-th)})}
\\
=\ & \frac{\displaystyle \prod_{n=0}^\infty
\bigg(1-\frac{(\theta-\varepsilon)^{q^{s+n\ell}}}{\theta-\varepsilon^{q^{s+n\ell}}}\bigg)}{\displaystyle \prod_{n=n_s+1}^\infty \bigg(1-\frac{(\theta-\varepsilon)^{q^{s+n\ell-d}}}{\theta-\varepsilon^{q^{s+n\ell}}}\bigg)}
\cdot \bigg(\prod_{n=0}^{n_s}(\varepsilon^{q^{s+n\ell}}-\theta)\bigg) 
\begin{cases}
(\theta^{q^{s+n_s\ell-d}}-\theta)^{-1}, & \text{ if $s+n_s\ell > d$,}\\
v^{-1}, & \text{ if $s+n_s\ell = d$,}
\end{cases}
\end{align*}
where $\prod_{n=0}^{n_{s}-1} \frac{D_{s+n\ell,v}}{D_{s-1+n\ell,v}^q}  $ occurring in the second identity and $\prod_{n=0}^{n_{s}-1} (\theta^{q^{s+n\ell}}-\theta)$ occurring in the third identity are referred to be $1$ in the case when $n_s=0$,  the forth and fifth identities come from~\eqref{E:(th^{q^{n}}-th)/(th^{q^{n-d}}-th)} and~\eqref{E:th^{q^n}-th} respectively.
\end{proof}

Consequently, for each $z \in \QQ$, let $\langle z \rangle_{\ari} \in \QQ$ be the fractional part of $z$, i.e.\ $0\leq \langle z \rangle_{\ari} < 1$ and $z - \langle z \rangle_{\ari} \in \ZZ$.
We get:

\begin{corollary}\label{cor: agv-form} 
Let $v\in A_{+}$ be irreducible of degree $d$.
Given $\ell\in \NN$ and $s \in \ZZ_{\geq 0}$ with $0\leq s < \ell$, write $s+d = s_0 + n_s'\ell$ where $s_0 \in \ZZ$ with $0\leq s_0<\ell$ and $n_s' \in \ZZ_{\geq 0}$.
Then
\begin{align*}
\frac{\agv\Big(1-\Big\langle \displaystyle \frac{q^{s+d}}{q^{\ell}-1}\Big\rangle_{\ari}\Big)}{\agv\Big(1-\Big\langle \displaystyle\frac{q^{s+d-1}}{q^{\ell}-1}\Big\rangle_{\ari}\Big)^q}
& = \frac{\displaystyle \prod_{n=1}^\infty
\bigg(1-\frac{(\theta-\varepsilon)^{q^{s_0+n\ell}}}{\theta-\varepsilon^{q^{s_0+n\ell}}}\bigg)}{\displaystyle \prod_{n=1}^\infty \bigg(1-\frac{(\theta-\varepsilon)^{q^{s+n\ell}}}{\theta-\varepsilon^{q^{s+n\ell}}}\bigg)}
\cdot \left(\prod_{n=0}^{n_s'-1}(\varepsilon^{q^{s_0+n\ell}}-\theta) \right)
\cdot \frac{\alpha_{s_0}}{\alpha_s} \cdot C_s^{-1},
\end{align*}
where
\begin{equation}\label{eqn: alpha-C}
\alpha_s := 
\begin{cases}
1-\displaystyle\frac{(\theta-\varepsilon)^{q^s}}{\theta-\varepsilon^{q^s}}, & \text{ if $s>0$,} \\
\displaystyle\frac{1}{\theta-\varepsilon}, &\text{ if $s=0$,}
\end{cases}
\quad \text{ and } \quad 
C_s := 
\begin{cases}
1, & \text{ if $s>0$,} \\
-v, & \text{ if $s=0$.}
\end{cases}
\end{equation}
\end{corollary}

\begin{proof}
Recall that  $n_{s_0}$ denotes the smallest non-negative integer so that $s_0+n_{s_0}\ell \geq d$.
Then $n_s' \geq n_{s_0}$ and 
\[
0\leq (n_s'-n_{s_0})\ell = (s_0+n_{s}'\ell) - (s_0 + n_{s_0}\ell)\leq s+d-d =s <\ell. 
\]
Thus $n_s' = n_{s_0}$ and $s_0+n_{s_0}\ell - d = s$.\\

Suppose $s_0>0$.
Then
\begin{align*}
&\ \frac{\agv\Big(1-\Big\langle \displaystyle \frac{q^{s+d}}{q^{\ell}-1}\Big\rangle_{\ari}\Big)}{\agv\Big(1-\Big\langle \displaystyle\frac{q^{s+d-1}}{q^{\ell}-1}\Big\rangle_{\ari}\Big)^q}
= \frac{\agv\Big(1-\displaystyle \frac{q^{s_0}}{q^{\ell}-1}\Big)}{\agv\Big(1- \displaystyle\frac{q^{s_0-1}}{q^{\ell}-1}\Big)^q} \\
= &\ 
\frac{\displaystyle \prod_{n=0}^\infty
\bigg(1-\frac{(\theta-\varepsilon)^{q^{s_0+n\ell}}}{\theta-\varepsilon^{q^{s_0+n\ell}}}\bigg)}{\displaystyle \prod_{n=1}^\infty \bigg(1-\frac{(\theta-\varepsilon)^{q^{s+n\ell}}}{\theta-\varepsilon^{q^{s+n\ell}}}\bigg)}
\cdot \bigg(\prod_{n=0}^{n_s'}(\varepsilon^{q^{s_0+n\ell}}-\theta)\bigg)  
\begin{cases}
(\theta^{q^{s}}-\theta)^{-1}, & \text{ if $s>0$,}\\
v^{-1}, & \text{ if $s=0$}
\end{cases} \\
=&\ 
\frac{\displaystyle \alpha_{s_0} \cdot \prod_{n=1}^\infty
\bigg(1-\frac{(\theta-\varepsilon)^{q^{s_0+n\ell}}}{\theta-\varepsilon^{q^{s_0+n\ell}}}\bigg)}{\displaystyle \prod_{n=1}^\infty \bigg(1-\frac{(\theta-\varepsilon)^{q^{s+n\ell}}}{\theta-\varepsilon^{q^{s+n\ell}}}\bigg)}
\cdot (\theta-\varepsilon^{q^s}) \cdot \bigg(\prod_{n=0}^{n_s'-1}(\varepsilon^{q^{s_0+n\ell}}-\theta)\bigg) 
\begin{cases}
(\theta-\theta^{q^{s}})^{-1}, & \text{ if $s>0$,}\\
(-v)^{-1}, & \text{ if $s=0$}
\end{cases} \\
=&\ 
\frac{\displaystyle \prod_{n=1}^\infty
\bigg(1-\frac{(\theta-\varepsilon)^{q^{s_0+n\ell}}}{\theta-\varepsilon^{q^{s_0+n\ell}}}\bigg)}{\displaystyle \prod_{n=1}^\infty \bigg(1-\frac{(\theta-\varepsilon)^{q^{s+n\ell}}}{\theta-\varepsilon^{q^{s+n\ell}}}\bigg)}
\cdot \left(\prod_{n=0}^{n_s'-1}(\varepsilon^{q^{s_0+n\ell}}-\theta) \right)
\cdot \frac{\alpha_{s_0}}{\alpha_s} \cdot C_s^{-1}.
\end{align*}

When $s_0 =0$, one has that $n_s'=n_{s_0}=n_0>0$ and $n_\ell = n_0-1$ (as $s+d = 0+n_0\ell = \ell + (n_0-1)\ell$). 
Then
\begin{align*}
&\ \frac{\agv(1-\left\langle \displaystyle \frac{q^{s+d}}{q^{\ell}-1}\right\rangle)}{\agv(1-\left\langle \displaystyle\frac{q^{s+d-1}}{q^{\ell}-1}\right\rangle)^q}
= \frac{\agv(1-\displaystyle \frac{1}{q^{\ell}-1})}{\agv(1- \displaystyle\frac{q^{\ell-1}}{q^{\ell}-1})^q} 
= (-1)\cdot \frac{\agv(1-\displaystyle \frac{q^{\ell}}{q^{\ell}-1})}{\agv(1- \displaystyle\frac{q^{\ell-1}}{q^{\ell}-1})^q}
\\
= &\ 
\frac{\displaystyle \prod_{n=0}^\infty
\bigg(1-\frac{(\theta-\varepsilon)^{q^{\ell+n\ell}}}{\theta-\varepsilon^{q^{\ell+n\ell}}}\bigg)}{\displaystyle \prod_{n=1}^\infty \bigg(1-\frac{(\theta-\varepsilon)^{q^{s+n\ell}}}{\theta-\varepsilon^{q^{s+n\ell}}}\bigg)}
\cdot \bigg(\prod_{n=0}^{n_\ell}(\varepsilon^{q^{\ell+n\ell}}-\theta)\bigg) \cdot (-1)\cdot
\begin{cases}
(\theta^{q^{s}}-\theta)^{-1}, & \text{ if $s>0$,}\\
v^{-1}, & \text{ if $s=0$}
\end{cases} \\
=&\ 
\frac{\displaystyle \prod_{n=1}^\infty
\bigg(1-\frac{(\theta-\varepsilon)^{q^{n\ell}}}{\theta-\varepsilon^{q^{n\ell}}}\bigg)}{\displaystyle \prod_{n=1}^\infty \bigg(1-\frac{(\theta-\varepsilon)^{q^{s+n\ell}}}{\theta-\varepsilon^{q^{s+n\ell}}}\bigg)}
\cdot \bigg(\prod_{n=0}^{n_0-1}(\varepsilon^{q^{n\ell}}-\theta)\bigg) \cdot
\frac{\theta-\varepsilon^{q^{s+d}}}{\theta-\varepsilon}\cdot 
\begin{cases}
(\theta-\theta^{q^{s}})^{-1}, & \text{ if $s>0$,}\\
(-v)^{-1}, & \text{ if $s=0$}
\end{cases} \\
=&\ 
\frac{\displaystyle \prod_{n=1}^\infty
\bigg(1-\frac{(\theta-\varepsilon)^{q^{n\ell}}}{\theta-\varepsilon^{q^{n\ell}}}\bigg)}{\displaystyle \prod_{n=1}^\infty \bigg(1-\frac{(\theta-\varepsilon)^{q^{s+n\ell}}}{\theta-\varepsilon^{q^{s+n\ell}}}\bigg)}
\cdot \left(\prod_{n=0}^{n_s'-1}(\varepsilon^{q^{n\ell}}-\theta) \right)
\cdot \frac{\alpha_{0}}{\alpha_s} \cdot C_s^{-1}.
\end{align*}
Therefore the proof is complete.
\end{proof}

\begin{remark} Let notation be given as above.
Suppose $\ell \geq d$. 
Since 
\[
\agv(1-\frac{q^s}{q^{\ell}-1}) = \agv(1-\frac{q^{s+\ell}}{q^{\ell}-1}) \cdot (-D_{s,v}),
\]
we compute that
for every integer $s$ with $0\leq s < \ell$,
\begin{align*}
\agv(1-\frac{q^s}{q^{\ell}-1})^{1-q^{\ell}}
& = \frac{\agv(1-\displaystyle \frac{q^s}{q^{\ell}-1} ) }{\agv(1-\displaystyle\frac{q^s}{q^{\ell}-1})^{q^{\ell}} }=\frac{(-D_{s,v}) \cdot \agv(1-\displaystyle \frac{q^{s+\ell}}{q^{\ell}-1})}{\agv(1-\displaystyle\frac{q^{s}}{q^{\ell}-1})^{q^{\ell}}} \\
&= (-D_{s,v})\cdot \prod_{j=0}^{\ell-1} \left(\frac{\agv(1-\displaystyle \frac{q^{s+\ell-j}}{q^{\ell}-1})}{\agv(1-\displaystyle\frac{q^{s+\ell-1-j}}{q^{\ell}-1})^q}\right)^{q^j} \\
& = D_{s,v}^{1-q^{\ell}}\cdot (-D_{s+\ell,v})\cdot \prod_{j=0}^{\ell-1}\prod_{n=2}^\infty\left(\displaystyle\frac{1-\frac{\displaystyle(\theta-\varepsilon)^{q^{s+n\ell-j}}}{\displaystyle\theta-\varepsilon^{q^{s+n\ell-j}}}}{\displaystyle 1-\frac{(\theta-\varepsilon)^{q^{s+n\ell-j-d}}}{\theta-\varepsilon^{q^{s+n\ell-j}}}}
\right)^{q^j}.
\end{align*}
In particular, when $\ell=d$, these values are actually algebraic (the Gross--Koblitz phenomenon). 
We get
\begin{align*}
\agv(1-\frac{q^s}{q^d-1})^{1-q^d}
&= D_{s,v}^{1-q^d} \cdot (-D_{s+d,v}) \cdot \prod_{j=0}^{d-1}(1-\frac{(\theta-\varepsilon)^{q^{s+d}}}{(\theta-\varepsilon^{q^{s-j}})^{q^j}})^{-1} \\ 
&= (-1)\cdot D_s^{1-q^d} \cdot \frac{D_{s+d}}{v^{q^s}\cdot D_s}
\cdot \prod_{j=0}^{d-1}\frac{\theta^{q^j}-\varepsilon^{q^s}}{\theta^{q^j}-\theta^{q^{s+d}}} \\
&= (-1)^{d-1} \cdot v^{-q^{s}} \cdot \prod_{j=0}^{d-1}(\theta^{q^j}-\varepsilon^{q^s})\\
&=(-1)^{d-1}\cdot \prod_{j=0}^{d-1}(\theta-\varepsilon^{q^{s-j}})^{-q^{s}}\cdot \prod_{j=0}^{d-1}(\theta-\varepsilon^{q^{s-j}})^{q^j} \\
&= (-1)^{d-1}\cdot \prod_{j=0}^{d-1}(\theta-\varepsilon^{q^{s-j}})^{-(q^{s}-q^j)},
\end{align*}
which coincides with Thakur's formula in \cite[Theorem~V]{T88}, indicating the exponents of the above decomposition corresponds to the `Stickelberger element' associated with the constant field extension of degree $d$.
\end{remark}

In the following section, we will connect the infinite product
\[
\prod_{n=0}^\infty \left(
1-\frac{(\theta-\varepsilon)^{q^{s+n\ell}}}{\theta-\varepsilon^{q^{s+n\ell}}}
\right), \quad s \in \ZZ \text{ with } 0< s\leq \ell,
\]
with the transcendental periods arising from the $v$-adic ``crystalline--de~Rham comparison isomorphism'' for the Carlitz $t$-motive of degree $\ell$.

\section{Carlitz \texorpdfstring{$t$}{t}-motive and the crystalline--de~Rham comparison}\label{sec: Car-cry-dR}

\subsection{Carlitz \texorpdfstring{$t$}{t}-motive}\label{sec: Car-t}

For each (commutative) $\FF_q$-algebra $\cA$,
let $\cA[t]$ be the polynomial ring with another variable $t$ over $\cA$.
Let $\cA[t,\tau] = \cA[t][\tau]$ be the twisted polynomial ring over $\cA[t]$ satisfying the following multiplication law:
\[
\tau\cdot f(t) = f^{(1)}(t) \cdot \tau, \quad \forall f(t) \in \cA[t],
\]
where for every $f(t) = \sum_{i=0}^r a_i t^i \in \cA[t]$, 
$f^{(n)}(t)$ is the $n$th Frobenius twist of $f$:
\[
f^{(n)}(t):= \sum_{i=0}^r a_i^{q^n} t^i  \in \cA[t], \quad \forall n \in \ZZ_{\geq 0}.
\]
Throughout this section, we
fix a positive integer $\ell$.

\begin{definition}\label{Def:Cl}
The \emph{Carlitz $t$-motive of degree $\ell$ over $\FF_{q}[\theta]$} is the $\FF_{q}[\theta][t,\tau^{\ell}]$-module  $\eC_{\ell}:= \FF_{q}[\theta][t]$ with the following $\tau^{\ell}$-action
\[
\tau^{\ell}\cdot m:= (t-\theta)\cdot m^{(\ell)}, \quad \forall m \in \eC_{\ell}.
\] 
The \emph{restriction of scalars of $\eC_{\ell}$} is 
the following $\FF_{q}[\theta][t,\tau]$-module:
\[
\eR_{\tau}(\eC_{\ell}):= \FF_{q}[\theta][\tau] \otimes_{\FF_{q}[\theta][\tau^{\ell}]}\big(\FF_{q}[\theta][t]\big).
\]
Note that we may identify
\begin{equation}\label{eqn: Rt-id-1}
\eR_{\tau}(C_{\ell}) = \bigoplus_{i=0}^{\ell-1} \tau^i \otimes \big(\FF_{q}[\theta][t]\big) \cong \prod_{i=0}^{\ell-1}\FF_{q}[\theta][t],
\end{equation}
and the equipped $\tau$-action is given by
\[
\tau\cdot (m_0,m_1,\cdots,m_{\ell-1}) = \big((t-\theta)\cdot m_{\ell-1}^{(1)}, m_0^{(1)},\cdots, m_{\ell-2}^{(1)}\big).
\]  
\end{definition}

For clarity in the subsequent discussions, we set up the following notation:

\begin{definition}\label{Def:ni and ni'}
Let $d$ and $\ell$ be two positive integers.
For every integer $i$ with $0\leq i<\ell$, we put
\[
n_i = \min\{n \in \ZZ_{\geq 0}\mid i+n\ell\geq d\} 
\quad \text{ and } \quad 
i_1:= i+n_i\ell-d\quad  (\text{for which } 0\leq i_1<\ell).
\]
Also, we take $i_0$ to be the unique integer with $0\leq i_0<\ell$ and $i+d \equiv i_0 \bmod \ell$, and set $n_i' := (i+d-i_0)/\ell \in \ZZ_{\geq 0}$.
\end{definition}

In particular,
the following explicit $\tau^{d}$-action on $\eR_{\tau}(C_{\ell})$ will be used for period interpretation of $v$-adic special gamma values in the crystalline-de Rham comparison isomorphism when $d = \deg v$.

\begin{lemma}\label{lem: tau-d}\label{eqn: tau-d} Let notation be given as above. Then we have
\begin{align*} 
\tau^d \cdot (m_0,...,m_{\ell-1})
& = (m_0',...,m_{\ell-1}'),  \notag \\
& 
\text{ where }
m_i' = \left(\prod_{n=0}^{n_i-1}(t-\theta^{q^{i+n\ell}}) \right)\cdot m_{i_1}^{(d)},\quad \text{ for } 0\leq i <\ell.
\end{align*}
\end{lemma}

\begin{proof}
For each integer $i$ with $0\leq i<\ell$,
we write $i+d = i_0 + n_{i}'\ell$ for unique integers $i_0$ and $n_i'$ with $0\leq i_0<\ell$.
Then
\[
m_{i_0}' = \left(\prod_{n=0}^{n_i'-1}(t-\theta^{q^{i_0+n\ell}})\right)\cdot m_{i}^{(d)}.
\]
Based on Definition~\ref{Def:ni and ni'}, as $n_i' \geq n_{i_0}$ and 
\[
d \leq i_0+n_{i_0}\ell \leq 
i_0+n_i'\ell
= i+d <\ell+d,
\]
we get
\[
0\leq n_i'\ell -n_{i_0}\ell< (\ell+d)-(d) = \ell,
\]
whence $n_i' = n_{i_0}$.
Hence $i = i_0+n_{i_0}\ell -d$ = $(i_0)_1$ and the result holds.
\end{proof}

In this section, our primary goal is to explicitly illustrate the crystalline--de~Rham comparison isomorphism for $\eR_{\tau}(\eC_{\ell})$ at $v$.

\subsection{Notation}\label{sec: no-2}

Put $d:= \deg v$ as before, and
write 
\[
v = \theta^d + \epsilon_1\theta^{d-1}+\cdots+\epsilon_{d-1}\theta+\epsilon_d, \quad 
\text{ where $\epsilon_1,...,\epsilon_d \in \FF_q$.}
\]
Let $\iota: \FF_{q}[t]\cong \FF_q[\theta]$ be the $\FF_q$-algebra isomorphism with $\iota(t) = \theta$.
Set
\[
T_v:= \iota^{-1}(v) = t^d + \epsilon_1 t^{d-1}+\cdots+\epsilon_{d-1} t +\epsilon_d,
\]
which is monic irreducible of degree $d$ in $\FF_q[t]$.
We may embed $\FF_q[t]$ into the formal power series ring $\FF_{q^d}[\![T_v]\!]$ (identified with the completion of $\FF_q[t]$ at $T_v$) which is compatible with $\iota$, i.e.\ $t \equiv \varepsilon \bmod T_v$, and $\iota$ is extended to an $\FF_{q^d}$-isomorphism $\FF_{q^d}[\![T_v]\!] \cong \FF_{q^d}[\![v]\!]$ sending $T_v$ to $v$.
These settings can be illustrated in the following diagram:

\[
\SelectTips{cm}{}
\xymatrixrowsep{0.6cm}
\xymatrixcolsep{0.3cm}
\xymatrix{
\FF_{q^d}[\![t-\varepsilon]\!] \ar@{=}[r]& \FF_{q^d}[\![T_v]\!] \ar@{->}^{\sim}[rrrrr] & & & & & \FF_{q^d}[\![v]\!] \ar@{=}[r] & \FF_{q^d}[\![\theta-\varepsilon]\!] \\
& & \FF_{q^d}[t] \ar@{-}[lu] \ar@{->}^{\sim}[rrr] & & &\FF_{q^d}[\theta] \ar@{-}[ru] & & \\
& & \FF_{q}[t] \ar@{-}[luu]\ar@{-}[u] \ar@{->}^{\sim}_{\iota}[rrr] & & &  \FF_{q}[\theta] \ar@{-}[ruu] \ar@{-}[u] & &
}
\]
In particular, the embedding $\FF_q[t]\hookrightarrow \FF_{q^d}[\![T_v]\!]$
sends $t$ to a power series in $T_v$ as
\begin{align}\label{eqn: T_v-expansion of t}
& t(T_v) = \sum_{n=0}^\infty a_n T_v^n \in \FF_{q^d}[\![T_v]\!], \quad \text{ where } \quad 
a_0 = \varepsilon \text{ and }
a_1 = 
\prod_{i=1}^{d-1}(\varepsilon-\varepsilon^{q^i})^{-1},\\
\label{eqn: evaluating at v} & \text{and} \quad t(T_v)\big|_{T_v=v} = \varepsilon+ \left(\prod_{i=1}^{d-1}(\varepsilon-\varepsilon^{q^i})^{-1}\right)v +\cdots = \theta \quad \in \FF_{q^d}[\![v]\!].
\end{align}

\begin{Subsubsec}{Comparison between twistings}\label{subsubsec: com-twist}
Let $A_v := \FF_{q^d}[\![v]\!]$ be the valuation ring of $k_v$.
Following Hartl--Kim we let 
\[
A_{v}[\![T_v,T_v^{-1}\}\!\}:= \Big\{\sum_{i\in \ZZ}b_iT_v^{i}\in A_{v}[\![T_v,T_v^{-1}]\!]\ \Big|\ \lim_{i\rightarrow -\infty}b_i v^{-ni} = 0, \quad \forall n>0\Big\},
\]
on which we define the $n$-th Frobenius twist as follows.
\begin{definition}\label{Def: Twisting on Av[[Tv]]}
Given $f = \sum_{i\in \ZZ} f_i T_v^i \in A_{v}[\![T_v,T_v^{-1}\}\!\}$, its $n$-th Frobenius twist is given by
\[
f^{(n)}:= \sum_{i\in \ZZ} f_i^{q^{n}}  T_v^i.
\]
\end{definition}

Note that the embedding $\FF_q[t]\hookrightarrow \FF_{q^d}[\![T_v]\!]$ given in \eqref{eqn: T_v-expansion of t}
induces an $A_v$-algebra embedding 
$A_v[t] \hookrightarrow A_{v}[\![T_v,T_v^{-1}\}\!\}$.
However, the respective $n$-th Frobenius twists on $A_v[t]$ and on $A_{v}[\![T_v,T_v^{-1}\}\!\}$ are only compatible when $d$ divides $n$.
Indeed, regarding $t = t(T_v)$ as a power series in $\FF_{q^d}[\![T_v]\!]$ given in \eqref{eqn: T_v-expansion of t}, we get
\begin{align*}
t^{(n)}(T_v) & = 
\bigg(\varepsilon + \Big(\prod_{i=1}^{d-1}(\varepsilon-\varepsilon^{q^i})\Big)\cdot T_v + \cdots \bigg)^{(n)}
=
\varepsilon^{q^n} + \Big(\prod_{i=1}^{d-1}(\varepsilon-\varepsilon^{q^i})\Big)^{q^n}\cdot T_v + \cdots \\
& = \varepsilon + \Big(\prod_{i=1}^{d-1}(\varepsilon-\varepsilon^{q^i})\Big)\cdot T_v + \cdots = t(T_v) \quad \in \FF_{q^d}[\![T_v]\!]
\quad \text{ if and only if $d \mid n$.}
\end{align*}
We would like to remind readers to be aware of the above ``non-compatibility'' between these two types of Frobenius twistings. 
\end{Subsubsec}

\begin{remark}
The twisted polynomial ring $A_{v}[\![T_v,T_v^{-1}\}\!\}[\tau]$ over $A_{v}[\![T_v,T_v^{-1}\}\!\}$ satisfies 
\[
\tau \cdot f(T_v) = f^{(1)}(T_v) \tau, \quad \forall f(T_v) \in A_{v}[\![T_v,T_v^{-1}\}\!\}.
\]
Because of the non-compatibility of the Frobenius twists on $A_v[t]$ and on $A_{v}[\![T_v,T_v^{-1}\}\!\}$,
we remark that $A_v[t][\tau]$ is not a subring of $A_{v}[\![T_v,T_v^{-1}\}\!\}[\tau]$ when $d>1$, 
but both do contain the subring $A_v[t][\tau^d]$.
\end{remark}

\subsection{Crystalline module of \texorpdfstring{$\eR_{\tau}(\eC_{\ell})$}{Rtau(Cl)}}\label{sec: Cris-RC}
Put $\theta_v:=\theta-\varepsilon$. 
Let
\[
A_{v}:= \FF_{q^{d}}[\![\theta_v]\!] =  \FF_{q^{d}}[\![v]\!] \subset k_{v}
\quad \text{ and } \quad 
\fm_{v}:= v\cdot A_{v} = \theta_v \cdot A_{v}.
\] be the valuation ring of $k_{v}$ with maximal ideal \[ \fm_{v}:= v\cdot A_{v} = \theta_v \cdot A_{v}.\]
Note that as $\FF_{q^d} \cong A_{v}/ \fm_{v}$ is an $\FF_{q}[\theta]$-algebra, we define $\overline{\eR_{\tau} (\eC_{\ell})}$ to be the reduction of 
$\eR_{\tau} (\eC_{\ell})$ modulo $m_{v}$, i.e.
\begin{equation}\label{eqn: red-RC}
\overline{\eR_{\tau} (\eC_{\ell})} = \frac{A_{v}}{\fm_{v}}\otimes_{\FF_{q}[\theta]}\big(\eR_{\tau}(\eC_{\ell})\big). 
\end{equation}
By \eqref{eqn: Rt-id-1}, we have the following identification
\begin{equation}\label{eqn: red-C}
\overline{\eR_{\tau} (\eC_{\ell})}
=
\prod_{i=0}^{\ell-1}\FF_{q^{d}}[t],
\end{equation}
which becomes an $\FF_{q^{d}}[t,\tau]$-module with the $\tau$-action given by
\[
\tau \cdot \big(\bar{m}_0,\bar{m}_1,\cdots, \bar{m}_{\ell-1}\big)
= \Big((t-\varepsilon)\cdot \bar{m}_{\ell-1}^{(1)},\bar{m}_0^{(1)},\cdots, \bar{m}_{\ell-2}^{(1)}\Big),
\] where we note that $\theta\equiv \epsilon $ (mod $v$).

\begin{definition}
(See \cite[Definition~3.5.14]{HK20}) 
The $v$-adic crystalline module of the $t$-motive $\eR_{\tau}(\eC_{\ell})$ is 
\[
H_{\mathrm{cris}}\big(\eR_{\tau}(\eC_{\ell}), \FF_{q^{d}}(\!(T_v)\!)\big):= \FF_{q^{d}}(\!(T_v)\!) \otimes_{\FF_{q^{d}}[t]}\Big(\tau^d \cdot \overline{\eR_{\tau}(\eC_{\ell})}\Big). 
\]
\end{definition}

\begin{remark}\label{rem: isocrystal}
We have a natural $\tau^d$-action on $H_{\mathrm{cris}}\big(\eR_{\tau}(\eC_{\ell}), \FF_{q^{d}}(\!(T_v)\!)\big)$ as follows:
\begin{align*}
\tau^d \cdot \left(\big(\sum_{i} f_i T_v^i\big) \otimes m\right) & := \big(\sum_{i} f_i^{q^d} T_v^i\big) \otimes \tau^d m \\
& = \big(\sum_{i} f_i T_v^i\big) \otimes \tau^d m
, \quad \forall \sum_{i} f_i T_v^i \in \FF_{q^{d}}(\!(T_v)\!),\ m \in \tau^d \cdot \overline{\eR_{\tau}(C_{\ell})}.
\end{align*}
(The pair $\big(H_{\mathrm{cris}}\big(\eR_{\tau}(\eC_{\ell}), \FF_{q^{d}}(\!(T_v)\!)\big) ,\tau^d\big)$ is called the \emph{$T_v$-isocrystal of $\eR_{\tau}(C_{\ell})$},  
see \cite[Example~3.5.7]{HK20}.)
\end{remark}

\subsection{de~Rham module of \texorpdfstring{$\eR_{\tau}(\eC_{\ell})$}{Rtau(Cl)}}\label{sec: dR-RC}

Let $\cR_{\tau}(\eC_{\ell})$ be the base change of $\eR_{\tau}(\eC_{\ell})$ to $k$, i.e.\ 
\[
\cR_{\tau}(\eC_{\ell}) = k \otimes_A \eR_{\tau}(\eC_{\ell}),
\]
and the $\tau$-action is extended by
\[
\tau \cdot (\alpha \otimes m) := \alpha^q \otimes (\tau \cdot m), \quad \forall \alpha \in k,\ m \in \eR_{\tau}(\eC_{\ell}).
\]
The {\it de~Rham module of $\eR_{\tau}(\eC_{\ell})$ over $k$} is defined as follows:
\begin{align}\label{eqn: dR-RC}
H_{\textrm{dR}}\big(\eR_\tau(\eC_\ell),k\big) &:= \frac{\langle \tau \eR_{\tau}(\eC_{\ell})\rangle_k}{(t-\theta)\cdot \langle \tau \eR_{\tau}(\eC_{\ell})\rangle_k}.
\end{align}
Here $\langle \tau \eR_{\tau}(\eC_{\ell})\rangle_k$ is the $k$-subspace of $\cR_{\tau}(\eC_{\ell})$ spanned by $\tau \eR_{\tau}(\eC_{\ell})$.

\begin{remark}
The definition of the de Rham module follows from Hartl--Juschka \cite[\S 2.3.5]{HJ20}.
Regarding $\cR_{\tau}(\eC_{\ell})$ as the $t$-motive associated to the ``Carlitz module of degree $\ell$'' over $k$, the de Rham module of $\cR_{\tau}(\eC_{\ell})$ can be defined alternatively via the theory of biderivations developed by Anderson, Deligne, Gekeler, and Yu for Drinfeld modules (see \cite{Ge89}, and \cite{Yu90}), and extended by Brownawell--Papanikolas for higher dimensional $t$-modules (see \cite{BP02}).
For the equivalence between these two definitions, we refer the reader to \cite[Prop.~4.1.3]{NP22} for a detailed comparison.
\end{remark}

\begin{lemma}\label{lem: dR-com}
The natural inclusion $\tau^d \eR_\tau(\eC_\ell) \subset \tau \eR_\tau(\eC_\ell)$
induces the $k$-linear isomorphism
\[
\frac{\langle \tau^d \eR_\tau(\eC_\ell)\rangle_k}{
(t-\theta)\cdot \langle \tau^d \eR_\tau(\eC_\ell)\rangle_k
}
\stackrel{\sim}{\longrightarrow}
H_{\mathrm{dR}}\big(\eR_\tau(\eC_\ell),k\big).
\]
\end{lemma}

\begin{proof}
From the description of the $\tau^d$-action on $\eR_{\tau}(\eC_{\ell})$ in Lemma~\ref{lem: tau-d}, we observe that
\[
\langle \tau^d \eR_\tau(\eC_{\ell})\rangle_k \supset \left(\prod_{i=1}^{d-1}(t-\theta^{q^i})\right) \cdot \langle \tau \eR_\tau(\eC_{\ell})\rangle_k. 
\]
By Chinese Remainder Theorem, we get
\begin{align*}
& (t-\theta)\cdot \langle \tau \eR_\tau(\eC_{\ell})\rangle_k
+ \langle \tau^d \eR_\tau(\eC_{\ell})\rangle_k \\ 
\supset\ &  (t-\theta)\cdot \langle \tau \eR_\tau(\eC_{\ell})\rangle_k + \left(\prod_{i=1}^{d-1}(t-\theta^{q^i})\right) \cdot \langle \tau \eR_\tau(\eC_{\ell})\rangle_k \\
=\ &  \langle \tau \eR_\tau(\eC_{\ell})\rangle_k,
\end{align*}
whence the canonical map
\[
\langle \tau^d \eR_\tau(\eC_{\ell})\rangle_k \longrightarrow \frac{\langle \tau \eR_\tau(\eC_{\ell})\rangle_k}{(t-\theta)\cdot \langle \tau \eR_\tau(\eC_{\ell})\rangle_k} = H_{\textrm{dR}}\big(\eR_{\tau}(\eC_{\ell}),k\big)
\]
is surjective with kernel $(t-\theta) \cdot \langle \tau^d \eR_\tau(\eC_{\ell})\rangle_k.$
\end{proof}

Consider the following base change:
\[
H_{\textrm{dR}}\big(\eR_\tau(\eC_\ell),k_v\big):= k_v \otimes_{k} H_{\textrm{dR}}\big(\eR_\tau(\eC_\ell),k\big),
\]
called the {\it de~Rham module of $\eR_\tau(\eC_\ell)$ over $k_v$}.
Also, we put 
\begin{align}\label{eqn: H-dR}
H_{\textrm{dR}}
\big(\cR_{\tau}(\eC_{\ell}),k_{v}[\![T_v-v]\!]\big) & := k_{v}[\![T_v-v]\!] \otimes_{k[t]} \langle \tau^d \eR_{\tau}(\eC_{\ell})\rangle_k.
\end{align}
Then the above lemma leads to the following $k_v$-linear isomorphism
\begin{equation}\label{eqn: dR-iso-2}
\frac{H_{\textrm{dR}}
\big(\cR_{\tau}(\eC_{\ell}),k_{v}[\![T_v-v]\!]\big) }{(T_v-v)\cdot H_{\textrm{dR}}
\big(\cR_{\tau}(\eC_{\ell}),k_{v}[\![T_v-v]\!]\big) } \stackrel{\sim}{\longrightarrow} H_{\mathrm{dR}}\big(\eR_\tau(\eC_\ell),k_v\big).    
\end{equation}
This will be used in the crystalline--de~Rham comparison isomorphism in the next subsection.

\subsection{Crystalline--de~Rham comparison isomorphism}
\label{sec: cris-dR-iso}
In order to describe the crystalline--de~Rham comparison isomorphism, we recall in Section~\ref{subsubsec: com-twist} that following Hartl--Kim, we define
\[
A_{v}[\![T_v,T_v^{-1}\}\!\}:= \Big\{\sum_{i\in \ZZ}b_iT_v^{i}\in A_{v}[\![T_v,T_v^{-1}]\!]\ \Big|\ \lim_{i\rightarrow -\infty}b_i v^{-ni} = 0, \quad \forall n>0\Big\}.
\]
We further consider the $v$-adic Anderson's scattering (cf.\ \cite{ABP04, GL11, Har11},
and \cite[(3.5.3)]{HK20})
\[
\varpi(T_v):= \prod_{i\in \ZZ_{\geq 0}}(1-\frac{v^{q^{di}}}{T_v}) \quad \in\  A_{v}\{\!\{T_v^{-1}\}\!\} := A_v[\![T_v^{-1}]\!] \cap A_{v}[\![T_v,T_v^{-1}\}\!\}.
\]

\begin{lemma}\label{lem: comp-iso-1}
There exists a unique isomorphism
\begin{align*}
A_{v}[\![T_v,T_v^{-1}\}\!\}[\varpi(T_v)^{-1}] \otimes_{\FF_{q^{d}}[t]}\overline{\eR_{\tau}(\eC_{\ell})}
& \stackrel{\sim}{\longrightarrow} 
A_{v}[\![T_v,T_v^{-1}\}\!\}[\varpi(T_v)^{-1}]\otimes_{\FF_{q}[\theta][t]}\eR_{\tau}(\eC_{\ell}), 
\end{align*}
which is compatible with the $\tau^d$-action on both sides and becomes the identity map modulo $\fm_{v}$.
\end{lemma}

The above lemma actually holds for other $t$-motives (in a more general setting), see \cite[Lemma~3.5.8]{HK20}, \cite[Lemma 7.4]{GL11}, and \cite[Lemma 2.3.1]{Har11}.
We illustrate the isomorphism in Lemma~\ref{lem: comp-iso-1} in concrete terms as follows.
Recall $\theta_v:= \theta-\varepsilon$.
We first observe that:

\begin{lemma}\label{lem: Omega-unit}\label{eqn: Omega} We identify $t$ with the power series $t(T_v)$ in \eqref{eqn: T_v-expansion of t}.
For every non-negative integer $i$ and positive integer $\ell$, we define
\[
\Omega_{\ell,v,i}(T_v) := \prod_{n=0}^\infty(1-\frac{\theta_v^{q^{ \ell n +i}}}{t-\varepsilon^{q^{\ell n +i}}}) .
\]
Then we have 
\[ 
\Omega_{\ell,v,i}(T_v)  \in A_{v}[\![T_v,T_v^{-1}\}\!\} \cap \left(A_{v}[\![T_v,T_v^{-1}\}\!\}[\varpi(T_v)^{-1}]\right)^{\times}.\]

\end{lemma}

\begin{proof}
Given a non-negative integer $i$,
from the $T_v$-expansion of $t$ in \eqref{eqn: T_v-expansion of t} and $t\equiv \varepsilon$ (mod $T_v$)
we get
\begin{align}\label{eqn: tv}
t-\varepsilon^{q^i} & \in
\begin{cases}
\bigg(\displaystyle\prod_{j=1}^{d-1}(\varepsilon-\varepsilon^{q^j})^{-1}\bigg) \cdot T_v \cdot \big(1+T_vA_{v}[\![T_v]\!]\big), & \text{ if $i\equiv 0 \bmod d$,} \\
(\varepsilon-\varepsilon^{q^i}) +T_v A_{v}[\![T_v]\!], & \text{ if $i \not\equiv 0 \bmod d$.}
\end{cases}
\end{align}
Writing
\[
\Omega_{\ell,v,i}(T_v)
= \prod_{\subfrac{n \geq 0}{\ell n+ i \equiv 0 \bmod d}}(1-\frac{\theta_v^{q^{ \ell n +i}}}{t-\varepsilon^{q^{\ell n +i}}}) \cdot \prod_{\subfrac{n \geq 0}{\ell n+i \not\equiv 0 \bmod d}}(1-\frac{\theta_v^{q^{ \ell n +i}}}{t-\varepsilon^{q^{\ell n + i}}}),
\] 
from \eqref{eqn: tv} one has that
\begin{align}\label{eqn: Omega-unit}
\prod_{\subfrac{n \geq 0}{\ell n + i \not\equiv 0 \bmod d}}(1-\frac{\theta_v^{q^{ \ell n +i}}}{t-\varepsilon^{q^{\ell n+i}}})
& \in
1+ \theta_v A_{v}[\![T_v]\!] 
\subset 
\big(A_{v}[\![T_v]\!]\big )^\times \subset 
\left(A_{v}[\![T_v,T_v^{-1}\}\!\}[\varpi(T_v)^{-1}]\right)^{\times}.
\end{align}
On the other hand, if $\gcd(d,\ell) \nmid i$, then
$\ell n + i \not\equiv 0 \bmod d$ for every $n \in \ZZ$.
Hence
\[
\Omega_{\ell,v,i}(T_v) = 
\prod_{\subfrac{n \geq 0}{\ell n+i \not\equiv 0 \bmod d}}(1-\frac{\theta_v^{q^{ \ell n +i}}}{t-\varepsilon^{q^i}})
\]
and the result holds.

Suppose $\gcd(d,\ell) \mid i$.
Write $\text{lcm}(d,\ell) = d\ell_0$, and take $ s_i \in \ZZ_{\geq 0}$ to be minimal so that $s_i \ell + i \equiv 0 \bmod d$. Put $i'=s_i\ell+i$.
Then
\[
\prod_{\subfrac{n \geq 0}{\ell n+i \equiv 0 \bmod d}}(1-\frac{\theta_v^{q^{ \ell n +i}}}{t-\varepsilon^{q^{\ell n+ i}}})
= \prod_{n=0}^\infty
(1-\frac{\theta_v^{q^{  d\ell_0 n +i'}}}{t-\varepsilon}).
\]
Let
\[
g(X) := X^d  +\epsilon_1 X^{d-1} + \cdots + \epsilon_{d-1}X +\epsilon_d - v \quad \in \FF_{q}[v][X] \subset \overline{\FF_q}[v,X]=\overline{\FF_q}[X][v],
\] which is an irreducible element as it is of degree one in $v$ over $\overline{\FF_q}[X]$. It follows that $g(X)$ is the minimal polynomial of $\theta$ over $\FF_{q^d}(v)$.
Since
\begin{align*}
g(X) & \equiv X^d +\epsilon_1 X^{d-1} + \cdots + \epsilon_{d-1}X +\epsilon_d \quad \bmod v  \\
& \hspace{0.5cm} = \prod_{i=0}^{d-1}(X-\varepsilon^{q^i}) \quad \quad \in \FF_{q^d}[X],
\end{align*}
by Hensel's lemma we get that $g(X)$ splits completely in $\FF_{q^d}[\![v]\!][X]$.
We denote by $\theta_{j} \in \FF_{q^d}[\![v]\!] = A_{v}$, $0\leq j<d$, the roots of $g(X)$ satisfying
$\theta_{j} \equiv \varepsilon^{q^j} \bmod \fm_{v}$,
and take $u_j \in A_{v}$ so that 
$u_{j} \theta_v = \theta_j -\varepsilon^{q^j} \in \fm_{v}$.
Note that $\theta_0 =  \theta - \varepsilon = \theta_v$ and $u_0 = 1$.
From the equalities
\begin{equation}\label{E:Tv-v}
T_v - v = (t^d + \epsilon_1 t^{d-1} + \cdots + \epsilon_{d-1} t + \epsilon_d) - v = g(t) = (t-\theta) \cdot \prod_{j=1}^{d-1}(t - \theta_{j})
\end{equation}
and
\begin{equation}\label{E:Tv}
T_v = \prod_{j=0}^{d-1} (t-\varepsilon^{q^j}),
\end{equation}
we obtain that
\begin{equation}\label{E:h(Tv)}
h(T_v) := \frac{1-\displaystyle \frac{v}{T_v}}{1-\displaystyle \frac{\theta_v}{t-\varepsilon}}
= \frac{t-\varepsilon}{T_v} \cdot \frac{T_v-v}{t-\theta}
=\prod_{j=1}^{d-1}\frac{t-\theta_{j}}{t-\epsilon^{q^{j}}}
= \prod_{j=1}^{d-1}(1-\frac{u_{j}\theta_v}{t-\varepsilon^{q^j}}) \quad \in 1 + \theta_v A_{v}[\![T_v]\!],
\end{equation}
where the third identity comes from~\eqref{E:Tv-v} and~\eqref{E:Tv}.
Considering the Frobenius twists of the identity
\[
1-\frac{\theta_v}{t-\varepsilon} = (1-\frac{v}{T_v})\cdot h(T_v)^{-1},
\]
we obtain that (as $i' \equiv 0 \bmod d$)
\begin{equation}\label{eqn: Omega-unit-2}
\prod_{n=0}^{\infty}(1-\frac{\theta_v^{q^{d\ell_0 n +i'}}}{t-\varepsilon})
= \prod_{n=0}^{\infty}(1-\frac{v^{q^{d\ell_0 n +i'}}}{T_v})\cdot \prod_{n=0}^{\infty} h^{(d\ell_0 n +i')}(T_v)^{-1}\in A_v[\![T_v,T_v^{-1}\}\!\}.
\end{equation}

Finally, since
\begin{align*}
& \prod_{n=0}^{\infty}(1-\frac{v^{q^{dn}}}{T_v}) = \varpi(T_v) \quad \in 
\left(A_{v}[\![T_v,T_v^{-1}\}\!\}[\varpi(T_v)^{-1}]\right)^{\times} \\
=\ & \left(\prod_{\subfrac{n \in \ZZ_{\geq 0}}{dn \not\equiv i' \bmod d\ell_0}}(1-\frac{v^{q^{dn}}}{T_v})
\prod_{\subfrac{0\leq n<i'/d}{dn \equiv i' \bmod d\ell_0}}(1-\frac{v^{q^{dn}}}{T_v})\right)
\cdot \left(\prod_{n=0}^{\infty}(1-\frac{v^{q^{d\ell_0 n + i'}}}{T_v})\right) \\
=\ &\left(\prod_{\subfrac{n \in \ZZ_{\geq 0}}{dn \not\equiv i' \bmod d\ell_0}}(1-\frac{v^{q^{dn}}}{T_v})
\prod_{\subfrac{0\leq n<i'/d}{dn \equiv i' \bmod d\ell_0}}(1-\frac{v^{q^{dn}}}{T_v}) \right)
\cdot \left(\prod_{n=0}^{\infty}(1-\frac{\theta_v^{q^{d\ell_0 n+i'}}}{t-\varepsilon}) \cdot h^{(d\ell_0 n + i')}(T_v)\right) \\
=\ & \left( \prod_{n=0}^{\infty}(1-\frac{\theta_v^{q^{d\ell_0 n+i'}}}{t-\varepsilon}) \right) \cdot \left(\prod_{\subfrac{n \in \ZZ_{\geq 0}}{dn  \not\equiv i' \bmod d\ell_0}}(1-\frac{v^{q^{dn}}}{T_v}) 
\prod_{\subfrac{0\leq n<i'/d}{dn \equiv i' \bmod d\ell_0}}(1-\frac{v^{q^{dn}}}{T_v})
\prod_{n=0}^{\infty} h^{(d\ell n+i)}(T_v)\right),
\end{align*}
one gets
\[
\prod_{n=0}^{\infty}(1-\frac{\theta_v^{q^{d\ell_0 n+i'}}}{t-\varepsilon}) \quad \in \left(A_{v}[\![T_v,T_v^{-1}\}\!\}[\varpi(T_v)^{-1}]\right)^{\times}
\] because of \eqref{E:h(Tv)}.
Together with \eqref{eqn: Omega-unit} and \eqref{eqn: Omega-unit-2}, we obtain that
\begin{align*}
\Omega_{\ell,v,i}(T_v)=
\prod_{n=0}^\infty
(1-\frac{\theta_v^{q^{  d\ell_0 n +i'}}}{t-\varepsilon}) & \cdot 
\prod_{\subfrac{n \geq 0}{\ell n + i \not\equiv 0 \bmod d}}(1-\frac{\theta_v^{q^{ \ell n +i}}}{t-\varepsilon^{q^{\ell n+i}}})\\
& \hspace{1cm}
\in A_{v}[\![T_v,T_v^{-1}\}\!\} \cap \left(A_{v}[\![T_v,T_v^{-1}\}\!\}[\varpi(T_v)^{-1}]\right)^{\times}
\end{align*}
as desired.
\end{proof}

\begin{Subsubsec}{Description of the isomorphism in Lemma~\ref{lem: comp-iso-1}}\label{sec: des-iso}
From the identification of $\overline{\eR_{\tau}(\eC_{\ell})}$ in \eqref{eqn: red-C}, we get 
\begin{align*}
&A_{v}[\![T_v,T_v^{-1}\}\!\}[\varpi(T_v)^{-1}] \otimes_{\FF_{q^{d}}[t]}\overline{\eR_{\tau}(\eC_{\ell})} \\
&\hspace{2cm}= \prod_{i=0}^{\ell-1}A_{v}[\![T_v,T_v^{-1}\}\!\}[\varpi(T_v)^{-1}] \otimes_{\FF_{q^{d}}[t]} \FF_{q^{d}}[t] \\
&\hspace{2cm} \cong 
\prod_{i=0}^{\ell-1}A_{v}[\![T_v,T_v^{-1}\}\!\}[\varpi(T_v)^{-1}].
\end{align*}
On the other hand, by \eqref{eqn: Rt-id-1} we have that
\begin{align*}
&A_{v}[\![T_v,T_v^{-1}\}\!\}[\varpi(T_v)^{-1}]\otimes_{\FF_{q}[\theta][t]}\eR_{\tau}(\eC_{\ell}) \\
& \hspace{2cm} = 
\prod_{i=0}^{\ell-1}A_{v}[\![T_v,T_v^{-1}\}\!\}[\varpi(T_v)^{-1}] \otimes_{\FF_{q}[\theta][t]} \FF_{q}[\theta][t]. \\
& \hspace{2cm} \cong 
\prod_{i=0}^{\ell-1}A_{v}[\![T_v,T_v^{-1}\}\!\}[\varpi(T_v)^{-1}].
\end{align*}
Then the isomorphism in Lemma~\ref{lem: comp-iso-1} is given as follows:
every element $(m_0,...,m_{\ell-1})$ in $\prod_{i=0}^{\ell-1}A_{v}[\![T_v,T_v^{-1}\}\!\}[\varpi(T_v)^{-1}]$
is sent to
\begin{equation}\label{eqn: comp-iso-2}
\left(m_0\Omega_{\ell,v,0}, m_1\Omega_{\ell,v,1},\cdots, m_{\ell-1}\Omega_{\ell,v,\ell-1}\right).
\end{equation}
To verify this, note that 
\[
\Omega_{\ell,v,i}(T_v)\equiv 1 \bmod \fm_v \quad \in \frac{A_{v}[\![T_v,T_v^{-1}\}\!\}}{\fm_v \cdot A_{v}[\![T_v,T_v^{-1}\}\!\}} \cong \FF_{q^d}(\!(T_v)\!).
\]
Lemma~\ref{lem: Omega-unit} ensures the bijectivity of \eqref{eqn: comp-iso-2}, which becomes the identity map modulo $\fm_v$.
Moreover, 
for each integer $i$ with $0\leq i < \ell$,
write 
$i+d = i_0 + n_i'\ell$ where $i_0,n_i'$ are unique integers with $0\leq i_0<\ell$ and $n_i' \geq 0$.
Then
\begin{align*}
\Omega_{\ell,v,i}^{(d)}(T_v) & = \prod_{n=0}^\infty (1-\frac{\theta_v^{q^{\ell n+i+d}}}{t-\varepsilon^{q^{\ell n+i+d}}}) = 
\prod_{n=n_i'}^{\infty}(1-\frac{\theta_v^{q^{\ell n+i_0}}}{t-\varepsilon^{q^{\ell n+i_0}}}) \\
&= \left(\prod_{n=0}^{n_i'-1} \frac{t-\varepsilon^{q^{\ell n+i_0}}}{t-\theta^{q^{\ell n + i_0}}}\right) \cdot \Omega_{\ell,v,i_0}(T_v) \quad \in A_v[\![T_v,T_v^{-1}\}\!\},
\end{align*}
where the product $\left(\displaystyle\prod_{n=0}^{n_i'-1} \frac{t-\varepsilon^{q^{\ell n+i_0}}}{t-\theta^{q^{\ell n + i_0}}}\right)$ is defined to be $1$ whenever $n_i'=0$.  We recall the notation $i_1$ given in Definition~\ref{Def:ni and ni'}. 
Since $n_i' = n_{i_0}$ and $(i_0)_1 = i$,
one has that for each $0\leq i<\ell$,
\[
\left(\prod_{n=0}^{n_i-1} (t-\theta^{q^{\ell n + i}})\right) \cdot \Omega_{\ell,v,i_1}^{(d)}(T_v) = 
\left(\prod_{n=0}^{n_i-1} (t-\varepsilon^{q^{\ell n + i}})\right) \cdot \Omega_{\ell,v,i}(T_v) \quad \in A_v[\![T_v,T_v^{-1}\}\!\}.
\]
Finally, 
for every $(m_0,...,m_{\ell-1}) \in A_{v}[\![T_v,T_v^{-1}\}\!\}[\varpi(T_v)^{-1}] \otimes_{\FF_{q^{d}}[t]}\overline{\eR_{\tau}(\eC_{\ell})}$,
by Lemma~\ref{lem: tau-d} we obtain that 
\begin{align*}
&\hspace{1.5cm} \tau^d \cdot \left(m_0\Omega_{\ell,v,0},\cdots, m_{\ell-1}\Omega_{\ell,v,\ell-1}\right) \quad \big(\in A_{v}[\![T_v,T_v^{-1}\}\!\}[\varpi(T_v)^{-1}]\otimes_{\FF_{q}[\theta][t]}\eR_{\tau}(\eC_{\ell})\big) \\
&\hspace{1cm} = \Big(\left(\prod_{n=0}^{n_0-1}(t-\theta^{q^{\ell n}})\right) \cdot m_{(0)_1}^{(d)}\Omega_{\ell,v,(0)_1}^{(d)}, \cdots, \left(\prod_{n=0}^{n_{\ell-1}-1}(t-\theta^{q^{\ell n+\ell-1}})\right) \cdot m_{(\ell-1)_1}^{(d)}\Omega_{\ell,v,(\ell-1)_1}^{(d)}\Big) \\
&\hspace{1cm} = \big( m_0'\Omega_{\ell,v,0} ,\cdots ,  m_{\ell-1}'\Omega_{\ell,v,\ell-1}\big),
\end{align*}
where
\begin{align*}
(m_0',...,m_{\ell-1}') 
& = \Big(\left(\prod_{n=0}^{n_0-1}(t-\varepsilon^{q^{\ell n}})\right) \cdot m_{(0)_1}^{(d)},\cdots, \left(\prod_{n=0}^{n_{\ell-1}-1}(t-\theta^{q^{\ell n +\ell-1}})\right) \cdot m_{(\ell-1)_1}^{(d)}\Big) \\
& = \tau^d \cdot (m_0,...,m_{\ell-1}) \quad \big(\in A_{v}[\![T_v,T_v^{-1}\}\!\}[\varpi(T_v)^{-1}] \otimes_{\FF_{q^{d}}[t]}\overline{\eR_{\tau}(\eC_{\ell})}\big).
\end{align*}
Therefore the constructed isomorphism in \eqref{eqn: comp-iso-2} indeed preserves the $\tau^d$-action on both sides.
\hfill $\Box$
\end{Subsubsec}

${}$

From the embedding $\FF_{q^d}(\!(T_v)\!)\hookrightarrow k_{v}[\![T_v-v]\!]$ sending $T_v$ to $v+(T_v-v)$, we view $\FF_{q^d}(\!(T_v)\!)$ as a subfield of 
$k_{v}[\![T_v-v]\!]$, and set
\begin{align*}
H_{\mathrm{cris}}\big(\eR_{\tau}(C_{\ell}), k_{v}[\![T_v-v]\!] \big)
& := k_{v}[\![T_v-v]\!] \otimes_{\FF_{q^d}(\!(T_v)\!)} H_{\mathrm{cris}}\big(\eR_{\tau}(\eC_{\ell}), \FF_{q^{d}}(\!(T_v)\!)\big) \\
& = k_{v}[\![T_v-v]\!] 
\otimes_{\FF_{q^{d}}[t]}\Big(\tau^d \cdot \overline{\eR_{\tau}(\eC_{\ell})}\Big).
\end{align*}
On the other hand,
recall that in \eqref{eqn: H-dR} we set
\begin{align*}
H_{\textrm{dR}}
\big(\eR_{\tau}(\eC_{\ell}),k_{v}[\![T_v-v]\!]\big) & := k_{v}[\![T_v-v]\!] \otimes_{k[t]} \langle \tau^d \eR_{\tau}(\eC_{\ell})\rangle_k.
\end{align*}
Regarding 
$A_{v}[\![T_v,T_v^{-1}\}\!\}[\varpi^{(d)}(T_v)^{-1}]$ as a subring of $k_{v}[\![T_v-v]\!]$ (see \cite[(3.5.7)]{HK20}),
Lemma~\ref{lem: comp-iso-1} induces the following crystalline--de Rham comparison isomorphism (see \cite[Theorem~3.5.18]{HJ20}):
\begin{theorem}\label{prop: dR-cris-iso} Let notation be given as above. Then we have the following isomorphism as $k_{v}[\![T_v-v]\!]$-modules:
\begin{equation*}
\Phi_v: 
H_{\text{\rm cris}}\big(\eR_{\tau}(\eC_{\ell}),k_{v}[\![T_v-v]\!]\big)
\cong 
H_{\text{\rm dR}}\big(\eR_{\tau}(\eC_{\ell}),k_{v}[\![T_v-v]\!]\big) .
\end{equation*}    
\end{theorem}

\begin{remark}\label{rem: cris-dR-kv}
Put
\begin{align*}
H_{\textrm{cris}}\big(\eR_{\tau}(\eC_{\ell}),k_{v}\big) &:= 
k_{v} \underset{\theta \mapsfrom t,\ \FF_{q^d}[t]}{\otimes}
\Big(\tau^d \cdot \overline{\eR_{\tau}(\eC_{\ell})}\Big) \\
&= k_{v} \underset{v \mapsfrom T_v,\ \FF_{q^{d}}(\!(T_v)\!)}{\otimes} H_{\textrm{cris}}\big(\eR_{\tau}(\eC_{\ell}),\FF_{q^{d}}(\!(T_v)\!) \big) \\
& \cong
\displaystyle
\frac{H_{\textrm{cris}}\big(\eR_{\tau}(\eC_{\ell}),k_{v}[\![T_v-v]\!]\big)}{(T_v-v)\cdot H_{\textrm{cris}}\big(\eR_{\tau}(\eC_{\ell}),k_{v}[\![T_v-v]\!]\big)} .
\end{align*}
As \eqref{eqn: dR-iso-2} gives us
\[
\frac{H_{\textrm{dR}}\big(\eR_{\tau}(\eC_{\ell}),k_{v}[\![T_v-v]\!]\big)}{(T_v-v)\cdot H_{\textrm{dR}}\big(\eR_{\tau}(\eC_{\ell}),k_{v}[\![T_v-v]\!]\big)} \cong H_{\textrm{dR}}\big(\eR_\tau(\eC_\ell),k_v\big),
\]
the isomorphism $\Phi_v$ in Theorem~\ref{prop: dR-cris-iso} induces a $k_{v}$-linear isomorphism
\begin{equation}\label{eqn: period-iso}
\phi_v: H_{\mathrm{cris}}\big(\eR_{\tau}(\eC_{\ell}),k_{v}\big) \cong 
H_{\mathrm{dR}}\big(\eR_{\tau}(\eC_{\ell}),k_{v}\big).
\end{equation}
\end{remark}

\subsection{Crystalline--de~Rham period matrix}\label{sec: CD matrix}

We will choose suitable ``algebraic bases'' and determine the crystalline--de~Rham period matrix associated to the isomorphism $\phi_v$ in \eqref{eqn: period-iso}.

First, 
for each integer $i$ with $0\leq i <\ell$, recall that $\Omega_{\ell,v,i}(T_v) \in A_v[\![T_v,T_v^{-1}\}\!\}$ is introduced in Lemma~\ref{lem: Omega-unit}. 
Recall the notation given in Definition~\ref{Def:ni and ni'}, $n_{i}$ is the smallest non-negative integer so that $i+n_{i}\ell \geq d$, and $i_1 = i+n_{i}\ell-d$.
In the case of $n_i>0$, we have
\[
\Big(\prod_{n=0}^{n_i-1}(t-\theta^{q^{\ell n + i}})\Big) \cdot \Omega_{\ell,v,i_1}^{(d)} = \Big(\prod_{n=0}^{n_i-1}(t-\varepsilon^{\ell n+i})\Big)\cdot \Omega_{\ell,v,i}.
\]
Moreover, by Lemma~\ref{lem: tau-d} we have the following rank-$\ell$ module
\[
\langle \tau^d \eR_{\tau}(\eC_{\ell})\rangle_k
= 
\prod_{i=0}^{\ell-1} \big(\prod_{n=0}^{n_{i}-1}(t-\theta^{q^{i+n \ell}})\big)\cdot k[t],\]
and the same arguments show that
\[
\tau^d \overline{\text{$\eR$}_{\tau}(\eC_{\ell})}
= \prod_{i=0}^{\ell-1} \big(\prod_{n=0}^{n_{i}-1}(t-\varepsilon^{q^{i+n \ell}})\big)\cdot \FF_{q^d}
[t].
\]
It follows that
\begin{align*}
H_{\mathrm{cris}}\big(\eR_{\tau}(\eC_{\ell}),k_{v}[\![T_v-v]\!]\big)
&= k_v[\![T_v-v]\!] \otimes_{\FF_{q^d}[t]}\Big(\tau^d \overline{\eR_{\tau}(\eC_{\ell})}\Big) \\
&= \prod_{i=0}^{\ell-1}\big(\prod_{n=0}^{n_{i}-1}(t-\varepsilon^{q^{i+n \ell}})\big)\cdot k_v[\![T_v-v]\!]
\end{align*}
and
\begin{align*}
H_{\text{\rm dR}}\big(\eR_{\tau}(\eC_{\ell}),k_{v}[\![T_v-v]\!]\big) & = k_v[\![T_v-v]\!] \otimes_{k[t]}\langle \tau^d \eR_{\tau}(\eC_{\ell})\rangle_k \\
&= \prod_{i=0}^{\ell-1}\big(\prod_{n=0}^{n_{i}-1}(t-\theta^{q^{i+n \ell}})\big)\cdot k_v[\![T_v-v]\!].
\end{align*}
The induced crystalline--de~Rham comparison isomorphism $\Phi_v$ in Theorem~\ref{prop: dR-cris-iso}
is given explicitly as follows:
\begin{align}\label{eqn: concrete-Phi-v}
& \Phi_v\Big(\big(\prod_{n=0}^{n_{0}-1}(t-\varepsilon^{q^{n \ell}})\big)m_0, \cdots, \big(\prod_{n=0}^{n_{\ell-1}-1}(t-\varepsilon^{q^{\ell-1+n \ell}})\big)m_{\ell-1}\Big) \notag \\
  & \hspace{1cm} = \Big( \Omega_{\ell,v,0} \cdot \big(\prod_{n=0}^{n_{0}-1}(t-\varepsilon^{q^{n \ell}})\big)m_0, \cdots, \Omega_{\ell,v,\ell-1} \cdot \big(\prod_{n=0}^{n_{\ell-1}-1}(t-\varepsilon^{q^{\ell-1+n \ell}})\big)m_{\ell-1}\Big) \notag \\
&\hspace{1cm} = \Big(\Phi_{v,0}\cdot \big(\prod_{n=0}^{n_{0}-1}(t-\theta^{q^{n \ell}})\big) m_0,\cdots,\Phi_{v,\ell-1} \cdot \big(\prod_{n=0}^{n_{\ell-1}-1}(t-\theta^{q^{\ell-1+n \ell}})\big)m_{\ell-1}\Big), 
\end{align}
where $\Phi_{v,i}(T_v)$ are given by infinite products
\[
\Phi_{v,i} = 
\left(
\prod_{n=0}^{n_{i}-1}\frac{t-\varepsilon^{q^{n \ell + i}}}{t-\theta^{q^{n \ell+i}}}
\right) \cdot \Omega_{\ell,v,i}(T_v)=
\prod_{n=n_i}^{\infty}(1- \frac{\theta_v^{q^{\ell n + i}}}{t-\varepsilon^{q^{\ell n + i}}}), \quad 0\leq i <\ell.
\]
In particular, by \eqref{eqn: evaluating at v}, evaluating $\Phi_{v,i}(T_v)$ at $T_{v}=v$ equals
\[
\Phi_{v,i}(T_v)\big|_{T_v=v}
= \left(\prod_{n=n_i}^{\infty}(1- \frac{\theta_v^{q^{\ell n + i}}}{t-\varepsilon^{q^{\ell n + i}}})\right)\Bigg|_{t=\theta}
= \prod_{n=n_i}^{\infty}(1- \frac{\theta_v^{q^{\ell n + i}}}{\theta-\varepsilon^{q^{\ell n + i}}}), \quad 0\leq i <\ell.
\]

Via the identification~\eqref{eqn: Rt-id-1}, we
take $\mathfrak{B}=\{b_0,...,b_{\ell-1}\}$
to be the standard $\FF_{q}[\theta][t]$ basis of $\eR_\tau(\eC_\ell)$, i.e.\ for $0\leq i<\ell$,
\begin{equation}\label{eqn: basis-B}
b_i = (b_{i,0},...,b_{i,\ell-1}) \quad  \in \eR_\tau(\eC_\ell)
\quad \text{ with }\quad 
b_{i,j} = \begin{cases}
1, & \text{ if $i=j$,}\\
0, & \text{ otherwise.}
\end{cases}
\end{equation}
For every integer $s$ with $0\leq s<\ell$, recall that $n_s$ denotes the smallest non-negative integer with $s+n_s\cdot \ell\geq d$ and $s_1= s+n_s \cdot \ell-d$.
Then by Lemma~\ref{lem: tau-d} we get
\begin{equation}\label{eqn: basis-bv}
b_{v,s}:=\tau^d \cdot b_{s_1} = (b_{s,0}',...,b_{s,\ell-1}')
\quad \text{ where }
\quad 
b_{s,j}' = 
\begin{cases}
\displaystyle\prod_{n=0}^{n_s-1}(t-\theta^{q^{\ell n+s}}), & \text{ if $j=s$,}\\
0, & \text{ otherwise.}
\end{cases}
\end{equation}
Recall in Lemma~\ref{lem: dR-com} that
the inclusion $\tau^d\eR_\tau(\eC_\ell)\subset \tau \eR_\tau(\eC_\ell)$
induces a surjective map
\[
\langle \tau^d\eR_\tau(\eC_\ell)\rangle_k
\twoheadrightarrow \frac{\langle \tau \eR_\tau(\eC_\ell)\rangle_k}{(t-\theta) \cdot \langle \tau \eR_\tau(\eC_\ell)\rangle_k} =
H_{\textrm{dR}}\big(\eR_{\tau}(\eC_{\ell}),k\big).
\]
Set 
\begin{equation}\label{eqn: beta-v}
\omega_{v,s}:= b_{v,s} \mod (t-\theta) \quad \in \frac{\langle \tau^d\eR_\tau(\eC_\ell)\rangle_k}{(t-\theta)\cdot \langle \tau^d\eR_\tau(\eC_\ell)\rangle_k} \cong H_{\textrm{dR}}\big(\eR_\tau(\eC_\ell),k\big).
\end{equation}
Then $\beta_v:= \{\omega_{v,0},...,\omega_{v,\ell-1}\}$ is a $k$-basis of $H_{\textrm{dR}}\big(\eR_\tau(\eC_\ell),k\big)$.
On the other hand, for each $0\leq s<\ell$, put
\begin{equation}\label{eqn: lambda-v}
\lambda_{v,s}:= b_{v,s} \mod \fm_v \quad \in \tau^d \cdot \overline{\eR_\tau(\eC_\ell)}.
\end{equation}
Then $\Lambda_v :=\{\lambda_{v,0},...,\lambda_{v,\ell-1}\}$
gives a $k_v$-basis of $H_{\textrm{cris}}\big(\eR_\tau(\eC_\ell),k_v\big)$ 
through the identification in Remark~\ref{rem: cris-dR-kv}.
From the description of $\Phi_v$ in \eqref{eqn: concrete-Phi-v},
we then obtain that:

\begin{theorem}\label{thm: period-matrix}
For every integer $s$ with $0\leq s<\ell$, we have
\[
\phi_v(\lambda_{v,s}) = \prod_{n=n_s}^\infty (1-\frac{(\theta-\varepsilon)^{q^{\ell n + s}}}{\theta-\varepsilon^{q^{\ell n + s}}}) \cdot \omega_{v,s}.
\]    
Here $n_s$ is the smallest non-negative integer so that $s+n_s\ell \geq d$.
Consequently,
the period matrix of isomorphism $\phi_v$ in \eqref{eqn: period-iso} with respect to $\beta_v$ and $\Lambda_v$ is 
\[
\begin{pmatrix}
\displaystyle\prod_{n=n_0}^{\infty}(1- \frac{\theta_v^{q^{\ell n }}}{\theta-\varepsilon^{q^{\ell n }}}) & & 0\\
& \ddots & \\
0& & \displaystyle\prod_{n=n_{\ell-1}}^{\infty}(1- \frac{\theta_v^{q^{\ell n + \ell-1}}}{\theta-\varepsilon^{q^{\ell n + \ell-1}}})
\end{pmatrix}.
\]
\end{theorem}

\subsection{Chowla--Selberg-type formula}\label{sec: vCSF}

Note that the action of $\tau_v:=\tau^d$ on the $T_v$-isocrystal
$\big(H_{\mathrm{cris}}\big(\text{$\cR_{\tau}$}(\eC_{\ell}),\FF_{q^d}(\!(T_v)\!)\big),\tau_v\big)$ introduced in Remark~\ref{rem: isocrystal} is actually $\FF_{q^d}(\!(T_v)\!)$-linear.
Through the isomorphism
\begin{align*}
k_{v} \underset{v \mapsfrom T_v, \FF_{q^{d}}(\!(T_v)\!)}{\otimes} H_{\mathrm{cris}}\big(\eR_{\tau}(\eC_{\ell}), \FF_{q^{d}}(\!(T_v)\!)\big) & = 
H_{\mathrm{cris}}\big(\text{$\cR_{\tau}$}(\eC_{\ell}),k_{v}\big) \stackrel{\phi_v}{\cong} 
H_{\mathrm{dR}}\big(\eR_{\tau}(\eC_{\ell}),k_{v}\big),
\end{align*}
there exists a $k_v$-linear automorphism $\varrho_v$ on $H_{\textrm{dR}}\big(\eR_{\tau}(\eC_{\ell}),k_{v}\big)$
satisfying 
\begin{equation}\label{eqn: rho-v}
\varrho_v \cdot \omega = \phi_v\Big(\tau_v\cdot \phi_v^{-1}(\omega)\Big), \quad  \forall \omega \in H_{\text{$\textrm{dR}$}}\big(\cR_{\tau}(\eC_{\ell}),k_{v}\big).
\end{equation}
That is, $\varrho_v$ is the existence of $k_v$-linear automorphism so that the following diagram commutes:
\[
\xymatrix{
 H_{\mathrm{cris}}\big(\eR_{\tau}(\eC_{\ell}) , k_v\big)\ar[d]^{\tau_v} \ar[r]^{\phi_v} & H_{\textrm{dR}}\big(\eR_{\tau}(\eC_{\ell}),k_{v}\big)\ar[d]^{\varrho_v}\\
H_{\mathrm{cris}}\big(\eR_{\tau}(\eC_{\ell}), k_v \big)\ar[r]^{\phi_v} & H_{\textrm{dR}}\big(\eR_{\tau}(\eC_{\ell}),k_{v}\big)}
.\]
From Theorem~\ref{thm: period-matrix}, we are able to realize the $\varrho_v$-action on $H_{\textrm{dR}}\big(\eR_{\tau}(\eC_{\ell}),k_{v}\big)$ via $v$-adic arithmetic gamma values in the following.\\

For each integer $s$ with $0\leq s < \ell$, 
recall that $n_s$ is the smallest non-negative integer so that $s+n_s\ell \geq d$, and we write
\[
s+d = s_0 + n_s' \ell \ (\geq d), \quad \text{ where } s_0, n_s' \in \ZZ  \text{ and } 0\leq s_0<\ell.
\]
From the inequalities
\[
\ell+d>s+d = s_0+n_s' \ell \geq s_0 +n_{s_0}\ell \geq d,
\]
we get 
\begin{equation}\label{E:n_s'=n_s0}
n_s' = n_{s_0}.
\end{equation} By using argument similar to Lemma~\ref{lem: tau-d}, we have
\[
\tau^d \cdot \lambda_{v,s} = 
\left(\prod_{n=0}^{n_{s}-1}(t-\varepsilon^{q^{s+n \ell}})\right)\cdot \lambda_{v,s_0},
\]
and hence Theorem~\ref{thm: period-matrix} implies the following identity: 
\begin{align}\label{eqn: Fro-act}
\varrho_v \cdot \omega_{v,s} &=
\left(\displaystyle \frac{\displaystyle \prod_{n=n_{s_0}}^{\infty}(1- \displaystyle \frac{\theta_v^{q^{\ell n + s_0}}}{\theta-\varepsilon^{q^{\ell n + s_0}}})}{\displaystyle \prod_{n=n_s}^{\infty}(1- \displaystyle \frac{\theta_v^{q^{\ell n + s}}}{\theta-\varepsilon^{q^{\ell n + s}}})}\right) \cdot \left(\prod_{n=0}^{n_{s}-1}(\theta-\varepsilon^{q^{s+n \ell}})\right)\cdot \omega_{v,s_0}
\end{align}

Finally, recall from the definition that
\begin{equation*}\label{eqn: dR-iso}
H_{\textrm{dR}}\big(\cR_{\tau}(\eC_{\ell}),k\big) = \frac{\langle \tau \eR_\tau(\eC_{\ell})\rangle_k}{(t-\theta) \cdot \langle \tau \eR_\tau(\eC_{\ell})\rangle_k}.
\end{equation*}
By \eqref{eqn: Rt-id-1} we may identify
\[
\langle \tau \eR_\tau(\eC_{\ell})\rangle_k = (t-\theta)\cdot k[t] \times \prod_{i=1}^{\ell-1} k[t].
\]
For $0\leq s <\ell$, let 
$m_s = \big((t-\theta) m_{s,0},m_{s,1}...,m_{s,\ell-1}\big) \in \langle \tau \eR_\tau(\eC_{\ell})\rangle_k$, where
$m_{s,j} = 1$ if $s=j$ or $0$ otherwise, and set
\begin{equation}\label{eqn: global-omega}
\omega_s := m_s \mod (t-\theta) \cdot \langle \tau \eR_\tau(\eC_{\ell})\rangle_k  \quad \in
\frac{\langle \tau \eR_\tau(\eC_{\ell})\rangle_k}{(t-\theta) \cdot \langle \tau \eR_\tau(\eC_{\ell})\rangle_k}
= H_{\text{dR}}\big(\eR_{\tau}(\eC_{\ell}),k\big).
\end{equation}
Then the description of $b_{v,s}$ in \eqref{eqn: basis-bv} says in particular that
\[
b_{v,s} = \begin{cases}    
\left(
\displaystyle\prod_{n=1}^{n_s-1}(t-\theta^{q^{s+n\ell}})
\right)
\cdot m_{s}, & \text{ if $s=0$,} \\ 
\left(
\displaystyle\prod_{n=0}^{n_s-1}(t-\theta^{q^{s+n\ell}})
\right)
\cdot m_{s}, & \text{ if $s>0$.}
\end{cases}
\] 
for which we get
\begin{equation}\label{eqn: dR-basis-2}
\omega_{v,s} = \begin{cases}    
\left(
\displaystyle\prod_{n=1}^{n_s-1}(\theta-\theta^{q^{s+n\ell}})
\right)
\cdot \omega_{s}, & \text{ if $s=0$,} \\ 
\left(
\displaystyle\prod_{n=0}^{n_s-1}(\theta-\theta^{q^{s+n\ell}})
\right)
\cdot \omega_{s}, & \text{ if $s>0$.}
\end{cases}
\end{equation}
By \eqref{eqn: Fro-act}, we obtain that
\begin{align}
\varrho_v \cdot \omega_s &=
\left(\displaystyle \frac{\displaystyle \prod_{n=n_{s_0}}^{\infty}(1- \displaystyle \frac{\theta_v^{q^{\ell n + s_0}}}{\theta-\varepsilon^{q^{\ell n + s_0}}})}{\displaystyle \prod_{n=1}^{\infty}(1- \displaystyle \frac{\theta_v^{q^{\ell n + s}}}{\theta-\varepsilon^{q^{\ell n + s}}})}\right) \cdot \alpha_{s}^{-1} \cdot \omega_{v,s_0} \notag \\
&=
\left(\displaystyle \frac{\displaystyle \prod_{n=1}^{\infty}(1- \displaystyle \frac{\theta_v^{q^{\ell n + s_0}}}{\theta-\varepsilon^{q^{\ell n + s_0}}})}{\displaystyle \prod_{n=1}^{\infty}(1- \displaystyle \frac{\theta_v^{q^{\ell n + s}}}{\theta-\varepsilon^{q^{\ell n + s}}})}\right) \cdot \left(\prod_{n=0}^{n_{s_0}-1}(\theta-\varepsilon^{q^{s_0+n \ell}})\right)\cdot \frac{\alpha_{s_0}}{\alpha_{s}} \cdot \omega_{s_0},
\end{align}
where $\alpha_s$ is given in \eqref{eqn: alpha-C}.
As $n_{s_0} = n_s'$ in~\eqref{E:n_s'=n_s0}, combining with Corollary~\ref{cor: agv-form} we arrive at the following $v$-adic Chowla--Selberg-type formula:

\begin{theorem}\label{thm: CSF-1}
Let $\omega_{0},...,\omega_{\ell-1}$ be the $k$-basis of $H_{\text{\rm dR}}\big(\text{$\cR$}_{\tau}(\eC_{\ell}),k\big)$ chosen in \eqref{eqn: dR-basis-2}.
Given an integer $s$ with $0\leq s<\ell$, write $s+d = s_0 + n_s' \ell$ for unique integers $s_0$ and $n_s'$ with $0\leq s_0<\ell$.
Then 
\[
\varrho_{v}\cdot \omega_s =  (-1)^{n_s'}\cdot C_{s}\cdot
\displaystyle 
\frac{\displaystyle\agv\Big(1-\Big\langle\frac{q^{s+d}}{q^{\ell}-1}\Big\rangle_{\ari}\Big)}{\displaystyle\agv\Big(1-\Big\langle\frac{q^{s+d-1}}{q^{\ell}-1}\Big\rangle_{\ari}\Big)^q}\cdot \omega_{s_0},
\quad \text{ where } \quad 
C_s:= 
\begin{cases}
-v , & \text{ if $s=0$,}\\
1, & \text{ if $s>0$,}
\end{cases}
\]
and $\langle z\rangle_{\ari}$ is the fractional part of every $z \in \QQ$, i.e.\ $0\leq \langle z \rangle_{\ari} <1$ and $z-\langle z \rangle_{\ari} \in \ZZ$.
\end{theorem}

\begin{remark}
Theorem~\ref{thm: CSF-1} indicates that the matrix presentation of the automorphism $\varrho_v$ on $H_{\text{\rm dR}}\big(\text{$\eR$}_{\tau}(\eC_{\ell}),k_v\big)$ with respect to the basis $\beta:=\{\omega_0,...,\omega_{\ell-1}\}$ can be concretely illustrated in terms of $v$-adic arithmetic gamma values.
\end{remark}

\begin{remark}
The basis $\beta = \{\omega_0,...,\omega_{\ell-1}\}$ of the de~Rham module $H_{\textrm{dR}}\big(\eR_{\tau}(\eC_{\ell}),k\big)$ is actually independent of the chosen finite place $v$ of $k$.
Hence Theorem~\ref{thm: CSF-1} shows in particular that with respect to this global basis $\beta$, we are able to realize the ``isocrystal'' action $\varrho_v$ simultaneously for every finite place $v$ of $k$.
In fact, this basis is also utilized implicitly in the $\infty$-adic Chowla--Selberg formula in \cite{CPTY10} (see also \cite{Wei22}), and we would have a natural consistency between these formulas (as in \cite[Theorem~3.15]{O89} for the number field case).
\end{remark}

\section{Algebraic relations among the \texorpdfstring{$v$}{v}-adic arithmetic gamma values}
\label{sec: tran-agv}

In the remaining sections of this paper, we will determine the algebraic relations among the $v$-adic arithmetic gamma values.

\subsection{Transcendence result of the \texorpdfstring{$v$}{v}-adic arithmetic gamma values}\label{sec: AI-agv} 

Given $\ell \in \NN$, recall in Lemma~\ref{lem: Omega-unit} that we have defined
\[
\Omega_{\ell,v,i}(T_v) = \prod_{n=0}^\infty (1-\frac{\theta_v^{q^{\ell n+i}}}{t-\varepsilon^{q^{\ell n +i}}}) \quad \in A_v[\![T_v,T_v^{-1}\}\!\}, \quad \forall i\in \ZZ_{\geq 0}.
\]
In particular, one observes that
\begin{align}\label{E:Omega ell,v,i}
\Omega_{\ell,v,i}(v):= \Omega_{\ell,v,i}(T_v)\big|_{T_v=v} & =
\prod_{n=0}^{\infty}(1-\frac{\theta_v^{q^{\ell n + i}}}{\theta-\varepsilon^{q^{\ell n+i}}}) \quad \in k_v^\times \\
& =
(1-\frac{\theta_v^{q^{ i}}}{\theta-\varepsilon^{q^{i}}}) 
\cdot 
\Omega_{\ell,v,i+\ell}(v),
\quad \forall i \in \NN. \notag
\end{align}
We first assert that:

\begin{theorem}\label{thm: AI-Omega}
Given $\ell \in \NN$,
the values 
\[
\Omega_{\ell,v,i}(v) = \prod_{n=0}^{\infty}(1-\frac{\theta_v^{q^{\ell n + i}}}{\theta-\varepsilon^{q^{\ell n+i}}}), \quad 0< i \leq \ell,    
\]
are algebraically independent over $\bar{k}$.
\end{theorem}

\begin{remark}
By Theorem~\ref{thm: period-matrix}, the values $\Omega_{\ell,v,i}(v)$, $0< i \leq \ell$ are (up to $k^\times$-multiples) the $v$-adic periods of the crystalline--de Rham comparison isomorphism for the $t$-motive $\eR_{\tau}(\eC_{\ell})$. Theorem~\ref{thm: AI-Omega} implies that
\[
\trdeg_{\bar{k}}\bar{k}\Big(
\Omega_{\ell,v,i}(v)\ \Big|\ 0< i \leq \ell\Big) = \ell = \dim \Gamma_{\eR_{\tau}(\eC_{\ell})},
\]
where $\Gamma_{\eR_{\tau}(\eC_{\ell})}$ is the ``$t$-motivic Galois group of the $t$-motive $\bar{k}\otimes_{\FF_q[\theta]}\eR_{\tau}(\eC_{\ell})$'' (see \cite[Lemma~3.2.1]{CPTY10}).
This coincidence seems to provide an evidence for a suitable $v$-adic analogue of Papanikolas' theorem in \cite{P08} (Grothendieck period conjecture).
\end{remark}

Assuming Theorem~\ref{thm: AI-Omega}, whose proof  will be given in Section~\ref{sec: pf-AL-Omega},
we are able to show that:

\begin{theorem}\label{thm: AL-gamma}
Given $\ell \in \NN$, we have the following equality
\[
\trdeg_{\bar{k}} \bar{k}\Big(\agv(z)\ \Big|\ z \in \QQ \text{ with } (q^\ell-1)\cdot z \in \ZZ \Big) = \ell - \gcd(\ell,\deg v).
\]
\end{theorem}

\begin{proof}
First, assume that $d = \deg v \mid \ell$, and write $\ell = d\ell_0$.
If $\ell_0=1$, then the result follows from Theorem~\ref{thm: GK-formula}.
Suppose $\ell_0\geq 2$.
For every integer $s$ with $0\leq s < d\ell_0$, by Lemma~\ref{lem: agv-form} we get
\[
\displaystyle\frac{\agv(1-\displaystyle\frac{q^{d\ell_0-s}}{q^{d\ell_0}-1})}{\agv(1-\displaystyle\frac{q^{d\ell_0-s-1}}{q^{d\ell_0}-1})^q}
\sim 
\prod_{n=2}^\infty
\left(
\frac{
\displaystyle
1-\frac{\theta_v^{q^{d\ell_0 n-s}}}{\theta-\varepsilon^{q^{-s}}}
}{
\displaystyle
1-\frac{\theta_v^{q^{d\ell_0 n-s-d}}}{\theta-\varepsilon^{q^{-s}}}
}
\right)
=
\frac{\Omega_{d\ell_0,v,2d\ell_0-s}(v)}{\Omega_{d\ell_0,v,2d\ell_0-s-d}(v)}
=:u_s.
\]
In particular, for every integer $s$ with $0\leq s < d(\ell_0-1)$, 
let $s_d \in \ZZ$ with $0\leq s_d < d$ and $s_d \equiv s \bmod d$.
Write
$s = s_d + d \ell_s$ with $0\leq \ell_s <\ell_0-1$.
Then one has that
\[
u'_s := u_s u_{s+d} \cdots u_{s+(\ell_0-2-\ell_s)d} = \frac{\Omega_{d\ell_0,v,2d\ell_0-s}(v)}{\Omega_{d\ell_0,v,d\ell_0+d-s_d}(v)} \sim \frac{\Omega_{d\ell_0,v,d\ell_0-s}(v)}{\Omega_{d\ell_0,v,d-s_d}(v)}.
\]
As $d<d\ell_0 -s \leq d\ell_0$ and $0 <d- s_d\leq d$, Theorem~\ref{thm: AI-Omega} ensures the algebraic independence of $u_s'$ over $\bar{k}$ for $0\leq s < d(\ell_0-1)$.
Hence
\[
\trdeg_{\bar{k}} \bar{k}\Big(\agv(z)\ \Big|\ z \in \QQ \text{ with } (q^{\ell}-1)\cdot z \in \ZZ \Big) \geq d(\ell_0 - 1) = \ell- \gcd(\ell,\deg v).
\]
By combining with the upper bound in Theorem~\ref{thm: up-bd}, the result follows.

Consequently, the following monomial relations
\begin{align}\label{E:generator relations}
\prod_{m=0}^{\ell_0-1} \agv(1-\frac{q^{j+dm}}{q^{d\ell_0}-1}) & = \agv(1-\frac{q^j+q^{j+d}+\cdots + q^{j+d(\ell_0-1)}}{q^{d\ell_0}-1}) \\
&= \agv(1-\frac{q^{j}}{q^d-1}), \quad \text{ for $0\leq j < d$,} \notag
\end{align}   
generate all algebraic relations among $\agv(1-q^s/(q^{d\ell_0}-1))$ for $0\leq s< d\ell_0-1$.
In particular,
let $\vec{\delta}_{0},...,\vec{\delta}_{d-1} \in \ZZ^{d\ell_0}$ 
be the ``exponent vectors'' corresponding to the above monomial relations, i.e.\ 
\[
\vec{\delta}_j = (\delta_{j,0},...,\delta_{j,d\ell_0-1}) \quad \text{ with } \quad 
\delta_{j,s} = \begin{cases}
1, & \text{ if $j\equiv s \bmod d$,}\\
0, & \text{ otherwise.}
\end{cases}
\]
Then 
\[
\prod_{s=0}^{d\ell_0-1}\agv(1-\frac{q^s}{q^{d\ell_0}-1})^{\delta_{j,s}} = \agv(1-\frac{q^j}{q^d-1}) \quad \in \bar{k},
\]
Moreover, as $\vec{\delta}_{0},...,\vec{\delta}_{d-1} \in \ZZ^{d\ell_0}$
span a $d$-dimensional ``relation subspace'' in $\QQ^{d\ell_0}$,
for arbitrary
given vectors $\vec{n}_1,...,\vec{n}_r \in \ZZ^{d\ell_0}$ with $\vec{n}_j =(n_{j,0},...,n_{j,d\ell-1})$, $j=1,...,r$,
we have that
\begin{align} \label{eqn: trdeg-1}
&\ \  \trdeg_{\bar{k}}
\bar{k}\left(\prod_{0\leq s<d\ell} \agv(1-\frac{q^s}{q^{d\ell_0}-1})^{n_{j,s}}\ \Bigg|\ 1\leq j \leq r\right) \notag \\
= & \ \ \dim_{\QQ}\langle \vec{\delta}_0,...,\vec{\delta}_{d-1}, \vec{n}_1,...,\vec{n}_r\rangle_{\QQ}-d.
\end{align}
Here 
$\langle \vec{\delta}_0,...,\vec{\delta}_{d-1}, \vec{n}_1,...,\vec{n}_r\rangle_{\QQ}$ is the subspace of $\QQ^{d\ell_0}$ spanned by 
$\vec{\delta}_0,...,\vec{\delta}_{d-1}, \vec{n}_1,...,\vec{n}_r$ over $\QQ$.
\\

In the general case, for a given $\ell \in \NN$,
write $\text{lcm}(\ell,d) = d\ell_0 = d_0 \ell$ with $\ell_0, d_0 \in \NN$ so that $\text{gcd}(\ell_0,d_0) = 1$.
Note that for an integer $j$ with $0\leq j<\ell$, one has
\[
\agv(1-\frac{q^{j}}{q^\ell-1}) = 
\prod_{s=0}^{d_0-1} \agv(1-\frac{q^{j+s\ell}}{q^{d_0 \ell}-1}).
\]
Let
$\vec{\lambda}_0,...,\vec{\lambda}_{\ell-1} \in \ZZ^{d\ell_0}$ satisfying that
\[
\vec{\lambda}_j = (\lambda_{j,0},...,\lambda_{j,d\ell_0-1}) \quad \text{ with } \quad 
\lambda_{j,s} = \begin{cases}
1, & \text{ if $j\equiv s \bmod \ell$,}\\
0, & \text{ otherwise.}
\end{cases}
\]
As for $0\leq j<\ell$,
\[
\prod_{s=0}^{d\ell_0-1}\agv(1-\frac{q^s}{q^{d\ell_0}-1})^{\lambda_{j,s}} = \agv(1-\frac{q^{j}}{q^\ell-1}),
\]
by the equality~\eqref{eqn: trdeg-1}, we obtain that
\begin{align}\label{E:dim delta}
& \ \trdeg_{\bar{k}} \bar{k}\Big(\agv(z)\ \Big|\ z \in \QQ \text{ with } (q^\ell-1)\cdot z \in \ZZ \Big) \notag \\
=&\ 
\trdeg_{\bar{k}}
\bar{k}\bigg(\agv(1-\frac{q^{j}}{q^\ell-1})\ \bigg|\ 
0\leq j<\ell\bigg) \\
=& \ 
\dim_\QQ \langle \vec{\delta}_0,...,\vec{\delta}_{d-1}, \vec{\lambda}_0,...,\vec{\lambda}_{\ell-1}\rangle_{\QQ}-d. \notag 
\end{align}
Since
\[
\langle
\vec{\delta}_0,...,\vec{\delta}_{d-1}\rangle_{\QQ} =\Big\{ (\alpha_0,...,\alpha_{d\ell_0-1}) \in \QQ^{d\ell_0} \ \Big|\ \alpha_i = \alpha_j \text{ if } i\equiv j \bmod d\Big\}
\]
and 
\[
\langle
\vec{\lambda}_0,...,\vec{\lambda}_{\ell-1}
\rangle_{\QQ} =\Big\{(\beta_0,...,\beta_{d_0\ell_0-1}) \in \QQ^{d_0\ell} \ \Big|\ \beta_i = \beta_j \text{ if } i\equiv j \bmod \ell\Big\},
\]
one has that
\begin{align*}
&\langle
\vec{\delta}_0,...,\vec{\delta}_{d-1}\rangle_{\QQ} \cap 
\langle
\vec{\lambda}_0,...,\vec{\lambda}_{\ell-1}
\rangle_{\QQ} \\
& \hspace{2cm} = 
\Big\{( \gamma_0,...,\gamma_{d_0\ell_0-1})\in \QQ^{d_0\ell} \ \Big|\ \gamma_i = \gamma_j \text{ if } i\equiv j \bmod \gcd(d,\ell)\Big\}.
\end{align*}
Therefore 
\[
\dim_\QQ \langle \vec{\delta}_0,...,\vec{\delta}_{d-1}, \vec{\lambda}_0,...,\vec{\lambda}_{\ell-1}\rangle_{\QQ} = \ell+d - \gcd(d,\ell),
\]
for which the result follows from~\eqref{E:dim delta}. 
\end{proof}

\begin{remark}
Theorem~\ref{thm: AL-gamma} says that all algebraic relations among $v$-adic arithmetic gamma values at rational $p$-adic integers are generated by the funcional equations in Proposition~\ref{prop: FE} and Thakur's analogue of the Gross--Koblitz formula in Theorem~\ref{thm: GK-formula}.
\end{remark}

One immediate consequence of Theorem~\ref{thm: AL-gamma} is the following.

\begin{corollary}\label{cor: tran-gamma}
Let $a$ and $b$ be nonzero integers so that  $\gcd(a,b)=1$ and $p \nmid b$. If  $q^d \not\equiv 1 \bmod b$,
then
\[
\agv(\frac{a}{b}) \quad \text{ is transcendental over $k$.}
\]
\end{corollary}

\begin{proof}
Without loss of generality, we may assume that $\frac{a}{b} = 1-\frac{a'}{b}$ with $0\leq a'<b$ and $\gcd(a',b)=1$.
Let $\ell_0$ be the order of $q^d \bmod b$ in $(\ZZ/b \ZZ)^\times$.
Then the assumption of $b$ implies that $\ell_0>1$.
Write $q^{d\ell_0}-1 = b \cdot b'$, and take $a_0,...,a_{d\ell_0-1} \in \ZZ$ with $0\leq a_0,...,a_{d\ell_0-1}<q$ satisfying that
\[
\frac{a'}{b} = \frac{a'b'}{q^{d\ell_0}-1} = \sum_{i=0}^{d\ell_0-1}\frac{a_i q^i}{q^{d\ell_0}-1}.
\]
Let $\vec{a} = (a_0,a_1,...,a_{d\ell_0-1}) \in \ZZ^{d\ell_0}$.
If $\vec{a} \in \langle \vec{\delta}_0,...,\vec{\delta}_{d-1}\rangle_\QQ$,
then we get $a_i =a_j$ if $i\equiv j \bmod d$.
Hence
\[
\frac{a'}{b} = \sum_{i=0}^{d-1} \frac{a_i(q^i+q^{d+i}+\cdots + q^{(\ell_0-1)d+i})}{q^{d\ell_0}-1}
=\frac{a_0+a_1q+\cdots +a_{d-1}q^{d-1}}{q^d-1},
\]
for which $b \mid q^d-1$, a contradiction.
Therefore $\vec{a} \notin \langle \vec{\delta}_0,...,\vec{\delta}_{d-1}\rangle_\QQ$,
and by \eqref{eqn: trdeg-1} we get
\[
\trdeg_{\bar{k}}\bar{k}\Big(\agv(1-\frac{a'}{b})\Big)
= \dim_{\QQ}\langle \vec{\delta}_0,...,\vec{\delta}_{d-1},\vec{a}\rangle_\QQ 
=(d+1)-d = 1
\]
as desired.
\end{proof}

\begin{remark}\label{Rem: Equiv}
Let $a$ and $b$ be nonzero integers so that   $\gcd(a,b)=1$ and $p \nmid b$. Combining Thakur's result in Theorem~\ref{thm: GK-formula} and Corollary~\ref{cor: tran-gamma}, we have the following equivalence:
\[
b\mid q^{\deg v}-1 \quad {\hbox{ if and only if }}\quad \agv(\frac{a}{b})\in \bar{k}^{\times}.
\]
\end{remark}

\subsection{\texorpdfstring{$v$}{v}-adic arithmetic gamma distribution}
\label{sec: agv-dist}
Let $\ZZ_{(p)}:= \ZZ_p \cap \QQ$, and $\Acal^{\text{ari}}$ be the free abelian group generated by all elements in $\ZZ_{(p)}/\ZZ$.
We identify $\Acal^{\text{ari}}$ with a subgroup of $\Acal^{\text{ari}}_\QQ := \QQ \otimes_\ZZ \Acal^{\text{ari}}$.
Every element in $\Acal^{\text{ari}}$
(resp.\ $\Acal^{\text{ari}}_\QQ$) can be written uniquely as a formal sum
\[
{\bf z} = \sum_{z \in \ZZ_{(p)}/\ZZ} n_z [z], \quad
\text{where $n_z \in \ZZ$ (resp.\ $\QQ$) and $n_z = 0$ for almost all $z$.}
\]
Given $\ell \in \NN$, let $\Acal^{\text{ari}}_\ell$ (resp.\ $\Acal^{\text{ari}}_{\ell,\QQ}$)
be the subgroup of $\Acal^{\text{ari}}$ (resp.\ subspace of $\Acal^{\text{ari}}_\QQ$) generated by $z \in \frac{1}{q^\ell-1}\ZZ/\ZZ$.
Also, recall that the \emph{arithmetic diamond bracket}
of $z \in \ZZ_{(p)}$, introduced by Thakur~\cite{T91}, comes from the fractional part of $z$, i.e.\ $\langle z\rangle_{\ari} \in \ZZ_{(p)}$ is the unique number with
\[
0\leq \langle z\rangle_{\ari} <1 
\quad \text{ and } \quad 
z \equiv \langle z\rangle_{\ari} \bmod \ZZ.
\]
We may regard $\langle \cdot \rangle_{\ari}$ as a function on $\ZZ_{(p)}/\ZZ$.
Consider the subspace $\Rcal^{\text{ari}}_\ell$ of $\Acal^{\text{ari}}_{\ell,\QQ}$ spanned by
\[
[z] - \sum_{i=0}^{\ell-1} z_i [\frac{q^i}{1-q^\ell}], \quad \forall z \in \frac{1}{q^\ell-1} \ZZ/\ZZ \quad \text{ (see \eqref{E:Standard FE})},
\]
where $z_0,...,z_{\ell-1} \in \ZZ$ with
$0\leq z_0,...,z_{\ell-1}<q$ so that
\[
\sum_{i=0}^{\ell-1}\frac{z_i q^i}{q^\ell-1} = 
\langle -z\rangle_{\ari} .
\]
Put 
$\Ucal^{\text{ari}} := \Acal^{\text{ari}}_\QQ/\Rcal^{\text{ari}}$, where
$\Rcal^{\text{ari}} := \cup_\ell \Rcal^{\text{ari}}_\ell$.
We call $\Ucal^{\text{ari}}$ the \emph{universal distribution associated to the arithmetic diamond bracket relations} (see \cite[\S 4.3]{Wei22}.

Following the $\infty$-adic case, we define
\[
\tilde{\Gamma}_{\text{ari},v}: \ZZ_{(p)}/\ZZ \longrightarrow \CC_v^\times, \quad 
\tilde{\Gamma}_{\text{ari},v}(z \bmod \ZZ) := \agv(1-\langle -z\rangle_{\ari}),
\]
which induces a $\QQ$-linear homomorphism
$\hat{\Gamma}_{\text{ari},v}: \Ucal_{\QQ}^{\text{ari}}\rightarrow \CC_v^\times/\bar{k}^\times$ (by the monomial relation \eqref{E:Standard FE}).
Recall in Remark~\ref{rem: relations} that having the monomial relation \eqref{E:Standard FE} together with the translation (2) of Proposition~\ref{prop: FE} (for $z \in \ZZ_p \cap \QQ$), the reflection (1) and multiplication (3) of Proposition~\ref{prop: FE} are valid for $z \in \ZZ_p \cap \QQ$.
We may call $\hat{\Gamma}_{\text{ari},v}$
the \emph{$v$-adic arithmetic gamma distribution}.
From Corollary~\ref{cor: GK-formula} and Theorem~\ref{thm: AL-gamma} together with~\eqref{E:generator relations},
we obtain that: 

\begin{theorem}
The kernel of the $v$-adic arithmetic gamma distribution $\hat{\Gamma}_{\text{ari},v}$ is spanned by the Gross--Koblitz--Thakur relation
\[
\sum_{i=0}^{r-1}[q^{i\deg v}z], \quad \forall z \in \frac{1}{n}\ZZ
\]
where $n$ runs through all positive integers with $p \nmid n$, and $r$ is the order of $q^{\deg v} \bmod n$ in $(\ZZ/n\ZZ)^\times$.
\end{theorem}

\section{Transcendence theory of Hartl--Juschka modules}\label{sec: tran-FM}

In this section, we briefly review the necessary properties in the trancendence theory of Hartl--Juschka modules for the proof of Theorem~\ref{thm: AI-Omega}.
To utilize this well-developed theory in the $v$-adic setting, our initial step is to regard $v$ as the infinite place of the ``projective $\frac{1}{v}$-line'' over $\FF_q$.

\subsection{Regarding \texorpdfstring{$v$}{v} as the infinite place}\label{sec: v to infinite}

Fix $d = \deg v$ and a positive integer $\ell$ as before.
Set $k_{d\ell}:= \FF_{q^{d\ell}}(\theta) = \FF_{q^{d\ell}}(\theta_v) \subset \bar{k}$, where $\theta_v = \theta-\varepsilon$, and $k_{d\ell,v}:= \FF_{q^{d\ell}}(\!(\theta_v)\!)$.
Let
\[
\tilde{\theta}_v:= (\theta-\varepsilon)^{-1}
\quad \text{ and } 
\tilde{v} := v^{-1}.
\]
Then we have
\[
\FF_q(\!(v)\!) =
\FF_q(\!(\tilde{v}^{-1})\!) \subset  \FF_{q^{d\ell}}(\!(\tilde{\theta}_v^{-1})\!) = \FF_{q^{d\ell}}(\!(\theta_v)\!) =k_{d\ell,v} \subset \CC_v.
\]

The trick of viewing $\CC_v$ as an \lq$\infty_{\tilde{v}}$-adic\rq\ field comes from the following translation. 
\begin{itemize}
\item[$\bullet$] $\FF_{q}(v)=\FF_q(\tilde{v})$ is the fraction field of $\FF_q[\tilde{v}]$, the polynomial ring  in   $\tilde{v}$ over $\FF_q$.
\item[$\bullet$] The $\infty_{\tilde{v}}$-adic valuation on $\FF_q(v)=\FF_q(\tilde{v})$ is given by ${\rm{ord}}_{\infty_{\tilde{v}}}:={\rm{ord}}_{\tilde{v}^{-1}}={\rm{ord}}_{v}$.
\item[$\bullet$]  $\FF_{q}(v)=\FF_q(\tilde{v})$ is the function  field of the projective $\tilde{v}$-line over $\FF_q$ with $v=\tilde{v}^{-1}$ as a uniformizer of the infinite place. 
\item[$\bullet$]   $\FF_q(\!(v)\!) =
\FF_q(\!(\tilde{v}^{-1})\!)$ is the completion of $\FF_{q}(v)=\FF_q(\tilde{v})$ with respect to the $\infty_{\tilde{v}}$-adic valuation ${\rm{ord}}_{\infty_{\tilde{v}}}={\rm{ord}}_{v}$.
\item[$\bullet$] $k_v\cong \FF_{q^{d}}(\!(v)\!) =
\FF_{q^{d}}(\!(\tilde{v}^{-1})\!)$ is a finite extension of the $\infty_{\tilde{v}}$-adic field $\FF_{q}(\!(\tilde{v}^{-1})\!)$.
\item[$\bullet$] A fixed algebraic closure $\overline{\FF_{q^{d}}(\!(v)\!)}$ of $\FF_{q^{d}}(\!(v)\!)$ is the same as a particularly chosen algebraic closure $\overline{\FF_{q}(\!(\tilde{v}^{-1})\!)}$ of $\FF_{q}(\!(\tilde{v}^{-1})\!)$.
\item[$\bullet$] $\CC_{v}$ is the $v$-adic completion of $\overline{\FF_{q}(\!(\tilde{v}^{-1})\!)}$, which is the same as the $\infty_{\tilde{v}}$-adic completion of $\overline{\FF_{q}(\!(\tilde{v}^{-1})\!)}$.
\end{itemize}
Since we deal with values in $\CC_{v}$, which is viewed as an $\infty_{\tilde{v}}$-adic field, one can apply Papanikolas' \lq\lq $\infty$-adic\rq\rq\ transcendence theory~\cite{P08}. The following comparison/replacement illustrates why we consider Frobenius modules over $\ok[\tilde{T_v}]$ in the following subsection when applying Papanikolas' theory fitting into our situation. 
\[
\begin{array}{rrcc|ccll}
{\hbox{Over }}k_{\infty}&  &  & &   & &  & {\hbox{Over }}k_{v}  \\ \hline\hline
{\hbox{uniformizer at }}\infty  & \rightarrow & 1/\theta& &  & \tilde{v}^{-1}&\leftarrow &{\hbox{uniformizer at }}{\tilde{\infty}_v} \\
t & \rightarrow &\theta & & & \tilde{v}& \leftarrow& \tilde{T}_v\\
{\hbox{polynomial ring}}& \rightarrow &\ok[t] & & &  \ok[\tilde{T}_v]&\leftarrow & {\hbox{polynomial ring}}\\
{\hbox{Tate algebra}}&\rightarrow  &\CC_{\infty}\{t\} & & & \CC_v\{\tilde{T}_v\}& \leftarrow& {\hbox{Tate algebra}} 
\end{array}
\]

The relations of the field extensions to be used is illustrated in the following diagram:

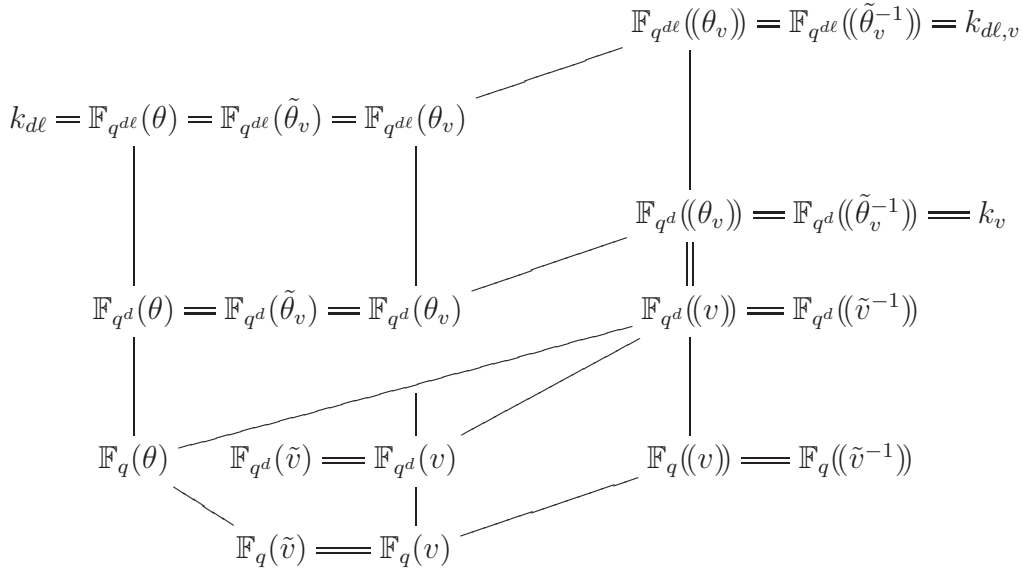
\[
\xymatrixrowsep{0.5cm}
\xymatrixcolsep{0.3cm}
\xymatrix{
&&& &&& & \FF_{q^{d\ell}}(\!(\theta_v)\!)\ar@{-}[dd] \ar@{=}[r] & \FF_{q^{d\ell}}(\!(\tilde{\theta}_v^{-1})\!)  \ar@{=}[r] & k_{d\ell,v} \\
k_{d\ell} \ar@{=}[r] &\FF_{q^{d\ell}}(\theta) \ar@{-}[dd]\ar@{=}[r] & \FF_{q^{d\ell}}(\tilde{\theta}_v) \ar@{=}[r]  & \FF_{q^{d\ell}}(\theta_v) \ar@{-}[dd] \ar@{-}[urrrr] &&& & & & \\
&&& &&& & \FF_{q^{d}}(\!(\theta_v)\!) \ar@{=}[d] \ar@{=}[r]& \FF_{q^{d}}(\!(\tilde{\theta}_v^{-1})\!) \ar@{=}[r] & k_v 
\\
& \FF_{q^d}(\theta)\ar@{-}[dd]\ar@{=}[r]& \FF_{q^d}(\tilde{\theta}_v) \ar@{=}[r]   &\FF_{q^d}(\theta_v) \ar@{-}[dd]|!{[ddll];[rrrr]}\hole
\ar@{-}[urrrr] &&& & \FF_{q^d}(\!(v)\!)\ar@{-}[dd] \ar@{=}[r]&  \FF_{q^d}(\!(\tilde{v}^{-1})\!) &   \\
&&& &&&& & \\
&\FF_q(\theta)\ar@{-}[dr]\ar@{-}[uurrrrrr] 
& \FF_{q^d}(\tilde{v}) \ar@{=}[r]  & \FF_{q^d}(v) \ar@{-}[d]\ar@{-}[uurrrr] &&& & \FF_q(\!(v)\!) \ar@{=}[r]&  \FF_{q}(\!(\tilde{v}^{-1})\!) &  \\
&&\FF_q(\tilde{v}) \ar@{=}[r]&\FF_q(v) \ar@{-}[urrrr]   &&&&& & 
}
\]

Recall that in Section~\ref{sec: no-2} we wrote 
\[
v = \theta^d + \epsilon_1 \theta^{d-1} + \cdots + \epsilon_{d-1}\theta + \epsilon_d, \quad \text{ where } \epsilon_1,...,\epsilon_d \in \FF_q.
\]
As 
\begin{equation}\label{eqn: v-1}
v = \prod_{i=0}^{d-1}(\theta-\varepsilon^{q^i}) = \prod_{i=0}^{d-1}\big(\theta_v - (\varepsilon^{q^i}-\varepsilon)\big),
\end{equation}
There exists $\tilde{\epsilon}_1,...,\tilde{\epsilon}_{d-1} \in \FF_{q^d}$
such that
\begin{equation}\label{eqn: v-2}
v = \theta_v^d + \tilde{\epsilon}_1 \theta_v^{d-1} + \cdots +\tilde{\epsilon}_{d-1}\theta_v\quad \in \FF_{q^d}[\theta_v] = \FF_{q^d}[\theta].
\end{equation}
Multiplying $\tilde{v}\tilde{\theta}_v^d$ on both sides of the equation \eqref{eqn: v-2}, one gets
\begin{equation}\label{eqn: v-tilde}
\tilde{\theta}_v^d -\Big((\tilde{\epsilon}_{d-1}\tilde{v})\cdot \tilde{\theta}_v^{d-1}  + \cdots + (\tilde{c}_1\tilde{v})\tilde{\theta}_v + \tilde{v}\Big) = 0.
\end{equation}

On the other hand, 
recall in Section~\ref{sec: no-2} that $\iota: \FF_q(t)\cong \FF_q(\theta)$ is the $\FF_q$-algebra isomorphism sending $t$ to $\theta$, and so $\iota(T_v) = v$.
Extending $\iota$ to an $\FF_{q^{d\ell}}$-algebra isomorphism $\iota: \FF_{q^{d\ell}}(t) \cong \FF_{q^{d\ell}}(\theta) = k_{d\ell}$, we
set
\[
\tilde{t}_v := \frac{1}{t-\varepsilon} \quad \text{ and } \quad 
\tilde{T}_v := \frac{1}{T_v}.
\]
Then $\iota(\tilde{t}_v) = \tilde{\theta}_v$ and $\iota(\tilde{T}_v) = \tilde{v}$.
Moreover, the following equality holds:

\begin{equation}\label{eqn: min-tv}
\tilde{t}_v^d-(\tilde{\epsilon}_{d-1}\tilde{T}_v) \tilde{t}_v^{d-1} - \cdots - (\tilde{\epsilon}_1\tilde{T}_v)\tilde{t}_v -\tilde{T}_v = 0.
\end{equation}
In other words, the minimal polynomial of $\tilde{t}_v$ over $\FF_{q^{d\ell}}(\tilde{T}_v)$ is 
\[
\tilde{g}_{\tilde{T}_v}(X) = X^d-(\tilde{\epsilon}_{d-1} \tilde{T}_v)X^{d-1} - \cdots - (\tilde{\epsilon}_{1}\tilde{T}_v) X - \tilde{T}_v \quad \in \FF_{q^d}[\tilde{T}_v][X],
\]
In particular, the field $\FF_{q^{d\ell}}(\tilde{t}_v)$  is a finite separable extension of $\FF_q(\tilde{T}_v)$ of degree $d^2 \ell$, and the integral closure of $\FF_q[\tilde{T}_v]$ in $\FF_{q^{d\ell}}(\tilde{t}_v)$ is $\FF_{q^{d\ell}}[\tilde{T}_v][\tilde{t}_v]$.
Indeed, as
\[
\tilde{T}_v = \frac{1}{T_v} = \prod_{i=0}^{d-1}\frac{1}{t-\varepsilon^{q^{i}}} = \tilde{t}_v^d \cdot \frac{1}{S},
\]
where $S = \prod_{i=1}^{d-1}(1-(\varepsilon^{q^i}-\varepsilon)\tilde{t}_v)$ is coprime to $\tilde{t}_v^d$ in $\FF_{q^d}[\tilde{t}_v]$.
Hence
\begin{equation}\label{E:PID}
\FF_{q^{d\ell}}[\tilde{T}_v][\tilde{t}_v] = \FF_{q^{d\ell}}[\tilde{t}_v][\frac{1}{S}],   
\end{equation}
which is a localization of the PID $\FF_{q^{d\ell}}[\tilde{t}_v]$. It follows that $\FF_{q^{d\ell}}[\tilde{T}_v][\tilde{t}_v] $ is a PID, which implies integrally closed property, and hence it equals the integral closure of $\FF_q[\tilde{T}_v]$ in $\FF_{q^{d\ell}}(\tilde{t}_v)$. Moreover, 
since $\tilde{g}_{\tilde{T}_v}(X)$ is Eisenstein with respect to the prime ideal $(\tilde{T}_v)$ of $\FF_{q^d}[\tilde{T}_v]$,
the prime ideal $(\tilde{T}_v)$ of $\FF_q[\tilde{T}_v]$ is totally ramified in $\FF_{q^{d\ell}}[\tilde{T}_v][\tilde{t}_v]$, and the unique prime ideal of $\FF_{q^{d\ell}}[\tilde{T}_v][\tilde{t}_v]$ lying above $(\tilde{T}_v)$ is the principal ideal generated by $\tilde{t}_v$.
Set
\[
\tilde{t}_{v,i}:= \frac{1}{t-\varepsilon^{q^{-i}}} = \frac{\tilde{t}_v}{1-(\varepsilon^{q^{-i}}-\varepsilon)\tilde{t}_v}, \quad 0\leq i < d\ell.
\]
One checks that
\begin{itemize}
\item[(1)] $\tilde{\theta}_{v,i}:=\iota(\tilde{t}_{v,i}) = (\theta-\varepsilon^{q^{-i}})^{-1}$.
\item[(2)] $\tilde{t}_{v,i} = \tilde{t}_{v,i'}$ if and only if $i\equiv i' \bmod d$;
\item[(3)] 
$\tilde{T}_v = \tilde{t}_{v,0} \cdots \tilde{t}_{v,d-1}$;

\item[(4)] The minimal polynomial of $\tilde{t}_{v,i}$ over $\FF_{q^{d\ell}}[\tilde{T}_v]$ is
\begin{equation}\label{eqn: min-tvi}
\tilde{g}_{\tilde{T}_v}^{(-i)}(X) = X^d-\Big(\tilde{\epsilon}_{d-1}^{q^{-i}}\tilde{T}_v X^{d-1} + \cdots + \tilde{\epsilon}_{1}^{q^{-i}}\tilde{T}_v\cdot X + \tilde{T}_v\Big) \quad \in \FF_{q^d}[\tilde{T}_v][X];
\end{equation}
\item[(5)] 
$(\tilde{t}_{v}) = (\tilde{t}_{v,0}) = (\tilde{t}_{v,1}) = \cdots = (\tilde{t}_{v,d\ell-1}) \subset \FF_{q^{d\ell}}[\tilde{T}_v][\tilde{t}_v]$.
\end{itemize}
The above setting (on the ``$t$-variable side'') can be illustrated by the following diagram:
\[
\xymatrixrowsep{0.5cm}
\xymatrixcolsep{0.2cm}
\xymatrix{
& & & & \FF_{q^{d\ell}}(t) \ar@{}[r]|-*[@]{=} & \FF_{q^{d\ell}}(\tilde{t}_v) \ar@{}[r]|-*[@]{\supset} & \FF_{q^{d\ell}}[\tilde{T}_v][\tilde{t}_v] \ar@{}[r]|-*[@]{\supset} & (\tilde{t}_{v,0}) \ar@{}[r]|-*[@]{=}& \cdots \ar@{}[r]|-*[@]{=} & (\tilde{t}_{v,d-1})  \\
& & & & \FF_{q^{d\ell}}(T_v)\ar@{-}[u] \ar@{}[r]|-*[@]{=} & \FF_{q^{d\ell}}(\tilde{T}_v)\ar@{-}[u] \ar@{}[r]|-*[@]{\supset} & \FF_{q^{d\ell}}[\tilde{T}_v]\ar@{-}[u] \ar@{}[r]|-*[@]{\supset} & (\tilde{T}_v)\ar@{-}[u] & \\
& & & & \FF_{q}(T_v)\ar@{-}[u] \ar@{}[r]|-*[@]{=} & \FF_{q}(\tilde{T}_v)\ar@{-}[u] \ar@{}[r]|-*[@]{\supset} & \FF_{q}[\tilde{T}_v]\ar@{-}[u]  \ar@{}[r]|-*[@]{\supset} & (\tilde{T}_v) \ar@{-}[u] &
}
\]
${}$

Now, let $\bar{k}(\tilde{T}_v)[\sigma]$ (resp. $\bar{k}(\tilde{t}_v)[\sigma]$) be the twisted polynomial ring over $\bar{k}(\tilde{T}_v)$ (resp.\ $\bar{k}(\tilde{t}_v)$) satisfying the following multiplication law
\[
\sigma \cdot f = f^{(-1)} \cdot \sigma, \quad \forall f \in \bar{k}(\tilde{T}_v) \ \text{ (resp.\ $\bar{k}(\tilde{t}_v)$)},
\]
where $f^{(n)}$ is the $n$-th Frobenius twist of $f$, i.e.\ raising the coefficients of powers of $\tilde{T}_v$) of $f$ to the $q^n$-power for each $n \in \ZZ$.

\begin{lemma}
$\bar{k}(\tilde{T}_v)[\sigma]$ is a subring of $\bar{k}(\tilde{t}_v)[\sigma]$ if and only if $\tilde{\epsilon}_1,...,\tilde{\epsilon}_{d-1} \in \FF_q$.
\end{lemma}

\begin{proof}
By definition one observes that 
$\bar{k}(\tilde{T}_v)[\sigma]$ is a subring of $\bar{k}(\tilde{t}_v)[\sigma]$ if and only if
$\tilde{T}_v \cdot \sigma = \sigma \cdot \tilde{T}_v$ in
$\bar{k}(\tilde{t}_v)[\sigma]$, 
which is equivalent to
\begin{align*}
&\ \Big(\tilde{t}_v^d-(\tilde{\epsilon}_{d-1}\tilde{T}_v) \tilde{t}_v^{d-1} - \cdots - (\tilde{\epsilon}_1\tilde{T}_v)\tilde{t}_v\Big) \cdot \sigma \\
= &\ \tilde{T}_v \cdot \sigma =
\sigma\cdot \tilde{T}_v \\
= &\  \sigma \cdot \Big(\tilde{t}_v^d-(\tilde{\epsilon}_{d-1}\tilde{T}_v) \tilde{t}_v^{d-1} - \cdots - (\tilde{\epsilon}_1\tilde{T}_v)\tilde{t}_v\Big) \\ 
= &\  \Big(\tilde{t}_v^d-(\tilde{\epsilon}_{d-1}^{q^{-1}}\tilde{T}_v) \tilde{t}_v^{d-1} - \cdots - (\tilde{\epsilon}_1^{q^{-1}}\tilde{T}_v)\tilde{t}_v\Big)\cdot \sigma.
\end{align*}
This equality holds if and only if $\tilde{\epsilon}_1,...,\tilde{\epsilon}_{d-1} \in \FF_q$.
\end{proof}

\begin{remark}
The above lemma indicates that $\bar{k}(\tilde{T}_v)[\sigma]$ might not be a subring of $\bar{k}(\tilde{t}_v)[\sigma]$ when $d>1$. However, both do contain a subring $\bar{k}(\tilde{T}_v)[\sigma^{d}]$.
\end{remark}

In the remaining part of this section, we will review the necessary properties of the Hartl--Juschka category $\cH\cJ$ of over $\bar{k}$ introduced in \cite[Definition 2.4.1]{HJ20} (with respect to the embedding $\iota: \FF_q(\tilde{T}_v) \cong \FF_q(\tilde{v}) \subset \bar{k} \subset \CC_v$),
as well as their corresponding ``Hodge--Pink cocharacters''.

\subsection{Hartl--Juschka Category}\label{sec: HJC} 
The Hartl--Juschka category over $\bar{k}$ is defined as follows:
an object of $\cH\cJ$ is a free left $\bar{k}[\tilde{T}_v]$-module $\cM$ of finite rank together with an invertible $\sigma$-action on $\itM :=  \bar{k}(\tilde{T}_v)\otimes_{\bar{k}[\tilde{T}_v]} \cM$ so that $\itM$ becomes a left $\bar{k}(\tilde{T}_v)[\sigma]$-module satisfying that there exists a sufficiently large integer $n$ so that
\[
    (\tilde{T}_v-\iota(\tilde{T}_v))^n \cM \subseteq \sigma \cM \subseteq (\tilde{T}_v-\iota(\tilde{T}_v))^{-n} \cM \quad (\subset \itM).
\]
We call such an object $\cM$ a \emph{Hartl--Juschka module over $\bar{k}$}.
In particular, we say that $\cM$ is {\it effective} if $\sigma \cM \subset \cM$.
A morphism $f : \cM \rightarrow \cM'$ in $\cH\cJ$ is a left $\bar{k}[\tilde{T}_v]$-module homomorphism so that the induced $\bar{k}(\tilde{T}_v)$-linear homomorphism $f: \itM = \bar{k}(\tilde{T}_v)\otimes_{\bar{k}[\tilde{T}_v]} \cM  \rightarrow \itM' = \bar{k}(\tilde{T}_v)\otimes_{\bar{k}[\tilde{T}_v]} \cM'$ satisfies
\[
f (\sigma m) = \sigma f(m), \quad \forall\, m \in \itM.
\]
Given two objects $\cM$ and $\cM'$ in $\cH\cJ$, the set $\Hom(\cM,\cM')$ of morphisms from $\cM$ to $\cM'$ has a natural $\FF_q[\tilde{T}_v]$-module structure, and elements in
\[
\Hom^{0}(\cM,\cM'):=  \FF_q(\tilde{T}_v)\otimes_{\FF_q[\tilde{T}_v]} \Hom(\cM,\cM')
\]
are called \textit{quasi-morphisms from $\cM$ to $\cM'$}.
The \textit{Hartl-Juschka Category $\cH\cJ^I$ up to isogeny} has the same objects in $\cH\cJ$ but the morphisms in $\cH\cJ^I$ are given by quasi-morphisms.\\

\begin{Subsubsec}{Uniformizability} 
Let $\CC_v\{\tilde{T}_v\}$ be the Tate algebra in the variable $\tilde{T}_v$ over $\CC_v$, i.e. 
\[
\CC_v\{\tilde{T}_v\}:=\Big\{ \sum_{i\geq 0} a_i \tilde{T}_v^i \in \CC_v [\![\tilde{T}_v]\!]\ \Big|\
\lim_{i\rightarrow \infty}|a_i|_v =0
\Big\}.
\]
We let $\CC_v\{\tilde{T}_v\}[\sigma]$ be the non-commutative algebra generated by $\sigma$ over $\CC_v\{\tilde{T}_v\}$ subject to the relation
\[
\sigma f=f^{(-1)} \sigma, \quad \forall\, f\in \CC_v\{\tilde{T}_v\},
\]
where for every $f = \sum_{i\geq 0} a_i \tilde{T}_v^i \in \CC_v\{\tilde{T}_v\}$, we define
\[
f^{(n)}:= \sum_{i\geq 0} a_i^{q^n} \tilde{T}_v^i , \quad \forall n \in \ZZ.
\]
Given an object $\cM$ of $\cH\cJ$, we set
$\MM :=  \CC_v\{\tilde{T}_v\} \otimes_{\bar{k}[\tilde{T}_v]}\cM$.
As $\tilde{T}_v-\iota(\tilde{T}_v) = \tilde{T}_v-\tilde{v}$ is invertible in $\CC_v\{\tilde{T}_v\}$,
one has that $\itM \subset \MM$, and
$\MM$ becomes a left $\CC_v\{\tilde{T}_v\}[\sigma]$-module with the $\sigma$-action given by
\[
\sigma\cdot (f\otimes m) := f^{(-1)} \otimes \sigma m, \quad \forall f \in \CC_v\{\tilde{T}_v\} \text{ and } m \in \cM.
\]
Put
\[
H_{\mathrm{Betti}}(\cM):=  \{ m \in \MM \mid \sigma m = m\}.
\]
We call $\cM$ \emph{uniformizable} (or \emph{rigid analytically trivial}) if the natural homomorphism
\begin{align*}
\CC_v\{\tilde{T}_v\} \otimes_{\FF_{q}[\tilde{T}_{v}]}  H_{\mathrm{Betti}} (\cM)
& \longrightarrow \MM
\end{align*}
is an isomorphism. 
In particular, let $\CC_v\{\tilde{T}_v\}^\dagger$ be the algebra of rigid-analytic functions on $\CC_v\setminus \{\tilde{v}^{q^n}\mid n \in \NN\}$. 
It is known that (see \cite[Proposition 2.4.27]{HJ20}) if $\cM$ is uniformizable, then $H_{\mathrm{Betti}}(\cM)$ is contained in $\MM^\dagger = \CC_v\{\tilde{T}_v\}^\dagger \otimes_{\bar{k}[t]}\cM$, and the natural homomorphism
\begin{align*}
\CC_v\{\tilde{T}_v\}^\dagger \otimes_{\FF_q[\tilde{T}_v]} H_{\text{Betti}}(\cM) & \longrightarrow \MM^\dagger
\end{align*}
is an isomorphism.    
\end{Subsubsec}

\begin{Subsubsec}{$\tilde{T}_v$-Motivic Galois group}
The full subcategory of $\cH\cJ^I$ consisting of all uniformizable objects is denoted by $\cH\cJ^I_U$.
Let $\mathbf{Vect}_{/\FF_q(\tilde{T}_v)}$ be the category of finite dimensional vector spaces over $\FF_q(\tilde{T}_v)$.
By \cite[Theorem 2.4.23]{HJ20}, $\cH\cJ^I_U$ is a neutral Tannakian category over $\FF_q(\tilde{T}_v)$ with fiber functor $\bomega: \cH\cJ^I_U \rightarrow \mathbf{Vect}_{/\FF_q(\tilde{T}_v)}$ defined by:
\[
\cM \longmapsto H_\cM :=  \FF_q(\tilde{T}_v)\otimes_{\FF_q[\tilde{T}_v]} H_{\mathrm{Betti}}(\cM).
\] 

In particular, 
for each object in $\cH\cJ^I_U$,
let $\cT_\cM$ be the strictly full Tannakian subcategory of $\cH\cJ^I_U$  generated by $\cM$.
By Tannakian duality, we have an affine algebraic group scheme $\Gamma_\cM$ defined over $\FF_q(\tilde{T}_v)$ so that $\cT_\cM$ is equivalent to the category $\mathbf{Rep}(\Gamma_\cM)$ of the $\FF_q(\tilde{T}_v)$-finite dimensional representations of the algebraic group $\Gamma_\cM$.
We call $\Gamma_\cM$ the \emph{$\tilde{T}_v$-motivic Galois group of $\cM$}. Note that we have the following property of $\dim \Gamma_{\cM}$ by Papanikolas' difference Galois theory~\cite{P08}.

Given an object $\cM\in \cH\cJ^I_U$, we let $r$ be the rank of $\cM$ over $\ok[\tilde{T}_{v}]$ and fix a $\ok[\tilde{T}_v]$-basis $\left\{m_1,\ldots,m_r \right\}$ of $\cM$. Suppose $\cM$ is effective.
Then the action of $\sigma$ on $\cM$ is represented by a matrix $\Phi\in \Mat_r(\ok[\tilde{T}_v])$ in the sense that \[ \sigma \begin{pmatrix}
    m_1\\
    \vdots\\
    m_r
\end{pmatrix} =\Phi \begin{pmatrix}
    m_1\\
    \vdots\\
    m_r
\end{pmatrix}.\]
Since $\cM$ is uniformizable, by~\cite{ABP04, P08} there exists a matrix $\Psi \in \GL_{r}(\CC_{v}\left\{\tilde{T}_v \right\}) $ satisfying the system of Frobenius difference equations
\begin{equation}\label{PhiPsi}
 \Psi^{(-1)}=\Phi \Psi .
\end{equation}

In~\cite{P08},  Papanikolas developed a difference Galois theory for the system of Frobenius difference equation~\eqref{PhiPsi} to explicitly construct an algebraic subgroup $\Gamma_{\Psi}\subset \GL_{r}$ over $\FF_{q}(t)$ for which 
\begin{itemize}
 \item [$\bullet$] $\Gamma_{\cM}$ is isomorphic to $\Gamma_{\Psi}$ as algebraic groups over $\FF_{q}(t)$ (see~\cite[Thm.~4.5.10]{P08}). 
    \item [$\bullet$]  $\dim \Gamma_{\cM}={\rm{tr.deg}}_{\ok(\tilde{T}_v)}  \ok(\tilde{T}_v)(\Psi) $ (see~\cite[Prop.~4.3.3 and Thm.~4.3.1]{P08}),
\end{itemize}
where the notation $\ok(\tilde{T}_v)(\Psi)$ is referred to be the field generated all entries of $\Psi$ over $\ok(\tilde{T}_v)$. Combining the two properties above, we have the following identity:

\begin{equation}\label{dim=trdegPsi}
   \dim \Gamma_{\cM}={\rm{tr.deg}}_{\ok(\tilde{T}_v)}  \ok(\tilde{T}_v)(\Psi) 
\end{equation}
\end{Subsubsec}

\begin{remark}
Note that we have a faithful representation $$\Gamma_{\cM} \hookrightarrow \GL(H_{\cM}),$$
which is functorial in $\cM$ (see~\cite[Thm.~4.5.3]{P08}). 
This implies that
\begin{equation}\label{eqn: CI}
\Gamma_{\cM}\hookrightarrow {\Cent}_{\GL(H_{\cM})}\left( \End_{\cT_{\cM}}(\cM)\right)\quad  \textup{(cf.~\cite[p.~1448, (6)]{CPY10})}
\end{equation}
as we have the natural embedding $\End_{\cT_{\cM}}(\cM)\hookrightarrow \End(H_{\cM})$.
\end{remark}

\begin{Subsubsec}{de~Rham pairing and the Grothendieck period conjecture}
Let $\cM$ be an object of $\cH\cJ^I_U$.
Recall that the {\it de Rham module of $\cM$ over $\bar{k}$}
is
\[
H_{\text{dR}}(\cM,\bar{k}):= \Hom_{\bar{k}}\left(\frac{\cM}{(\tilde{T}_v - \tilde{v})\cdot \cM},\bar{k}\right).
\]
In particular, every differential $\omega \in H_{\text{dR}}(\cM,\bar{k})$ can be extended to a $\CC_v$-linear functional on 
$\MM^\dagger =\CC_v\{\tilde{T}_v\}^\dagger \otimes_{\bar{k}[\tilde{T}_v]} \cM$ vanishing on $(\tilde{T}_v-\tilde{v})\cdot \MM^\dagger$.
As the uniformizability of $\cM$ ensures the inclusion $H_{\text{Betti}}(\cM) \subset \MM^\dagger$,
we have the following \emph{de~Rham pairing}
\begin{align}\label{eqn: dR-pairing}
H_{\mathrm{Betti}}(\cM) \times H_{\mathrm{dR}}(\cM,\bar{k}) & \longrightarrow \CC_v, & (\gamma,\omega) & \longmapsto \int_{\gamma} \omega := \omega(\gamma).
\end{align}

The primary tool in this paper of dealing with transcendence problems is the following function field version of Grothendieck's period conjecture, which was achieved by Papanikolas~\cite{P08} using the so-called ABP criterion~\cite{ABP04}. 
\end{Subsubsec}

\begin{theorem}\label{thm: GPC}
(See \cite{P08} and its refined version in \cite{C09}.)
Let $\cM$ be a uniformizable Hartl--Juschka module $\cM$ over $\bar{k}$.
Suppose that $\cM$ is effective.
Then we have that
\begin{equation*}
\dim \Gamma_{\textrm{$\cM$}} = \trdeg_{\bar{k}} \bar{k}\left(\int_\gamma \omega \ \bigg|\ \gamma \in H_{\mathrm{Betti}}(\cM),\ \omega \in H_{\mathrm{dR}}(\cM,\bar{k})\right).
\end{equation*}
\end{theorem}
\begin{proof}
Let $\Phi\in \Mat_{r}(\ok[\tilde{T}_v])$ be the matrix represented by the $\sigma$-action on $\cM$ with respect to a $\ok[\tilde{T}_v]$-basis $\left\{m_1,\ldots,m_r \right\}$. 
Since $\cM$ is uniformizable, there exists  $\Psi \in \GL_{r}(\CC_{v}\left\{\tilde{T}_v \right\}) $ for which $\Psi^{(-1)}=\Phi\Psi$. Put $\mathbf{m}=(m_{1},\ldots,m_{r})^{\rm{tr}}$, and note that the entries of $\Psi^{-1} \mathbf{m}$ comprises an $\FF_q[\tilde{T}_v]$-basis of $H_{\mathrm{Betti}}(\cM)$ (see~\cite[Lem.~4.4.12]{ABP04} for their arguments without the condition that $\cM$ is free of finite rank over $\ok[\sigma]$.)
As mentioned above that $H_{\mathrm{Betti}}(\cM)\subset \cM^{\dagger}$, elements of $\Psi^{-1}$ are in fact in $\CC_v\{\tilde{T}_v\}^\dagger$.    

Our hypothesis on $\det \Phi$ ensures that the conditions of the refined version of the ABP criterion is satisfied (see~[C09]), and so we have 
\[{\rm{tr.deg}}_{\ok(\tilde{T}_v)}  \ok(\tilde{T}_v)(\Psi) ={\rm{tr.deg}}_{\ok}  \ok(\Psi|_{\tilde{T}_v=\tilde{v}})={\rm{tr.deg}}_{\ok}  \ok(\Psi^{-1}|_{\tilde{T}_v=\tilde{v}}),\]
where $\ok(\Psi|_{\tilde{T}_v=\tilde{v}})$ (resp.~$\ok(\Psi^{-1}|_{\tilde{T}_v=\tilde{v}})$) denotes by the field over $\ok$ generated by all entries of $\Psi$ (resp.~of $\Psi^{-1}$) evaluated at $\tilde{T}_v=\tilde{v}$. It follows that by~\eqref{dim=trdegPsi}, we have

\begin{equation}\label{E:dim=trdegPsiv}
 \dim \Gamma_{\cM}={\rm{tr.deg}}_{\ok}  \ok(\Psi^{-1}|_{\tilde{T}_v=\tilde{v}}). 
\end{equation}

As each $\omega\in  H_{\mathrm{dR}}(\cM,\bar{k})$ can be viewed as a $\ok$-linear functional on $\cM$ factoring through $\cM/(t-\theta)\cM$, a natural basis of  $H_{\mathrm{dR}}(\cM,\bar{k})$ can be chosen as $\left\{ \omega_{j} \right\}_{j=1}^{r}\subset H_{\mathrm{dR}}(\cM,\bar{k})$ given by $\omega_{j}(m_i)=\delta_{ij}$, the standard Kronecker symbol. Note that $\left\{ \gamma_{i}:=\sum_{j=1}^{r} (\Psi^{-1} )_{ij}m_{j} \right\}_{i=1}^{r}$ is an $\FF_q[\tilde{T}_v]$-basis of $H_{\mathrm{Betti}}(\cM)$, by~\eqref{eqn: dR-pairing} we have that 
\[\int_{\gamma_{i}} \omega_{j}=((\Psi^{-1})_{ij})|_{\tilde{T}_v=\tilde{v}}=(\Psi^{-1}|_{\tilde{T}_v=\tilde{v}})_{ij} \]
for every $1\leq i, j\leq r$. So the desired result follows from~\eqref{E:dim=trdegPsiv}. 
\end{proof}

\subsection{Hodge--Pink cocharacters}\label{sec: H-P-co}

Given an object $\cM$ of $\cH\cJ^I_U$, we take a sufficiently large integer $n$ such that
\begin{equation}\label{eqn: Fil}
(\tilde{T}_v-\tilde{v})^n \cM \subseteq \sigma \cM \subseteq (\tilde{T}_v-\tilde{v})^{-n} \cM.
\end{equation}
Let $r$ be the rank of $\cM$ over $\bar{k}[\tilde{T}_v]$.
There exists a unique (non-ordered) $r$-tuple of integers $(w_1,\ldots, w_r)$ so that
\[
\frac{(\tilde{T}_v-\tilde{v})^{-n} \cM}{\sigma \cM} \cong \bigoplus_{i=1}^r \frac{\bar{k}[\tilde{T}_v]}{(\tilde{T}_v-\tilde{v})^{n-w_i} \bar{k}[\tilde{T}_v]}.
\]
We call $(w_1,\dots, w_r)$ the \emph{Hodge--Pink weights of~$\cM$} and mention that they are independent of the choices of $n$.
Let $\{m_1,\ldots, m_r\}$ be a basis of the free $\bar{k}[\tilde{T}_v]$-module $\cM$ so that
\[
\sigma\cM = \bigoplus_{i=1}^r(\tilde{T}_v-\tilde{v})^{-w_i}\bar{k}[\tilde{T}_v] \cdot m_i \quad \subset \bar{k}(\tilde{T}_v)\otimes_{\bar{k}[\tilde{T}_v]} \cM.
\]
Put
\[
H_{\cM,\CC_v}:= \CC_v \underset{\tilde{v}\mapsfrom \tilde{T}_v,\ \FF_{q}[\tilde{T}_v]}{\otimes} H_{\cM},
\]
and let $\lambda_1,\ldots, \lambda_r \in H_{\eM,\CC_v}$ be the images of $m_1,\dots, m_r$ through the following identification
\begin{align}\label{eqn: H-id} 
\CC_v \otimes_{\bar{k}}\frac{\cM}{(\tilde{T}_v-\tilde{v})\cM} 
\cong \frac{\MM^\dagger}{(\tilde{T}_v-\tilde{v})\MM^\dagger}
& \cong 
\frac{\CC_v\{\tilde{T}_v\}^\dagger }{(\tilde{T}_v-\tilde{v})\CC_v\{\tilde{T}_v\}^\dagger } \otimes_{\FF_q[\tilde{T}_v]} H_{\cM}
\notag \\
& 
\cong 
\CC_v \underset{\tilde{v}\mapsfrom \tilde{T}_v,\ \FF_{q}[\tilde{T}_v]}{\otimes}  H_{\cM} 
= H_{\cM,\CC_v}.
\end{align}

Then we have that:

\begin{proposition}
(See \cite[Definition~7.10]{Pink} and \cite[Section~7.3]{BCPW22}.)
Keep the notation as above.
Let $\Gamma_{\cM,\CC_v}:= \CC_v \times_{\FF_q(\tilde{T}_v)} \Gamma_{\cM}$, the base change of $\Gamma_{\cM}$ to $\CC_v$ via the evaluation map $\iota:=\left(\tilde{T}_v\mapsto \tilde{v} \right): \FF_q(\tilde{T}_v)\hookrightarrow \CC_v$.
There exists a cocharacter $\chi_{\cM}^{\vee}:\GG_{m/\CC_v}\rightarrow \Gamma_{\cM,\CC_v}$ satisfying that after composing with the embedding
$\Gamma_{\cM,\CC_v}\hookrightarrow \GL(H_{\cM,\CC_v})$,
one has that
\[
    \chi_{\cM}^{\vee}(x) (\lambda_i) = x^{w_i}\cdot \lambda_i, \quad \forall x \in \GG_{m/\CC_v}(\CC_v) \text{ and } 1\leq i \leq r.
\]
\end{proposition}

Since $\GL(H_{\cM,\CC_v})= \CC_{v}\underset{\tilde{v}\mapsfrom \tilde{T}_v, \FF_{q}(\tilde{T}_v)}{\times} \GL(H_{\cM})$, we have a natural action of $\Aut_{\FF_q(v)}(\CC_{v})$ on $\GL(H_{\cM,\CC_v})$, which leaves $\Gamma_{\cM,\CC_v}$ invariant. 
Thus any $\Aut_{\FF_q(v)}(\CC_v)$-conjugate of $\chi_\cM^\vee$ is still a cocharacter of $\Gamma_{\cM,\CC_v}$.

\begin{definition}\label{defn: HPC}
(cf.\ \cite[Definition 7.10]{Pink}.) Given an object $\cM$ in $\cH\cJ^I_U$, a \emph{Hodge--Pink cocharacter} of $\Gamma_{\cM,\CC_{v}}$ is a cocharacter which lies in the $\Gamma_{\cM,\CC_v} \rtimes \Aut_{\FF_q(v)}(\CC_v)$-conjugacy class of $\chi_\cM^\vee$.
\end{definition}

\section{Carlitz-type Hartl--Juschka modules and their periods}\label{sec: Car-HJ-module}

Throughout this section,
we fix an irreducible element $v\in A_{+}$ of degree $d$ as before and fix a positive integer $\ell$.
Having sufficient tools at hand, 
we are now able to interpret $\Omega_{d\ell,v,i}(v)$ for $0< i \leq d\ell$ (occuring in Theorem~\ref{thm: AI-Omega}) as the periods of the restriction of scalars of the ``Carlitz-type'' Hartl--Juschka modules,
and establish their algebraic independence in this section.

\subsection{Carlitz-type Hartl--Juschka modules and restriction of scalars}\label{sec: CHJ}

\begin{definition}
Let $\tilde{\cC}_{d\ell}$
be the left $\bar{k}[\tilde{T}_v][\sigma^{d\ell}]$-module  $\bar{k}[\tilde{T}_v][\tilde{t}_v]$ whose $\sigma^{d\ell}$-action is given by
\[
\sigma^{d\ell} \cdot m = (1-\frac{\tilde{t}_v}{\tilde{\theta}_v}) \cdot m^{(-d\ell)}, \quad \forall m \in \bar{k}[\tilde{T}_v][\tilde{t}_v].
\]
\end{definition}

\begin{remark}\label{rem: T-t-relation}
By \eqref{eqn: v-tilde} and \eqref{eqn: min-tv}, one has that
\[
\tilde{T}_v \equiv \tilde{v} \bmod (\tilde{t}_v-\tilde{\theta}_v)\bar{k}[\tilde{T}_v][\tilde{t}_v],
\]
for which we get
\[
(\tilde{T}_v-\tilde{v})\cdot \tilde{\cC}_{d\ell}\subset \sigma^{d\ell} \cdot \tilde{\cC}_{d\ell} 
\subset \tilde{\cC}_{d\ell}.
\]
By~\eqref{E:PID}, we have
$\bar{k}[\tilde{T}_v][\tilde{t}_v] = \sum_{n = 0}^\infty \frac{1}{S^{n}} \cdot \bar{k}[\tilde{t}_v]$ is not finitely generated as a $\bar{k}[\tilde{t}_v]$-module if $d>1$.
Since $\frac{1}{S^{n}} \cdot \bar{k}[\tilde{t}_v] \subset \tilde{\cC}_{d\ell}$ is invariant by $\sigma^{d\ell}$-action for every non-negative integer $n$, we see that $\tilde{\cC}_{d\ell}$ is not finitely generated as a $\bar{k}[\sigma^{d\ell}]$-module when $d>1$ (i.e.\  $\tilde{\cC}_{d\ell}$ is not an `Anderson dual $\tilde{T}_v$-motive' in the terminology of~\cite{ABP04}).
\end{remark}

\begin{definition}\label{def: RC}
The restriction of scalars of $\tilde{\cC}_{d\ell}$ is the following $\bar{k}[\tilde{T}_v][\sigma]$-module:
\[
\Rscr_{\sigma} (\tilde{\cC}_{d\ell}) := \bar{k}[\tilde{T}_v][\sigma] \otimes_{\bar{k}[\tilde{T}_v][\sigma^{d\ell}]} \tilde{\cC}_{d\ell}.
\]
\end{definition}

Decomposing
$\Rcal_{\sigma} (\tilde{\cC}_{d\ell})$ into
\[
\bigoplus_{i=0}^{d\ell -1} \sigma^i \otimes \tilde{\cC}_{d\ell},
\]
we get
\begin{align*}
\Rcal_{\sigma} (\tilde{\cC}_{d\ell}) \supset 
\sigma \cdot \Rcal_{\sigma} (\tilde{\cC}_{d\ell})
&= \left((1-\frac{\tilde{t}_v}{\tilde{\theta}_v})\cdot 1\otimes \tilde{\cC}_{d\ell}\right) \oplus \bigoplus_{i=1}^{d\ell-1} \sigma^i \otimes \tilde{\cC}_{d\ell} \\
& \supset (\tilde{T}_v-\tilde{v})\cdot \Rcal_{\sigma} (\tilde{\cC}_{d\ell}).
\end{align*}
Thus $\Rcal_{\sigma}(\tilde{\cC}_{d\ell})$ is an effective Hartl--Juschka module $\bar{k}$.

We see from the following proposition that $\Rscr_{\sigma} (\tilde{\cC}_{d\ell})$ has \lq complex multiplication\rq\  by $\FF_{q^{d\ell}}[\tilde{T}_v][\tilde{t}_v]$.
\begin{proposition}\label{Prop:CM endomorphisms}
Let notation be given as above. Then we have the following $\FF_q[\tilde{T}_v]$-algebra  embedding:
\[
\phi:=\big(\alpha\mapsto  \phi_{\alpha} \big): \FF_{q^{d\ell}}[\tilde{T}_v][\tilde{t}_v]\hookrightarrow \End_{\bar{k}[\tilde{T}_v][\sigma]}(\Rscr_{\sigma} (\tilde{\cC}_{d\ell})),
\]where 
\[\phi_{\alpha}:=\big( \sum_{i=0}^{d\ell-1}\sigma^{i}\otimes m_{i}\mapsto \sum_{i=0}^{d\ell-1}\sigma^{i}\otimes \alpha m_{i} \big). \]
\end{proposition}
\begin{proof}
From the definition of $\phi$, one sees that $\phi$ is a ring homomorphism. We  confirm that for any $\alpha\in \FF_{q^{d\ell}}[\tilde{T}_v][\tilde{t}_v]$, $\phi_{\alpha}$ commutes with $\sigma$ as follows:

\begin{align*}
&\phi_{\alpha}\circ \sigma \left(\sum_{i=0}^{d\ell-1} \sigma^{i}\otimes m_{i} \right) \\
&\hspace{1cm}  = \phi_{\alpha}\left(\sum_{i=0}^{d\ell-2}\sigma^{i+1}\otimes m_{i}+1\otimes \sigma^{d\ell} m_{d\ell-1} \right)\\
& \hspace{1cm} =\sum_{i=0}^{d\ell-2}\sigma^{i+1}\otimes \alpha m_{i}+1\otimes \alpha \sigma^{d\ell} m_{d\ell-1}  =\sum_{i=0}^{d\ell-2}\sigma^{i+1}\otimes \alpha m_{i}+1\otimes  \sigma^{d\ell} \alpha m_{d\ell-1}   \\
&\hspace{1cm} = \sigma \left( \sum_{i=0}^{d\ell-1}\sigma^{i}\otimes \alpha m_{i} \right)  
= \sigma\circ \phi_{\alpha}\left(\sum_{i=0}^{d\ell-1}\sigma^{i}\otimes m_{i} \right).
\end{align*}
\end{proof}

\begin{remark}\label{rem: CM-0}
Alternatively, we may regard $\tilde{\cC}_{d\ell}$ as a $(\bar{k}[\tilde{T}_v][\sigma^{d\ell}],\FF_{q^{d\ell}}[\tilde{T}_v][\tilde{t}_v])$-bimodule, and get the $\FF_q[\tilde{T}_v]$-algebra homomorphism $\FF_{q^{d\ell}}[\tilde{T}_v][\tilde{t}_v]\hookrightarrow \End_{\bar{k}[\tilde{T}_v][\sigma]}(\Rscr_{\sigma} (\tilde{\cC}_{d\ell}))$ explicitly described in Proposition~\ref{Prop:CM endomorphisms}.
\end{remark}

\begin{Subsubsec}{Explicit descriptions of $\Rcal_{\sigma} (\tilde{\cC}_{d\ell})$}
Let $\bar{k}[\tilde{T}_v][X]$ be the polynomial ring with variable $X$ over $\bar{k}[\tilde{T}_v]$, and let $\tilde{g}_{\tilde{T}_v}(X)$ be the minimal polynomial of $\tilde{t}_v$ over $\bar{k}[\tilde{T}_v]$, i.e., by~\eqref{eqn: min-tv}, we have
\[
\tilde{g}_{\tilde{T}_v}(X) = X^d - (\tilde{\epsilon}_{d-1} \tilde{T}_v) X^{d-1} - \cdots - (\tilde{\epsilon}_1 \tilde{T}_v) X - \tilde{T}_v \in \FF_{q^d}[\tilde{T}_v][X] \subset \bar{k}[\tilde{T}_v][X].
\]
Then $\bar{k}[\tilde{T}_v][\tilde{t}_v]$ can be identified with
\[
\frac{\bar{k}[\tilde{T}_v][X]}{\tilde{g}_{\tilde{T}_v}(X)\cdot \bar{k}[\tilde{T}_v][X]}.
\]

We extend the Frobenius twists by fixing $X$ to $\bar{k}(\tilde{T}_v)[X]$ as follows: for each $f(X) = \sum_{m} f_m X^m$ and $n \in \ZZ$,
\[
f^{(n)}(X) := \sum_{m} f_m^{(n)} X^m.
\]
We point out that when $f(X) \in \bar{k}[X]$, i.e.\ $f_m \in \bar{k}$ for all $m$, one has that $f^{(n)}(X) = \sum_m f_m^{q^n}X^m$.
In the decomposition
\[
\Rscr_{\sigma} (\tilde{\cC}_{d\ell}) = \bigoplus_{i=0}^{d\ell -1} \sigma^i \otimes \tilde{\cC}_{d\ell},
\]
for each $i$ with $0\leq i <d\ell$ we have the $\bar{k}[\tilde{T}_v]$-linear isomorphism
\[
\frac{\bar{k}[\tilde{T}_v][X]}{\tilde{g}_{\tilde{T}_v}^{(-i)}(X)\cdot \bar{k}[\tilde{T}_v][X]} \stackrel{\sim}{\longrightarrow} \sigma^i \otimes \tilde{\cC}_{d\ell} = \sigma^i \otimes \left(\frac{\bar{k}[\tilde{T}_v][X]}{\tilde{g}_{\tilde{T}_v}(X)\cdot \bar{k}[\tilde{T}_v][X]}\right)
\]
defined by
\[
\left(\sum_{m=0}^n f_m X^m\right) \bmod \tilde{g}_{\tilde{T}_v}^{(-i)}(X) \longmapsto \sigma^i \otimes \left(\sum_{m=0}^n f_m^{(i)} X^m \bmod \tilde{g}_{\tilde{T}_v}(X) \right).
\]
Therefore:
\end{Subsubsec}

\begin{proposition}\label{prop: ID-C}
We may identify $\Rscr_{\sigma} (\tilde{\cC}_{d\ell})$ with the product space
\[
\frac{\bar{k}[\tilde{T}_v][X]}{\tilde{g}_{\tilde{T}_v}^{(0)}(X)\cdot \bar{k}[\tilde{T}_v][X]} \times \frac{\bar{k}[\tilde{T}_v][X]}{\tilde{g}_{\tilde{T}_v}^{(-1)}(X)\cdot \bar{k}[\tilde{T}_v][X]}\times \cdots \times \frac{\bar{k}[\tilde{T}_v][X]}{\tilde{g}_{\tilde{T}_v}^{(1-d\ell)}(X)\cdot \bar{k}[\tilde{T}_v][X]},
\]
equipped with the $\sigma$-action given by:
\begin{align*}
& \sigma \cdot (h_0 \bmod \tilde{g}_{\tilde{T}_v}^{(0)}, h_1 \bmod \tilde{g}_{\tilde{T}_v}^{(-1)},\cdots,h_{d\ell-1} \bmod \tilde{g}_{\tilde{T}_v}^{(1-d\ell)}) \\
& \hspace{1.5cm} := \left((1-\frac{X}{\tilde{\theta}_v})\cdot h_{d\ell-1}^{(-1)} \bmod \tilde{g}_{\tilde{T}_v}^{(0)}, h_0^{(-1)} \bmod \tilde{g}_{\tilde{T}_v}^{(-1)},\cdots,h_{d\ell-2}^{(-1)} \bmod \tilde{g}_{\tilde{T}_v}^{(1-d)}\right).
\end{align*}
\end{proposition}

\begin{remark}
As $\tilde{t}_{v,i}$ is a root of $\tilde{g}_{\tilde{T}_v}^{(-i)}(X)$ for each integer $i$ with $0\leq i < d\ell$, one has that
\[
\bar{k}[\tilde{T}_v][\tilde{t}_{v,i}] \cong \frac{\bar{k}[\tilde{T}_v][X]}{\tilde{g}_{\tilde{T}_v}^{(-i)}(X)\cdot \bar{k}[\tilde{T}_v][X]}.
\]
The above proposition allows us to identify $\Rscr_{\sigma} (\tilde{\cC}_{d\ell})$ with 
\begin{equation}\label{eqn: C-id}
\bar{k}[\tilde{T}_v][\tilde{t}_{v,0}] \times \bar{k}[\tilde{T}_v][\tilde{t}_{v,1}] \times \cdots \times \bar{k}[\tilde{T}_v][\tilde{t}_{v,d\ell-1}],
\end{equation}
and the $\sigma$-action is given by: for every polynomials $h_0(X),...,h_{d\ell-1}(X) \in \bar{k}[\tilde{T}_v][X]$,
\begin{align*}
&\sigma \cdot \Big(h_0(\tilde{t}_{v,0}), h_1(\tilde{t}_{v,1}),\cdots, h_{d\ell-1}(\tilde{t}_{v,d\ell-1})\Big) \\
&\hspace{1.5cm} = \Big((1-\frac{\tilde{t}_{v,0}}{\tilde{\theta}_v}) \cdot h_{d\ell-1}^{(-1)}(\tilde{t}_{v,0}), h_0^{(-1)}(\tilde{t}_{v,1}),\cdots, h_{d\ell-2}^{(-1)}(\tilde{t}_{v,d\ell-1})\Big).
\end{align*}
\end{remark}

\subsection{Uniformizability}\label{sec: CHJ-U}

Recall that $\CC_v\{\tilde{T}_v\}$ is the Tate algebra with variable $\tilde{T}_v$ over $\CC_v$.
Denote by $\CC_v\{\tilde{t}_v\}$ the Tate algebra with variable $\tilde{t}_v$ over $\CC_v$, i.e.
\[
\CC_v\{\tilde{t}_v\}:=\Big\{ \sum_{i\geq 0} a_i \tilde{t}_v^i \in \CC_v [\![\tilde{t}_v]\!]\ \Big|\
\lim_{i\rightarrow \infty}|a_i|_v =0\Big\}.
\]
Put
\[
\CC_v\{\tilde{T}_v\}[\tilde{t}_v] = \CC_v\{\tilde{T}_v\} \otimes_{\bar{k}[\tilde{T}_v]}\bar{k}[\tilde{T}_v][\tilde{t}_v].
\]

\begin{lemma}\label{lem: Tate-alg-inclusion}
We have the following inclusion
\[
\CC_v\{\tilde{t}_v\} \subset \CC_v\{\tilde{T}_v\}[\tilde{t}_v].
\]
\end{lemma}

\begin{proof}
For each non-negative integer $i$, by \eqref{eqn: min-tv} there exist $f_{i,j} \in \FF_{q^d}[\tilde{T}_v]$, $0\leq j <d$, with $\deg_{\tilde{T}_i} f_{i,j}\leq i$ such that 
\[
\tilde{t}_v^i = f_{i,0}\tilde{t}_v^{d-1}+\cdots + f_{i,d-1}.
\]
Write $f_{i,j}(\tilde{T}_v) = \sum_{s=0}^i u_{i,j,s} \tilde{T}_v^s$ where $u_{i,j,s} \in \FF_{q^d}$.
Then for each $f = \sum_{i\geq 0} a_i \tilde{t}_v^i \in \CC_v\{\tilde{t}_v\}$, the property $\lim_{i\rightarrow \infty} |a_i|_v = 0$ implies that for $0\leq j <d$ and $s\geq 0$,
\[
b_{j,s}:= \lim_{n\rightarrow \infty} \sum_{i=s}^n a_i u_{i,j,s} \quad \text{ converges absolutely in $\CC_v$ with } \quad  \lim_{s\rightarrow \infty} b_{j,s} = 0.
\]
Put
\begin{align*}
f_j &:= \sum_{s\geq 0} b_{j,s} \tilde{T}_v^s = \sum_{s=0}^\infty \left(\sum_{i=s}^\infty a_i u_{i,j,s}\right) \tilde{T}_v^s \\
&\, = \sum_{i=0}^{\infty}a_i \left(\sum_{s=0}^i u_{i,j,s} \tilde{T}_v^s\right) = \sum_{i=0}^\infty a_i f_{i,j}(\tilde{T}_v)\quad 
\in \CC_v\{\tilde{T}_v\}.
\end{align*}
We then get 
\[
f = \sum_{i=0}^\infty a_i\tilde{t}_v^i = \sum_{j=0}^{d-1} \sum_{i=0}^\infty a_i f_{i,j} \tilde{t}_v^j = \sum_{j=0}^{d-1} f_j \cdot \tilde{t}_v^j  \quad \in \CC_v\{\tilde{T}_v\}[\tilde{t}_v]
\]
as desired.
\end{proof}

\begin{remark}
The inclusion $\bar{k}(\tilde{T}_v) \subset \bar{k}(\tilde{t}_v)$ induces a finite morphism $\phi$ from the projective $\tilde{t}_v$-line $\PP^1_{\tilde{t}_v}$ onto the projective $\tilde{T}_v$-line $\PP^1_{\tilde{T}_v}$ (over $\bar{k}$).
In particular, $\bar{k}[\tilde{t}_v]$ and $\bar{k}[\tilde{T}_v][\tilde{t}_v]$ are the affine coordinate rings of $\AA^1_{\tilde{t}_v}$ and $U:=\AA^1_{\tilde{t}_v} \setminus \phi^{-1}(\tilde{v})$, respectively.
Since $\CC_v\{\tilde{t}_v\}$ and $\CC_v\{\tilde{T}_v\}[\tilde{t}_v]$  correspond to the the affinoid domains $\AA^1_{\tilde{t}_v,\leq 1}:=\{x \in \AA^1_{\tilde{t}_v}(\CC_v)\mid |x|_v \leq 1\}$ and $U_{\leq 1}:=\{x \in U(\CC_v)\mid |x|_v \leq 1\}$, respectively, the natural inclusion $U_{\leq 1} \subset \AA^1_{\tilde{t}_v,\leq 1}$ alternatively implies $\CC_v\{\tilde{t}_v\} \subset \CC_v\{\tilde{T}_v\}[\tilde{t}_v]$.
\end{remark}

To show the uniformizability of $\Rscr_{\sigma}(\tilde{\cC}_{d\ell})$, we first note that the identification of $\Rscr_{\sigma} (\tilde{\cC}_{d\ell})$
extends to
\begin{equation}\label{E:Decom Tate-alg Rsigma}
\CC_v\{\tilde{T}_v\} \otimes_{\bar{k}[\tilde{T}_v]} \Rscr_{\sigma} (\tilde{\cC}_{d\ell})=
\CC_v\{\tilde{T}_v\}[\tilde{t}_{v,0}] \times \CC_v\{\tilde{T}_v\}[\tilde{t}_{v,1}] \times \cdots \times \CC_v\{\tilde{T}_v\}[\tilde{t}_{v,d\ell-1}].
\end{equation}
In particular, Lemma~\ref{lem: Tate-alg-inclusion} implies that
\begin{equation}\label{eqn: Tate-alg-inclusion}
\CC_v\{\tilde{t}_{v,0}\} \times \CC_v\{\tilde{t}_{v,1}\} \times \cdots \times \CC_v\{\tilde{t}_{v,d\ell-1}\} \subset \CC_v\{\tilde{T}_{v}\} \otimes_{\bar{k}[\tilde{T}_v]}\Rscr_{\sigma}(\tilde{\cC}_{d\ell}).
\end{equation}
In particular, 
for every $\big(h_0(\tilde{t}_{v,0}),...,h_{d\ell-1}(\tilde{t}_{v,d\ell-1})\big) \in \CC_v\{\tilde{t}_{v,0}\} \times \CC_v\{\tilde{t}_{v,1}\} \times \cdots \times \CC_v\{\tilde{t}_{v,d\ell-1}\}$,
one has that
\begin{align}\label{E:-sigma-action on Decomp}
&\sigma \cdot \Big(h_0(\tilde{t}_{v,0}), h_1(\tilde{t}_{v,1}),\cdots, h_{d\ell-1}(\tilde{t}_{v,d\ell-1})\Big) \notag  \\
&\hspace{1.5cm} = \Big((1-\frac{\tilde{t}_{v,0}}{\tilde{\theta}_v}) \cdot h_{d\ell-1}^{(-1)}(\tilde{t}_{v,0}), h_0^{(-1)}(\tilde{t}_{v,1}),\cdots, h_{d\ell-2}^{(-1)}(\tilde{t}_{v,d\ell-1})\Big),
\end{align}
and for $0\leq i<d\ell$,
\[
h_{i}^{(-1)}(\tilde{t}_{v,\overline{i+1}}) = \sum_{m}a_m^{q^{-1}}\tilde{t}_{v,\overline{i+1}}^m\quad \text{ when } \quad 
h_{i}(\tilde{t}_{v,i}) = \sum_{m}a_m\tilde{t}_{v,i}^m \in \CC_v\{\tilde{t}_{v,i}\}.
\]
Here $\overline{i+1} =i+1$ if $i+1<d\ell$ or $0$ if $i+1=d\ell$.
\\

Note that since $|\tilde{\theta}_{v}|_{v}>1$, 
given an integer $i$ with $0\leq i<d\ell$, we can set $\psi(\tilde{t}_v)$
to be the following infinite product 
\[
\psi(\tilde{t}_v):= \prod_{i=1}^\infty\left(1-\frac{\tilde{t}_v}{\tilde{\theta}_v^{q^{d\ell i}}}\right)^{-1} \quad \in \CC_v\{\tilde{t}_{v}\}^\times \subset \CC_v\{\tilde{T}_{v}\}[\tilde{t}_v]^\times.
\]
Take $\xi \in \FF_{q^{d\ell}}$ so that $\FF_{q^{d\ell}} = \FF_q(\xi)$.
For $0\leq i <d\ell$ and $0\leq j < d$,
Put
\begin{align*}
\psi_{i,j} &:= \bigg(\xi^i \tilde{t}_{v,0}^j \psi(\tilde{t}_{v,0}), (\xi^{q^{-1}})^i \tilde{t}_{v,1}^j \psi^{(-1)}(\tilde{t}_{v,1}), \cdots, (\xi^{q^{1-d\ell}})^i \tilde{t}_{v,d\ell-1}^j \psi^{(1-d\ell)}(\tilde{t}_{v,d\ell-1})\bigg) \\
& \hspace{0.7cm} \in \CC_v\{\tilde{t}_{v,0}\} \times \CC_v\{\tilde{t}_{v,1}\} \times \cdots \times \CC_v\{\tilde{t}_{v,d\ell-1}\}.
\end{align*}
Then 
\begin{lemma}\label{lem: HB(C)} Via the decomposition~\eqref{E:Decom Tate-alg Rsigma}, we have that
\[
H_{\text{\rm Betti}}\big(\Rscr_{\sigma} (\tilde{\cC}_{d\ell})\big) = \bigoplus_{\subfrac{0\leq i<d\ell}{0\leq j<d}} \FF_q[\tilde{T}_v] \psi_{i,j}.
\]
Consequently, $\Rscr_{\sigma} (\tilde{\cC}_{d\ell})$ is uniformizable.
\end{lemma}

\begin{proof}
We first note that \[ \psi^{(-d\ell)}(\tilde{t}_v)=(1-\frac{\tilde{t}_{v,0}}{\tilde{\theta}_v})^{-1} \cdot \psi(\tilde{t}_v).\]
From~\eqref{E:-sigma-action on Decomp} we see  that $\psi_{i,j} \in H_{\text{\rm Betti}}\big(\Rscr_{\sigma} (\tilde{\cC}_{d\ell})\big)$ for $0\leq i<d\ell$ and $0\leq j<d$.
Since 
$\rank_{\bar{k}[\tilde{T}_v]}\Rscr_{\sigma} (\tilde{\cC}_{d\ell}) = d^2 \ell$,
it suffices to show that the set $\left\{\psi_{i,j}|\hbox{ } 0\leq i<d\ell, 0\leq j <d\right\}$ generates $\CC_v\{\tilde{T}_{v}\} \otimes_{\bar{k}[\tilde{T}_v]}\Rscr_{\sigma}(\tilde{\cC}_{d\ell})$ over $\CC_v\{\tilde{T}_{v}\}$.
First, from the isomorphism 
\[
\begin{tabular}{ccl}
$\bar{k}\otimes_{\FF_q} \FF_{q^{d\ell}}$ & $\stackrel{\sim}{\longrightarrow}$ & $\bar{k}\times \bar{k}\times  \cdots \times \bar{k}$ \\
$\alpha \otimes 1$ & $\longmapsto$ & $(\alpha\ ,\alpha\ ,\cdots ,\alpha)$ \\
$1\otimes \xi$ & $\longmapsto$ & $(\xi, \xi^{q^{-1}},\cdots,\xi^{q^{d\ell-1}})$,
\end{tabular}
\]
one has that for each integer $i$ with $0\leq i < d\ell$, there exist $\alpha_{i,0},...,\alpha_{i,d\ell-1} \in \bar{k}$ such that
\[
\sum_{m=0}^{d\ell-1}\alpha_{i,m}\xi^{mq^{-i'}} = \begin{cases}
1, & \text{ if $i'=i$,}\\
0, & \text{ otherwise.}
\end{cases}
\]
Put
\[
\psi_{i,j}' :=
\sum_{m=0}^{d\ell-1}\alpha_{i,m} \psi_{m,j}
= \Big(0,...,\tilde{t}_{v,i}^j\psi^{(-i)}(\tilde{t}_{v,i}),0,...,0\Big).
\]
We then obtain that
\begin{align*}
&\  \sum_{\subfrac{0\leq i<d\ell}{0\leq j<d}} \CC_v\{\tilde{T}_{v}\}  \cdot \psi_{i,j}
\ = \ \sum_{\subfrac{0\leq i<d\ell}{0\leq j<d}} \CC_v\{\tilde{T}_{v}\}  \cdot \psi'_{i,j} \\
=& \  \Big(\CC_v\{\tilde{T}_{v}\} [\tilde{t}_{v,0}]\cdot \psi(\tilde{t}_{v,0})\Big) \times \Big(\CC_v\{\tilde{T}_{v}\} [\tilde{t}_{v,1}]\cdot \psi^{(-1)}(\tilde{t}_{v,1}) \Big)\times \cdots \times \Big(\CC_v\{\tilde{T}_{v}\} [\tilde{t}_{v,d\ell-1}]\cdot \psi^{(1-d\ell)}(\tilde{t}_{v,d\ell-1})\Big) \\
=&\
\CC_v\{\tilde{T}_{v}\} [\tilde{t}_{v,0}] \times \CC_v\{\tilde{T}_{v}\} [\tilde{t}_{v,1}] \times \cdots \times 
\CC_v\{\tilde{T}_{v}\} [\tilde{t}_{v,d\ell-1}],
\end{align*}
where the last equality comes from the observation:
\[
\psi^{(-i)}(\tilde{t}_{v,i}) \in \CC_v\{\tilde{t}_{v,i}\} ^\times \subset \CC_v\{\tilde{T}_{v}\} [\tilde{t}_{v,i}]^\times, \quad 0\leq i < d\ell.
\]
This completes the proof.
\end{proof}

\begin{remark}\label{rem: CM}
Since $\Rscr_{\sigma} (\tilde{\cC}_{d\ell})$
has Complex Multiplication by $\FF_{q^{d\ell}}[\tilde{T}_v][\tilde{t}_v]$ (see Remark~\ref{rem: CM-0}),
we may regard
$H_{\text{Betti}}\big(\Rscr_{\sigma} (\tilde{\cC}_{d\ell})\big)$ as an $\FF_{q^{d\ell}}[\tilde{T}_v][\tilde{t}_v]$-module.
Then the proof of the above lemma indicates that
\[
H_{\text{Betti}}\big(\Rscr_{\sigma} (\tilde{\cC}_{d\ell})\big) = \FF_{q^{d\ell}}[\tilde{T}_v][\tilde{t}_v] \cdot \psi_{0,0}.
\]
\end{remark}

\subsection{Periods of \texorpdfstring{$\Rscr_{\sigma} (\tilde{\cC}_{d\ell})$}{Rsigma(Cdl)}}
\label{sec: CHJ-P}

Recall that the {\it de Rham module of $\Rscr_{\sigma} (\tilde{\cC}_{d\ell})$ over $\bar{k}$}
is 
\[
H_{\text{dR}}(\Rscr_{\sigma} (\tilde{\cC}_{d\ell}),\bar{k}):= \Hom_{\bar{k}}\left(\frac{\Rscr_{\sigma} (\tilde{\cC}_{d\ell})}{(\tilde{T}_v - \tilde{v})\cdot \Rscr_{\sigma} (\tilde{\cC}_{d\ell})},\bar{k}\right),
\]
and the de~Rham pairing is given by (see \eqref{eqn: dR-pairing})
\[
H_{\text{Betti}}\big(\Rscr_{\sigma} (\text{$\tilde{\cC}$}_{d\ell})\big) \times H_{\text{dR}}\big(\Rscr_{\sigma} (\text{$\tilde{\cC}$}_{d\ell}),\bar{k}\big) \longrightarrow \CC_v, \quad (\gamma,\omega)\mapsto \int_{\gamma} \omega := \omega(\gamma).
\]
In particular, for each integer $i$ with $0\leq i <d\ell$,
put
\begin{equation}\label{eqn: g-v-i}
\tilde{g}_{v,i}(X) 
:= \tilde{g}_{\tilde{T}_v}^{(-i)}(X)\big|_{\tilde{T}_v = \tilde{v}} = X^d-(\tilde{\epsilon}_{d-1}^{q^{(-i)}}\tilde{v})X^{d-1} - \cdots - (\tilde{\epsilon}_1^{q^{(-i)}}\tilde{v})X - \tilde{v}
\in \bar{k}[X].
\end{equation}
The identification in Proposition~\ref{prop: ID-C} leads to the following isomorphism:
\begin{equation}\label{eqn: C-decomp}
\frac{\Rscr_{\sigma} (\tilde{\cC}_{d\ell})}{(\tilde{T}_v - \tilde{v})\cdot \Rscr_{\sigma} (\tilde{\cC}_{d\ell})} 
\cong  \prod_{i=0}^{d\ell-1}\frac{\bar{k}[X]}{\tilde{g}_{v,i}(X) \cdot \bar{k}[X]}.
\end{equation}
Moreover, for each integer $s$ with $0\leq s<d\ell$, from~\eqref{eqn: min-tv} the element 
\[
\tilde{\theta}_{v,s}:=  (\theta-\varepsilon^{q^{-s}})^{-1} = \iota(\tilde{t}_{v,s}) \quad \in \bar{k}
\]
is a root of
$\tilde{g}_{v,s}(X)$, which gives a differential $\tilde{\omega}_{s} \in H_{\text{dR}}(\Rscr_{\sigma} (\tilde{\cC}_{d\ell}),\bar{k})$ induced from
\begin{equation}\label{eqn: diff}
\begin{tabular}{ccccl}
$\hspace{-0.8cm}\displaystyle\prod_{i=0}^{d\ell-1}\frac{\bar{k}[X]}{\tilde{g}_{v,i}(X) \cdot \bar{k}[X]}$& $\longrightarrow$ &$\displaystyle\frac{\bar{k}[X]}{\tilde{g}_{v,s}(X) \cdot \bar{k}[X]}$& $\longrightarrow$ &
$\bar{k}$\\ $\Big(\bar{h}_i(X) \bmod \tilde{g}_{v,i}\Big)_{0\leq i<d\ell}$ & $\longmapsto$ &
$\bar{h}_{s}(X) \bmod \tilde{g}_{v,s}$ & $\longmapsto$ &\hspace{-0.1cm}$\bar{h}_{s}(\tilde{\theta}_{v,s})$. 
\end{tabular}
\end{equation}
Therefore we obtain the following period interpretation:

\begin{proposition}\label{prop: Omega-period} Recall the notation~\eqref{E:Omega ell,v,i}.
Given an integer $s$ with $0\leq s<d\ell$, one has
\[
\int_{\psi_{0,0}}\tilde{\omega}_{s} = \frac{1}{\Omega_{d\ell,v,d\ell-s}(v)}.
\]
\end{proposition}

\begin{proof}
From the definition of $\omega_{s}$ and the de Rham pairing, one gets that
\begin{align*}
\int_{\psi_{0,0}}\tilde{\omega}_{s} &=
\psi^{(-s)}(\tilde{t}_{v,s})\Big|_{\tilde{t}_{v,s}=\tilde{\theta}_{v,s}} \\
&= \prod_{i=1}^\infty \left(1-\frac{\tilde{\theta}_{v,s}}{\tilde{\theta}_v^{q^{d\ell i -s}}}\right)^{-1} = \prod_{i=1}^\infty \left(1-\frac{(\theta-\varepsilon^{q^{-s}})^{-1}}{((\theta-\varepsilon)^{q^{d\ell i - s}})^{-1}}\right)^{-1} \\
&= 
\prod_{i=1}^\infty \left(1-
\frac{(\theta-\varepsilon)^{q^{d\ell i - s}}}{\theta-\varepsilon^{q^{-s}}}
\right)^{-1} = \frac{1}{\Omega_{d\ell,v,d\ell-s}(v)}.
\end{align*}
\end{proof}

\begin{remark}\label{rem: CM-period}
As $\Rscr_{\sigma} (\tilde{\cC}_{d\ell})$ has CM by $\FF_{q^{d\ell}}[\tilde{T}_v][\tilde{t}_v]$ (see Remark~\ref{rem: CM}), we have that
\begin{align*}
&\bar{k}\bigg(\int_{\gamma}\tilde{\omega}\ \bigg|\ \gamma \in H_{\text{Betti}}\big(\Rscr_{\sigma} (\tilde{\cC}_{d\ell})\big),\ \tilde{\omega} \in H_{\text{dR}}\big(\Rscr_{\sigma} (\tilde{\cC}_{d\ell}),\bar{k}\big)\bigg)\\
&\hspace{2cm} =\  \bar{k}\bigg(\int_{\psi_{0,0}}\tilde{\omega} \ \bigg|\ \tilde{\omega} \in H_{\text{dR}}\big(\Rscr_{\sigma} (\tilde{\cC}_{d\ell}),\bar{k}\big)\bigg).
\end{align*}
The Gronthendieck period conjecture in Theorem~\ref{thm: GPC} leads to
\begin{equation}\label{eqn: trdeg-2}
\trdeg_{\bar{k}}\bar{k}\bigg(\int_{\psi_{0,0}}\omega \ \bigg|\ \omega \in H_{\text{dR}}\big(\Rscr_{\sigma} (\tilde{\cC}_{d\ell}),\bar{k}\big)\bigg) =
\dim \Gamma_{\Rscr_{\sigma} (\tilde{\cC}_{d\ell})},
\end{equation}
where $\Gamma_{\Rscr_{\sigma} (\tilde{\cC}_{d\ell})}$ is the $\tilde{T}_v$-motivic Galois group of $\Rscr_{\sigma} (\tilde{\cC}_{d\ell})$.
Next, we will determine the dimension of $\Gamma_{\Rscr_{\sigma} (\tilde{\cC}_{d\ell})}$, and complete the proof of Theorem~\ref{thm: AI-Omega}.
\end{remark}

\subsection{\texorpdfstring{$\tilde{T}_v$}{Tv}-Motivic Galois group of \texorpdfstring{$\Rscr_{\sigma} (\tilde{\cC}_{d\ell})$}{Rsigma(Cdl)} and the Hodge--Pink cocharacters}
\label{sec: CHJ-MG-HPC}

We let $\eK_{d\ell}:= \FF_{q^{d\ell}}(\tilde{t}_v)$ and $\eF:= \FF_q(\tilde{T}_v)$.
Since $\Rscr_{\sigma} (\tilde{\cC}_{d\ell})$ has complex multiplication by $\FF_{q^{d\ell}}[\tilde{T}_v][\tilde{t}_v]$ by Proposition~\ref{Prop:CM endomorphisms}
and
\[
\rank_{\FF_q[\tilde{T}_v]}\FF_{q^{d\ell}}[\tilde{T}_v][\tilde{t}_v] = [\eK_{d\ell}:\eF] = d^2 \ell = \dim_{\FF_q(\tilde{T})}H_{\Rscr_{\sigma} (\tilde{\cC}_{d\ell})}, 
\]
by \eqref{eqn: CI}
we obtain that
\[
\Gamma_{\Rscr_{\sigma} (\tilde{\cC}_{d\ell})} \hookrightarrow \Rscr_{K_{d\ell}/F} (\GG_{m/\eK_{d\ell}}) \subset \GL(H_{\Rscr_{\sigma} (\tilde{\cC}_{d\ell})}),
\]
where
$\Rscr_{K_{d\ell}/F} (\GG_{m/\eK_{d\ell}})$ is the \emph{restriction of scalars of $\GG_{m/\eK_{d\ell}}$ to $\eF$}.
Our goal is to show 
$\Gamma_{\Rscr_{\sigma} (\tilde{\cC}_{d\ell})}=\Rscr_{K_{d\ell}/F} (\GG_{m/\eK_{d\ell}})$
by analysing the images of its Hodge--Pink cocharacters.

\begin{Subsubsec}{Hodge--Pink cocharacters of $\Gamma_{\Rscr_{\sigma} (\tilde{\cC}_{d\ell})}$}\label{sec: HP-co}
For $0\leq i <d\ell$,
let $\tilde{\theta}_{v,i,0}(=\tilde{\theta}_{v,i}),..., \tilde{\theta}_{v,i,d-1}$ be the distinct roots of 
$\tilde{g}_{v,i}(X) \in \bar{k}[X]$ defined in \eqref{eqn: g-v-i}.
Comparing the minimal polynomials of $\tilde{t}_{v,i}$ in \eqref{eqn: min-tvi} and $\tilde{g}_{v,i}(X)$, we get that
\[
\tilde{T}_v- \tilde{v} \in (\tilde{t}_{v,i}-\tilde{\theta}_{v,i,j}) \bar{k}[\tilde{T}_v][\tilde{t}_{v,i}], 
\]
whence 
\begin{equation}\label{eqn: decomp-T-v}
(\tilde{T}_v-\tilde{v})\cdot \bar{k}[\tilde{T}_v][\tilde{t}_{v,i}] = \prod_{j=0}^{d-1}(\tilde{t}_{v,i}-\tilde{\theta}_{v,i,j})\cdot \bar{k}[\tilde{T}_v][\tilde{t}_{v,i}].
\end{equation}
On the other hand,
under the identification~\eqref{eqn: C-id},
we get
\begin{align}\label{eqn: decomp-sigma}
\sigma \cdot \Rscr_{\sigma} (\tilde{\cC}_{d\ell})
&=
\Big((\tilde{t}_{v,0}-\tilde{\theta}_{v,0,0}) \cdot \bar{k}[\tilde{T}_v][\tilde{t}_{v,0}] \Big)\times \bar{k}[\tilde{T}_v][\tilde{t}_{v,1}] \times \cdots \times \bar{k}[\tilde{T}_v][\tilde{t}_{v,d\ell-1}]  \\
&\supset
(\tilde{T}_v-\tilde{v})\cdot \Rscr_{\sigma} (\tilde{\cC}_{d\ell}), \notag
\end{align}
and the Hodge--Pink weight of $\Rscr_{\sigma} (\tilde{\cC}_{d\ell})$ is $(w_{i,j})_{0\leq i<d\ell,0\leq j<d}$ with
\[
w_{i,j} =
\begin{cases}
-1, & \text{ if $i=j=0$,}\\
0, & \text{ otherwise.}
\end{cases}
\]
Moreover precisely,
we may take a $\bar{k}[\tilde{T}_v]$-base $\{m_{i,j}\mid 0\leq i<d\ell,\ 0\leq j<d\}$
satisfying that
under the identification
\[
\frac{\Rscr_{\sigma} (\tilde{\cC}_{d\ell})}{(\tilde{T}_v-\tilde{v})\Rscr_{\sigma} (\tilde{\cC}_{d\ell})} \cong 
\prod_{i=0}^{d\ell-1}\frac{\bar{k}[\tilde{T}_v][\tilde{t}_{v,i}]}{(\tilde{T}_v-\tilde{v})\bar{k}[\tilde{T}_v][\tilde{t}_{v,i}]} \cong 
\prod_{\subfrac{0\leq i<d\ell}{0\leq j<d}}\frac{\bar{k}[\tilde{T}_v][\tilde{t}_{v,i}]}{(\tilde{t}_{v,i}-\tilde{\theta}_{v,i,j})\bar{k}[\tilde{T}_v][\tilde{t}_{v,i}]} \cong 
\prod_{\subfrac{0\leq i<d\ell}{0\leq j<d}} \bar{k},
\]
the image of $m_{i,j}$ is
\[
(\delta_{(i,j),(i',j')})_{\subfrac{0\leq i<d\ell}{0\leq j<d}} \in \prod_{\subfrac{0\leq i<d\ell}{0\leq j<d}} \bar{k},
\quad
\text{ where }
\quad 
\delta_{(i,j),(i',j')}=
\begin{cases}
    1, & \text{if $i=i'$ and $j=j'$,}\\
    0, & \text{ otherwise.}
\end{cases}
\]
From the isomorphism \eqref{eqn: H-id}, we may decompose
\begin{align*}
H_{\Rscr_{\sigma} (\tilde{\cC}_{d\ell}),\CC_v}  = \CC_v \underset{\tilde{v}\mapsfrom \tilde{T}_v,\eF}{\otimes} H_{\Rscr_{\sigma} (\tilde{\cC}_{d\ell})} & \stackrel{\sim}{\longrightarrow} \bigoplus_{\subfrac{0\leq i<d\ell}{0\leq j<d}}H_{i,j}. 
\end{align*}
Notice that
for each $0\leq i<d\ell$ and $0\leq j<d$,
let $\iota_{i,j}: \eK_{d\ell} = \FF_{q^{d\ell}}(\tilde{t}_v)\hookrightarrow \CC_v$ be the embedding satisfying that $\iota_{i,j}(\varepsilon) = \varepsilon^{q^{-j}}$ and $\iota_{i,j}(\tilde{t}_v) = \tilde{\theta}_{v,i,j}$.
Regarding $H_{\sigma}(\tilde{\cC}_{d\ell})$ as a one dimensional vector space over $\eK_{d\ell}$ (by Remark~\ref{rem: CM}),  we have that 
\[
\alpha \cdot z = \iota_{i,j}(\alpha) \cdot z, \quad \forall z \in H_{i,j} \quad \text{and} \quad \alpha \in \FF_{q^{d\ell}}(\tilde{t}_v).
\]
Therefore under the identification
\[
\CC_v \underset{\tilde{v} \mapsfrom \tilde{T}_v,\eF}{\times}\Rcal_{\eK_{d\ell}/\eF}(\GG_{m/\eK_{d\ell}}) \cong \prod_{\subfrac{0\leq i<d\ell}{0\leq j<d}}\CC_v \underset{\iota_{i,j}, \eK_{d\ell}}{\times} \GG_{m/\eK_{d\ell}} = \prod_{\subfrac{0\leq i<d\ell}{0\leq j<d}}\GL(H_{i,j}),
\]
the cocharacter 
\[
\chi_{\Rscr_{\sigma} (\tilde{\cC}_{d\ell})}^{\vee}: \GG_{m/\CC_v}\rightarrow \Gamma_{\Rscr_{\sigma} (\tilde{\cC}_{d\ell}),\CC_v}\subset \CC_v \underset{\tilde{v} \mapsfrom \tilde{T}_v,\eF}{\times} \Rcal_{\eK_{d\ell}/\eF}(\GG_{m/\eK_{d\ell}})
\]
is then realized as follows:
\[
\chi_{\Rscr_{\sigma} (\tilde{\cC}_{d\ell})}^{\vee}(x) = 
\begin{pmatrix}
    x^{-1} & & &\\
    & 1 & &\\
    & & \ddots & \\
    & & & 1
\end{pmatrix} 
\quad \in \prod_{\subfrac{0\leq i<d\ell}{0\leq j<d}}\CC_v \underset{\iota_{i,j}, \eK_{d\ell}}{\times} \GG_{m/\eK_{d\ell}}.
\]
Since $\Aut_{\FF_q(\tilde{v})}(\CC_v)$ acts transitively on the coordinates of $\prod_{\subfrac{0\leq i<d\ell}{0\leq j<d}}\CC_v \underset{\iota_{i,j}, \eK_{d\ell}}{\times} \GG_{m/\eK_{d\ell}}$,
the Hodge--Pink cocharacters of $ \Gamma_{\Rscr_{\sigma} (\tilde{\cC}_{d\ell})}$ generate the whole group $\CC_{v} \underset{\tilde{v} \mapsfrom \tilde{T}_v,\eF}{\times} \Rcal_{\eK_{d\ell}/\eF}(\GG_{m/\eK_{d\ell}})$.
Consequently, we finally arrive at:
\end{Subsubsec}

\begin{theorem}\label{thm: dim-G}
\[
\Gamma_{\Rscr_{\sigma} (\tilde{\cC}_{d\ell})}
= \Rcal_{\eK_{d\ell}/\eF}(\GG_{m/\eK_{d\ell}}),
\]
whence $\dim \Gamma_{\Rscr_{\sigma} (\tilde{\cC}_{d\ell})} = [\eK_{d\ell}:\eF]=d^2 \ell$.
\end{theorem}

Combining with \eqref{eqn: trdeg-2}, we finally arrive at:

\begin{corollary}\label{cor: AI-Omega}
\begin{align*}
\trdeg_{\bar{k}}\bar{k}\bigg(\int_{\psi_{0,0}} \tilde{\omega} \ \bigg|\ \tilde{\omega} \in H_{\text{\rm dR}}\big(\Rscr_{\sigma} (\tilde{\cC}_{d\ell}),\bar{k}\big)\bigg) & = d^2\ell.
\end{align*}
\end{corollary}

\begin{Subsubsec}{Proof of Theorem~\ref{thm: AI-Omega}}\label{sec: pf-AL-Omega}
By Proposition~\ref{prop: Omega-period} we have that
\[
\Omega_{d\ell,v,i}(v) = \left(\int_{\psi_{0,0}}\tilde{\omega}_{d\ell-i} \right)^{-1}, \quad 0< i \leq d\ell.
\]
On the other hand,
for $0\leq i<d\ell$ and $0\leq j<d$,
define $\tilde{\omega}_{i,j}\in H_{dR}\big(\Rscr_{\sigma} (\tilde{\cC}_{d\ell}),\bar{k}\big)$ as follows:
\begin{equation*}
\begin{tabular}{ccccl}
$\hspace{-0.8cm}\displaystyle\prod_{i'=0}^{d\ell-1}\frac{\bar{k}[X]}{\bar{g}_{v,i'}(X) \cdot \bar{k}[X]}$& $\longrightarrow$ &$\displaystyle\frac{\bar{k}[X]}{\bar{g}_{v,i}(X) \cdot \bar{k}[X]}$& $\longrightarrow$ &
$\bar{k}$\\ $\Big(\bar{h}_{i'}(X) \bmod \bar{g}_{v,i'}\Big)_{0\leq i<d\ell}$ & $\longmapsto$ &
$\bar{h}_{i}(X) \bmod \bar{g}_{v,i}$ & $\longmapsto$ &\hspace{-0.1cm}$\bar{h}_{i}(\tilde{\theta}_{v,i,j})$,
\end{tabular}
\end{equation*}
where $\tilde{\theta}_{v,i,j}$ are given in the beginning of Section~\ref{sec: HP-co}.
Then \[
\big\{\tilde{\omega}_{i,j} \ \big|\  0\leq i<d\ell,\ 0\leq j<d \big\} \quad \text{ is a basis of $H_{dR}\big(\Rscr_{\sigma} (\tilde{\cC}_{d\ell}),\bar{k}\big)$},
\]
and
$\tilde{\omega}_{i} = \tilde{\omega}_{i,0}$ for $0\leq i<d\ell$ by comparing with
\eqref{eqn: diff}.
Since
\begin{align*}
\dim_{\bar{k}}H_{dR}\big(\Rscr_{\sigma} (\tilde{\cC}_{d\ell}),\bar{k}\big)=d^2\ell  = \trdeg_{\bar{k}}\bar{k}\bigg(\int_{\psi_{0,0}}\tilde{\omega}_{i,j} \ \bigg|\ 0\leq i<d\ell,\ 0\leq j<d \bigg),
\end{align*}
we obtain that
\[
\int_{\psi_{0,0}}\tilde{\omega}_{i,j}, \quad 0\leq i<d\ell, \quad 0\leq j<d, \quad \text{ are algebraically independent over $\bar{k}$.}
\]

Finally, as for $0< i \leq \ell$, one has that
\begin{equation}\label{eqn: omega-dltol}
\prod_{j=0}^{d-1}\Omega_{d\ell,v,j\ell +i}(v)
= \prod_{j=0}^{d-1}\prod_{n=0}^{\infty}(1-\frac{\theta_v^{q^{d\ell n +j\ell+i}}}{\theta-\varepsilon^{q^{j\ell+i}}})
= \prod_{n=0}^{\infty}(1-\frac{\theta_v^{q^{\ell n +i}}}{\theta-\varepsilon^{q^{\ell n +i}}}) =
\Omega_{\ell,v,i}(v).
\end{equation}
Therefore the result follows.

\hfill $\Box$    
\end{Subsubsec}

\bibliographystyle{alpha}

\end{document}